\documentclass{amsart}
\usepackage{amssymb,amsmath,mathtools}

\numberwithin{equation}{section}

\newcommand\cZt{\mathcal{X}}
\newcommand{\Cut}{\operatorname{Cut}}   
\newcommand{\ribbongraph}{G}
\newcommand{\poles}{p}
\newcommand{\Graph}{\Gamma}
\newcommand{\Separating}{\Delta}
\newcommand{\cyl}{\mathit{cyl}}
\newcommand{\separatingQ}{\chi_{\Graph}}
\newcommand{\carea}{\operatorname{c_{\mathit{area}}}}
\newcommand{\zeroes}{\ell} 
\newcommand{\dprinc}{d} 
\newcommand{\dVolMV}{d\!\Vol}

\newcommand{\elow}{\varepsilon_{\mathit{below}}(g)}
\newcommand{\eup}{\varepsilon_{\mathit{above}}(g)}

\newlength{\halfbls}

\newcommand{\QuadStrat}{\operatorname{\mathcal{Q}}}

\newcommand{\cD}{\mathcal{D}}
\newcommand{\cI}{\mathcal{I}}
\newcommand{\cG}{\mathcal{G}}
\newcommand{\cH}{\mathcal{H}}
\newcommand{\cL}{\mathcal{L}}
\newcommand{\cM}{\mathcal{M}}
\newcommand{\cML}{\mathcal{ML}}
\newcommand{\cN}{\mathcal{N}}
\newcommand{\cO}{\mathcal{P}}

\newcommand{\cQ}{\mathcal{Q}}
\newcommand{\cR}{\mathcal{R}}
\newcommand{\cS}{\mathcal{S}}

\newcommand{\cY}{\mathcal{Y}}
\newcommand{\cZ}{\mathcal{Z}}
\newcommand{\cST}{\mathcal{ST}}

\newcommand{\E}{{\mathbb E}}
\newcommand{\N}{{\mathbb N}}
\newcommand{\Z}{{\mathbb Z}}
\newcommand{\Proj}{{\mathbb P}}
\newcommand{\Q}{{\mathbb Q}}
\newcommand{\R}{{\mathbb R}}
\newcommand{\C}{{\mathbb C}}

\newcommand{\Area}{\operatorname{Area}}
\newcommand{\Vol}{\operatorname{Vol}}
\newcommand{\Aut}{\operatorname{Aut}}
\newcommand{\CP}{{\mathbb C}\!\operatorname{P}^1}
\newcommand{\card}{\operatorname{card}}
\newcommand{\Mod}{\operatorname{Mod}}
\newcommand{\Stab}{\operatorname{Stab}}
\newcommand{\Sym}{\operatorname{Sym}}

\newcommand{\SL}{\operatorname{SL}(2,{\mathbb R})}
\newcommand{\SLZ}{\operatorname{SL}(2,{\mathbb Z})}

\newcommand{\GL}{\operatorname{GL}(2,\mathbb{R})}

\renewcommand{\epsilon}{\varepsilon}
\renewcommand{\Re}{\operatorname{Re}}
\renewcommand{\Im}{\operatorname{Im}}

\newtheorem{Theorem}{Theorem}[section]

\newtheorem{Proposition}[Theorem]{Proposition}
\newtheorem{Lemma}[Theorem]{Lemma}
\newtheorem{Corollary}[Theorem]{Corollary}
\newtheorem{Conjecture}[Theorem]{Conjecture}

\newtheorem*{NNTheorem}{Theorem}

\newtheorem*{NNLemma}{Lemma}

\theoremstyle{definition}
\newtheorem{Definition}[Theorem]{Definition}

\theoremstyle{remark}
\newtheorem{Example}[Theorem]{Example}
\newtheorem{Remark}[Theorem]{Remark}

\title[Masur--Veech volumes, simple closed geodesics, intersection numbers]
{Masur-Veech volumes, frequencies of simple closed geodesics and intersection numbers of moduli spaces of curves
}

\author[V.~Delecroix]{Vincent Delecroix}
\address{
LaBRI,
Domaine universitaire,
351 cours de la Lib\'eration, 33405 Talence, FRANCE
}
\email{20100.delecroix@gmail.com}

\author[\'E.~Goujard]{\'Elise Goujard}
\thanks{Research of E.~Goujard was partially supported by
PEPS}
\address{
Institut de Math\'ematiques de Bordeaux,
Universit\'e de Bordeaux,
351 cours de la Lib\'eration, 33405 Talence, FRANCE
}
\email{elise.goujard@gmail.com}

\author[P.~G.~Zograf]{Peter~Zograf}
\thanks{The results of Section~\ref{s:2:correlators}
were obtained at Saint Petersburg State University under support of the RSF grant 19-71-30002.
The work of P.~Zograf was partially supported by the Government of Russian Federation megagrant 14.W03.31.0030.}

\address{
St.~Petersburg Department, Steklov Math. Institute, Fontanka 27,
St. Petersburg 191023, and Chebyshev Laboratory,
St. Petersburg State University, 14th
Line V.O. 29B, St.Petersburg 199178 Russia}
\email{zograf@pdmi.ras.ru}

\author[A.~Zorich]{Anton Zorich}
\thanks{This material is based upon work supported by the NSF Grant DMS-1440140 while
part of the authors were in residence at the MSRI during
the Fall 2019 semester. It is also supported by the grant
ANR-19-CE40-0021.
}
\address{
Center for Advanced Studies, Skoltech;
Institut de Math\'ematiques de Jussieu --
Paris Rive Gauche,
Case 7012,
8 Place Aur\'elie Nemours,
75205 PARIS Cedex 13, France}
\email{anton.zorich@gmail.com}

\begin{document}

\begin{abstract}
We express the Masur--Veech volume and
the area Siegel--Veech constant of
the moduli space $\cQ_{g,n}$ of genus $g$ meromorphic quadratic
differentials with at most $n$ simple poles and no other poles as
polynomials in the intersection numbers
$\int_{\overline{\cM}_{g'\!,n'}} \psi_1^{d_1}\dots\psi_{n'}^{d_{n'}}$
with explicit rational coefficients, where $g'<g$ and $n'<2g+n$.
The formulae obtained in this article are derived from
lattice point counts involving the Kontsevich volume
polynomials $N_{g'\!,n'}(b_1, \ldots, b_{n'})$ that also appear
in Mirzakhani's recursion for the Weil--Petersson volumes
of the moduli spaces $\cM_{g'\!,n'}(b_1,\dots,b_{n'})$ of
bordered hyperbolic surfaces with geodesic boundaries
of lengths $b_1,\dots,b_{n'}$.

A similar formula for the Masur--Veech volume (though
without explicit evaluation) was obtained earlier by
M.~Mirzakhani through a completely different approach. We
prove a further result: the density of the mapping class
group orbit $\Mod_{g,n}\cdot\gamma$ of any simple closed
multicurve $\gamma$ inside the ambient set $\cML_{g,n}(\Z)$
of integral measured laminations computed by Mirzakhani,
coincides with the density of square-tiled surfaces having
horizontal cylinder decomposition associated to $\gamma$
among all square-tiled surfaces in $\cQ_{g,n}$.

We study the resulting densities (or, equivalently,
volume contributions) in more detail in the special case
when $n=0$. In particular, we compute explicitly the
asymptotic frequencies of separating and non-separating
simple closed geodesics on a closed hyperbolic surface of
genus $g$ for all small genera $g$ and we show that in
large genera the separating closed geodesics are
$\sqrt{\frac{2}{3\pi g}}\cdot\frac{1}{4^g}$
times less frequent.
\end{abstract}

\maketitle

\tableofcontents

\section{Introduction and statements of main theorems}
\label{s:intro}

\subsection{Masur--Veech volume of the moduli space of
quadratic differentials}
\label{ss:MV:volume}
Let $g,n$ be non-negative integers with $2g+n > 2$.
Consider the moduli space $\cM_{g,n}$ of complex curves of
genus $g$ with $n$ distinct labeled marked points. The
total space $\cQ_{g,n}$ of the cotangent bundle over
$\cM_{g,n}$ can be identified with the moduli space of
pairs $(C,q)$, where $C\in\cM_{g,n}$ is a smooth complex
curve with $n$ (labelled) marked points and $q$ is a meromorphic quadratic differential on
$C$ with at most simple poles at the marked points and no other
poles. In the case $n=0$ the quadratic differential $q$ is
holomorphic. Thus,
as any total space of the cotangent bundle,
the \textit{moduli space of quadratic
differentials} $\cQ_{g,n}$ is endowed with the canonical
(real) symplectic structure. The induced volume element
on $\cQ_{g,n}$ is called the \textit{Masur--Veech volume
element}. In Section~\ref{ss:background:strata}
we provide an alternative, more common definition
of the Masur--Veech volume element and explain why the
two definitions are equivalent
up to a global normalization constant.

A non-zero differential $q$ in $\cQ_{g,n}$ defines a flat
metric $|q|$ on the complex curve $C$. The resulting metric
has conical singularities at zeroes and simple poles of
$q$. The total area of $(C,q)$
$$
\Area(C,q)=\int_C |q|
$$
is positive and finite. For any real $a > 0$, consider the following
subset in $\cQ_{g,n}$:
\begin{equation}
\label{eq:ball:of:radius:a:in:Q:g:n}
\cQ^{\Area\le a}_{g,n} := \left\{(C,q)\in\cQ_{g,n}\,|\, \Area(C,q) \le a\right\}\,.
\end{equation}
Since $\Area(C,q)$ is a norm in each fiber of the bundle
$\cQ_{g,n} \to \cM_{g,n}$, the set $\cQ^{\Area \le
a}_{g,n}$ is a ball bundle over $\cM_{g,n}$. In particular,
it is non-compact. However, by independent results of
H.~Masur~\cite{Masur:82} and
W.~Veech~\cite{Veech:Gauss:measures}, the total mass of
$\cQ^{\Area\le a}_{g,n}$ with respect to the Masur--Veech
volume element is finite. One of the objectives of this
article is to provide a formula for this total mass. Our
construction relies on square-tiled surface counting that
we introduce next.

\subsection{Square-tiled surfaces, simple
closed multicurves and stable graphs}
\label{ss:Square:tiled:surfaces:and:associated:multicurves}

We have already mentioned that a non-zero meromorphic
quadratic differential $q$ on a complex curve $C$ defines a
flat metric with conical singularities. One can construct a
discrete collection of quadratic differentials of this kind
by assembling together identical flat squares in the
following way. Take a finite set of copies of the oriented
$1/2 \times 1/2$-square for which two opposite sides are
chosen to be horizontal and the remaining two sides are
declared to be vertical. Identify pairs of sides of the
squares by isometries in such way that horizontal sides are
glued to horizontal sides and vertical sides to vertical.
We get a topological surface $S$ without boundary. We
consider only those surfaces obtained in this way which are
connected and oriented. The form $dz^2$ on each square is
compatible with the gluing and endows $S$ with a complex
structure and with a non-zero quadratic differential $q$
with at most simple poles. The total area $\Area(S,q)$ is
$\frac{1}{4}$ times the number of squares. We call such a
surface a \textit{square-tiled surface}.

\begin{figure}[htb]
   %
   %
\includegraphics{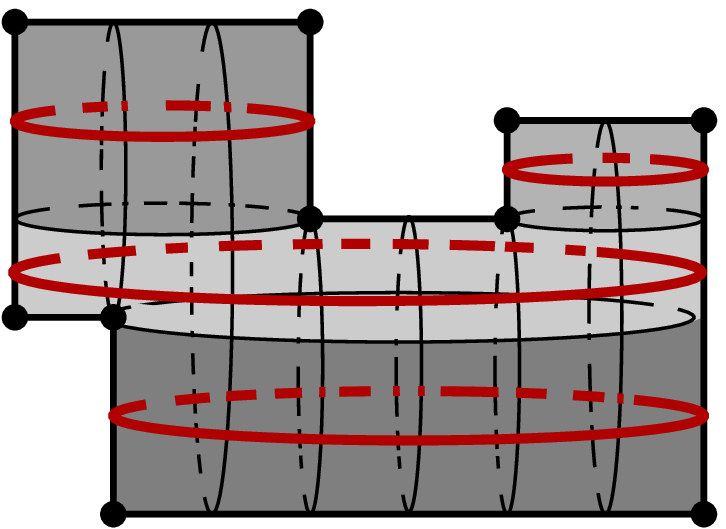}
\includegraphics{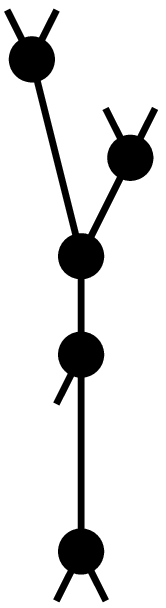}

\begin{picture}(0,0)(175,10)
\put(2,-13){$2\gamma_1$}
\put(113,-23){$\gamma_2$}
\put(1.5,-46){$\phantom{2}\gamma_3$}
\put(23,-77){$2\gamma_4$}
\end{picture}

\includegraphics{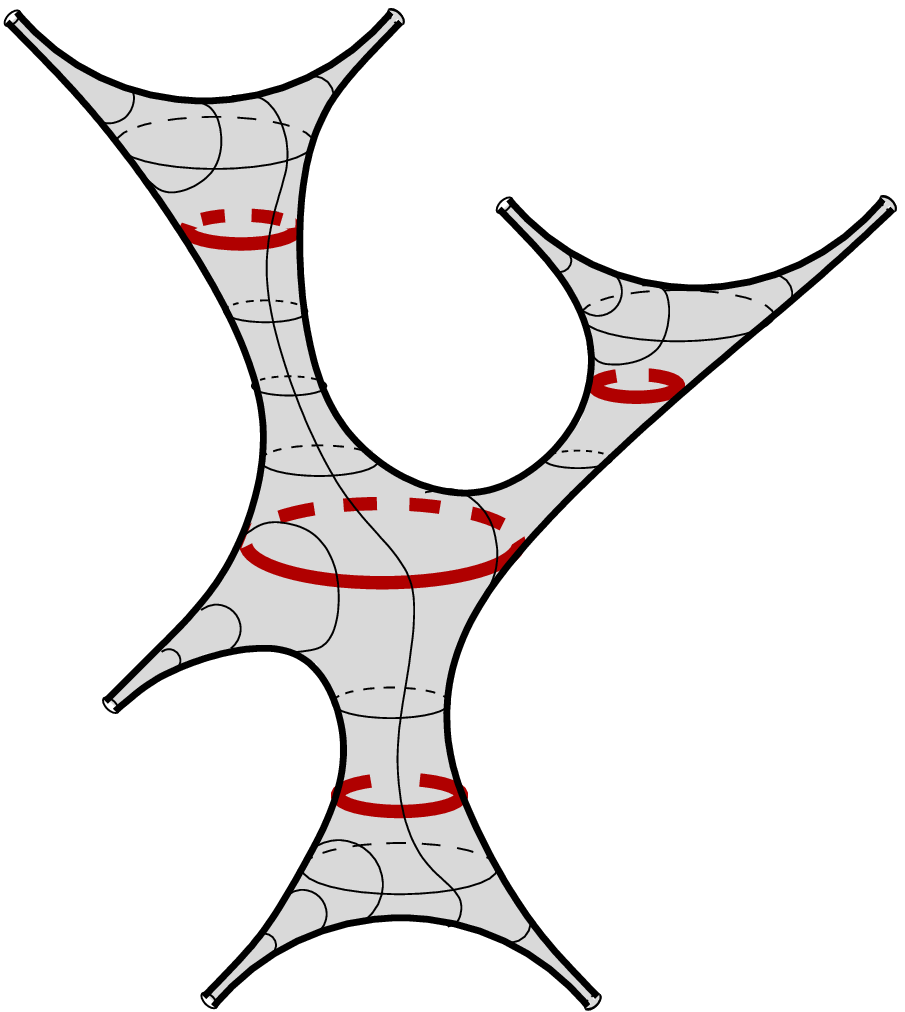}

\begin{picture}(0,0)(-26,-6)
\put(4,-20){$2\gamma_1$}
\put(53.5,-36.5){$\gamma_2$}
\put(15,-54){$\gamma_3$}
\put(20,-82){$2\gamma_4$}
\end{picture}
\vspace{95pt}
\caption{
\label{fig:square:tiled:surface:and:associated:multicurve}
A square-tiled surface in $\QuadStrat(1^3, -1^7)$,
and its associated multicurve and stable graph.
}
\end{figure}

In order to be consistent with the literature on
Masur--Veech volumes, we always label zeros and poles of
our square-tiled surfaces. Each square-tiled surface
uniquely determines the ambient stratum of quadratic
differentials. Given a list $\mu = (\mu_1, \ldots, \mu_m)$
of integers not smaller than $-1$, whose sum is $4g-4$, the
\textit{stratum of quadratic differentials}
$\QuadStrat(\mu)$ is the set of equivalence classes of
pairs: (complex curve $C$ with $m$ marked points $p_1$,
\ldots, $p_m$; a quadratic differential $q$ whose divisor
is $\sum_{i=1}^m \mu_i p_i$). The stratum $\QuadStrat(\mu)$
is naturally embedded in $\cQ_{g, m}$.

For any pair of nonnegative integers $(g,n)$ satisfying $2g+n > 3$, the \textit{principal stratum} of
meromorphic quadratic differentials in genus $g$ with $n$ simple poles
and with no other poles is $\QuadStrat(1^{4g-4+n}, -1^n)$.
By definition, $\QuadStrat(1^{4g-4+n}, -1^n)$ is a subset of
$\cQ_{g, 4g-4+2n}$. Under the morphism $\cQ_{g, 4g-4+2n} \to \cQ_{g,n}$
that forgets the points where the quadratic differential has a simple
zero, the image of the principal stratum $\QuadStrat(1^{4g-4+n}, -1^n)$ is open
and dense in $\cQ_{g,n}$. The fibers of $\QuadStrat(1^{4g-4+n}, -1^n) \to \cQ_{g,n}$ are discrete; they are in a bijective correspondence with
different ways in which one can label the $4g-4+n$ simple zeros of a
given generic quadratic differential in $\cQ_{g,n}$.

\begin{Remark}
The special cases $(g,n) = (0,3)$ and $(g,n) = (1,1)$ which
correspond to $2g+n=3$ are discussed in
Appendix~\ref{ss:Q03:and:Q11}. In these two cases
$\cQ_{g,n}$ does not admit any natural interpretation in
terms of meromorphic quadratic differentials with simple
zeros and simple poles.
\end{Remark}

We denote by $\cST(\cQ(\mu), N)$ the set of square-tiled surfaces in the
stratum $\cQ(\mu)$ made of at most $N$ squares. For example, the square-tiled
surface in Figure~\ref{fig:square:tiled:surface:and:associated:multicurve}
has genus $g=0$. It has $3$ simple zeros and $n=7$ conical singularities with angle $\pi$. Hence, it belongs to the principal stratum $\QuadStrat(1^3, -1^7)$.

We shall see in Section~\ref{ss:background:strata} that the
principal strata have a natural linear structure and that the
square-tiled surfaces form a covolume one lattice in
associated period coordinates. This justifies the following
definition of the Masur--Veech volume of $\cQ_{g,n}$ (for
$(g,n)$  different from $(0,3)$ and $(1,1)$):
\begin{equation}
\label{eq:Vol:sq:tiled}
\Vol\cQ_{g,n}
:= \Vol \QuadStrat(1^{4g-4+n}, -1^n)
= 2 d \cdot
\lim_{N\to+\infty}
\frac{\card(\cST(\QuadStrat(1^{4g-4+n}, -1^n), 2N)}{N^{d}}\,,
\end{equation}
where $d = \dim_\C \QuadStrat(1^{4g-4+n}, -1^n) = \dim_\C \cQ_{g,n} = 6g - 6 +
2n$. We emphasize that in the above formula we assume that all conical
singularities of square-tiled surfaces are labeled.
Formula~\ref{eq:Vol:sq:tiled} is the starting point of our expression for
$\Vol \cQ_{g,n}$.
\medskip

\noindent\textbf{Cylinder decomposition, multicurve and stable graph.}
A square-tiled surface admits a decomposition into
maximal horizontal cylinders filled with isometric closed
regular flat geodesics. Every such maximal horizontal
cylinder has at least one conical singularity on each of
the two boundary components. The square-tiled surface in
Figure~\ref{fig:square:tiled:surface:and:associated:multicurve}
has four maximal horizontal cylinders which are represented
in the picture by different shades.

Let $S$ be a square-tiled surface and let $S=\mathit{cyl}_1\cup \ldots\cup\mathit{cyl}_k$ be its decomposition into the set of maximal horizontal cylinders. To each cylinder $\mathit{cyl}_i$ we associate the corresponding
waist curve $\gamma_i$ considered up to a free homotopy.
The curves $\gamma_i$ are non-peripheral (i.e. none of them
bounds a disc containing a single pole) and pairwise
non-homotopic. We encode the number of circular horizontal bands of
squares contained in the corresponding maximal horizontal
cylinder by the integer weight $H_i$ associated to the
curve $\gamma_i$. The formal linear combination $\gamma=\sum
H_i\gamma_i$ is a simple closed integral multicurve in the
space $\cML_{g,n}(\Z)$ of measured laminations. For
example, the simple closed multicurve associated to the
square-tiled surface as in
Figure~\ref{fig:square:tiled:surface:and:associated:multicurve}
has the form $2\gamma_1+\gamma_2+\gamma_3+2\gamma_4$.

The multicurve $\gamma=\sum H_i\gamma_i$ as above defines the associated \textit{reduced} multicurve $\gamma_{\mathit{reduced}}=\sum \gamma_i$. Here we assume that $\gamma_i$ and $\gamma_j$ are pairwise non-isotopic for $i\neq j$. We associate to $\gamma_{\mathit{red}}$ its \textit{stable graph} $\Gamma(\gamma_{\mathit{reduced}})$ which should be thought as the dual
graph to $\gamma_{\mathit{reduced}}$. More precisely, $\Gamma(\gamma)$ is the decorated graph whose vertices represent the components of $S\setminus\gamma_{\mathit{reduced}}$ and are decorated with the genus of the corresponding component. By convention, when this number is not explicitly indicated, it is zero. The edges of $\Gamma(\gamma)$ represent the components $\gamma_i$ of $\gamma_{\mathit{reduced}}$, where the endpoints of the edge associated to $\gamma_i$ are the two vertices corresponding to the two components of $S\setminus\gamma_{\mathit{reduced}}$ adjacent to $\gamma_i$. When $\gamma_i$ has the same component of $S\setminus\gamma_{\mathit{reduced}}$ on both sides, the corresponding edge of $\Gamma(\gamma)$ is a loop. Finally, $\Gamma(\gamma)$ is endowed with $n$ ``legs`` (or half-edges) labelled from $1$ to $n$. The leg with label $i$ is attached to the vertex that represents the component that contains the $i$-th marked point of $S$. The right picture in Figure~\ref{fig:square:tiled:surface:and:associated:multicurve} shows the stable graph associated to the multicurve $\gamma$. A formal combinatorial definition of a stable graph is provided in Appendix~\ref{s:stable:graphs}.

The total number of stable graphs (considered up to
isomorphism) is finite and is equal to the number of
$\Mod_{g,n}$-orbits of reduced multicurves in
$\cML_{g,n}(\Z)$. For $2g+n > 2$, we denote by $\cG_{g,n}$
the set of stable graphs.
Table~\ref{tab:2:0} in
Section~\ref{ss:intro:Masur:Veech:volumes} and
Table~\ref{tab:1:2} in Appendix~\ref{a:1:2} list all stable
graphs in $\cG_{2,0}$ and $\cG_{1,2}$ respectively. The
special cases $\cG_{0,3}$ and $\cG_{1,1}$ are considered in
Appendix~\ref{ss:Q03:and:Q11}.

Given a pair of nonnegative integers $(g,n)$ satisfying $2g+n > 3$ and a stable graph $\Gamma$ in $\cG_{g,n}$, let us consider
the subset $\cST_{\Gamma,\mathbf{H}}(\QuadStrat(1^{4g-4+n}, -1^{n}))$ of those square-tiled
surfaces, for which the associated stable graph is $\Gamma$ and the heights of the cylinders are $\mathbf{H} = (H_1, \ldots, H_k)$. Let us denote by
$\cST_\Gamma(\QuadStrat(1^{4g-4+n}, -1^n))$ the analogous subset without
restriction on the heights. Let us define $\Vol(\Gamma, \mathbf{H})$ and
$\Vol(\Gamma)$ to be respectively the contributions to
$\Vol \cQ_{g,n}$ of square-tiled surfaces from the subsets
$\cST_{\Gamma, \mathbf{H}}(\QuadStrat(1^{4g-4+n}, -1^n))$
and $\cST_\Gamma(\QuadStrat(1^{4g-4+n}, -1^n))$:
\begin{align}
\label{eq:Vol:gamma:H}
\Vol(\Gamma, \mathbf{H})
&:= 2d\cdot
\lim_{N\to+\infty}
\frac{\card(\cST_{\Gamma, \mathbf{H}}(\QuadStrat(1^{4g-4+n}, -1^n), 2N)}{N^{d}}\,, \\
\label{eq:Vol:gamma}
\Vol(\Gamma)
&:= 2d\cdot
\lim_{N\to+\infty}
\frac{\card(\cST_\Gamma(\QuadStrat(1^{4g-4+n}, -1^n), 2N)}{N^{d}}\,,
\end{align}
where $d = 6g-6+2n$.
The results in~\cite{DGZZ:meanders:and:equidistribution}
imply that for any $\Gamma$ in $\cG_{g,n}$ the above limits
exist, are strictly positive, and that
\begin{equation}
\label{eq:Vol:Q:as:sum:of:Vol:gamma}
\Vol \cQ_{g,n}
= \sum_{\Gamma \in \cG_{g,n}} \Vol(\Gamma)\,
= \sum_{\Gamma \in \cG_{g,n}} \sum_{\mathbf{H} \in \N^{E(\Gamma)}} \Vol(\Gamma, \mathbf{H})\,.
\end{equation}
Dividing both sides of~\eqref{eq:Vol:Q:as:sum:of:Vol:gamma} by $\Vol \cQ_{g,n}$ we see that the ratio $\Vol(\Gamma)/\Vol \cQ_{g,n}$ can be interpreted as the ``asymptotic probability'' that a square-tiled surface taken at random has $\Gamma$ as stable graph associated to its horizontal cylinder decomposition.

\begin{Remark}
\label{rk:degeneration:horizontal:cylinders}
A stable graph is commonly used to encode the boundary classes of the Deligne-Mumford compactification $\overline{\cM}_{g,n}$ of $\cM_{g,n}$. Informally speaking, the stable curves in the boundary of $\overline{\cM}_{g,n}$ are obtained by pinching along appropriate multicurves $\gamma_{\mathit{reduced}}$. In our situation, this can be done algebraically in the following way. Let $(C, q)$ be a quadratic differential whose horizontal cylinders fill the associated flat surface $S$. Consider the sequence of quadratic differentials $(C_t, q_t)$ obtained as follows. Define $q_t := \Re(q) + i e^t \Im(q)$. The differential $q_t$ is meromorphic for a unique complex structure $C_t$. The horizontal cylinder decompositions of all $(C_t, q_t)$, $t\in\R$, are topologically identical. Metrically, cylinders of $(C_t, q_t)$ are $e^t$ times higher than cylinders of $(C,q)$. The path $(C_t, q_t)$ in $\overline{\cQ_{g,n}}$ converges towards a stable quadratic differential $(C_\infty, q_\infty)$ in $\overline{\cQ_{g,n}}$ with exactly double poles at the nodes of $C_\infty$. Each double pole corresponds to a half-infinite cylinder associated to the corresponding boundary component of $S\setminus\gamma_{\mathit{reduced}}$.
\end{Remark}

\subsection{Ribbon graphs, intersection numbers and volume polynomials}
\label{ss:intro:volume:polynomials}

In this section we introduce multivariate polynomials
$N_{g,n}(b_1, \ldots, b_n)$ that appear in different
contexts. They are an essential ingredient to our formula
for the Masur--Veech volume.

Let $g,n$ be non-negative integers with $2g+n > 2$. Let $b_1$, \ldots, $b_n$
be variables. For a multi-index $\boldsymbol{d} = (d_1, \ldots, d_n)$ we denote
by $b^{2\boldsymbol{d}}$ the product $b_1^{2d_1}\cdot\cdots\cdot b_n^{2d_n}$,
by $|\boldsymbol{d}|$ the sum $d_1 + \cdots + d_n$ and by
$\boldsymbol{d}!$ the product $d_1! \cdots d_n!$

Define the homogeneous polynomial $N_{g,n}(b_1,\dots,b_n)$ of degree
$6g-6 + 2n$ in the variables $b_1,\dots,b_n$ as
\begin{equation}
\label{eq:N:g:n}
N_{g,n}(b_1,\dots,b_n)=
\sum_{|\boldsymbol{d}|=3g-3+n}c_{\boldsymbol{d}} b^{2\boldsymbol{d}}\,,
\end{equation}
where
\begin{equation}
\label{eq:c:subscript:d}
c_{\boldsymbol{d}}=\frac{1}{2^{5g-6+2n}\, \boldsymbol{d}!}\,
\langle \psi_1^{d_1} \dots \psi_n^{d_n}\rangle
\end{equation}
\begin{equation}
\label{eq:correlator}
\langle \psi_1^{d_1} \dots \psi_n^{d_n}\rangle
=\int_{\overline{\cM}_{g,n}} \psi_1^{d_1}\dots\psi_n^{d_n}\,,
\end{equation}
where $\psi_1$, \ldots, $\psi_n$ are the $\psi$-classes on the Deligne--Mumford
compactification $\overline{\cM}_{g,n}$. That is, $\psi_i$ is the first Chern class of the $i$-th tautological bundle over $\overline{\cM}_{g,n}$.
Informally, the fiber of this bundle over $(C,p_1,\dots,p_n)$
is the cotangent line $T^\ast_{p_i}C$ to $C$ at the $i$-th marked point. Note that $N_{g,n}(b_1,\dots,b_n)$
contains only even powers of $b_i$, where $i=1,\dots,n$. For small $g$ and $n$
we get:
$$
\begin{array}{ll}
N_{0,3}(b_1,b_2,b_3)&=1\\
[-\halfbls] &\\
N_{0,4}(b_1,b_2,b_3,b_4)&=\cfrac{1}{4}(b_1^2+b_2^2+b_3^2+b_4^2)\\
[-\halfbls] &\\
N_{1,1}(b_1)&=\cfrac{1}{48}(b_1^2)\\
[-\halfbls] &\\
N_{1,2}(b_1,b_2)&=\cfrac{1}{384}(b_1^2+b_2^2)(b_1^2+b_2^2)\,.
\end{array}
$$

A \emph{ribbon graph} $G$ is a connected graph endowed at each vertex with a cyclic ordering of adjacent edges. The cyclic ordering determines faces, so one can consider a tubular neighborhood of $G$ as a surface with boundary, where boundary components correspond to faces of $G$. We denote by $\cR_{g,n}$ the set of isomorphism classes of trivalent ribbon graphs of genus $g$ with $n$ faces labeled from $1$ to $n$. For each trivalent ribbon graph $G$ in $\cR_{g,n}$ and any collection of integers $b_1$, \ldots, $b_n$, denote by $\cN_G(b_1, \ldots, b_n)$ the number of integral metrics on $G$ (assigning a positive integral length to each edge) such that the perimeters of the faces get lengths $b_1$, \ldots, $b_n$.

\begin{NNTheorem}[Kontsevich]
Consider a collection of positive integers $b_1,\dots, b_n$
such that $\sum_{i=1}^n b_i$ is even.
The weighted count of genus $g$ connected trivalent metric ribbon graphs $G$
with integer edges and with $n$ labeled boundary components of lengths $b_1,\dots,b_n$
is equal to $N_{g,n}(b_1,\dots,b_n)$ up to the lower order terms:
$$
\sum_{G \in \cR_{g,n}} \frac{1}{|\Aut(G)|}\, \cN_G(b_1,\dots,b_n)
=N_{g,n}(b_1,\dots,b_n)+\text{lower order terms}\,,
\,
$$
where $\cR_{g,n}$ denotes the set of (nonisomorphic)
trivalent ribbon graphs $G$ of genus $g$ and with $n$
boundary components.
\end{NNTheorem}

This Theorem is a part of Kontsevich's
proof~\cite{Kontsevich} of Witten's
conjecture~\cite{Witten}.

\begin{Remark}
P.~Norbury~\cite{Norbury} and
K.~Chapman--M.~Mulase--B.~Safnuk~\cite{Chapman:Mulase:Safnuk}
refined the count of Kontsevich proving that the function
counting lattice points in the moduli space $\cM_{g,n}$
corresponding to covers of the sphere ramified over three
points (the so-called \textit{Grothendieck's dessins d'enfants}) is a
quasi-polynomial in variables $b_i$. In other terms, when
considering all ribbon graphs (and not only trivalent ones)
the lower order terms in Kontsevich's theorem form a
quasi-polynomial. This quasi-polynomiality of the
expression on the left hand side endows the notion of
``lower order terms'' with a natural formal sense.
\end{Remark}

We also use the following common notation for the
intersection numbers~\eqref{eq:correlator}. Given an
ordered partition $d_1+\dots+d_n=3g-3+n$ of $3g - 3 + n$
into a sum of non-negative integers we define

\begin{equation}
\label{eq:tau:correlators}
\langle \tau_{d_1} \dots \tau_{d_n}\rangle
=\int_{\overline{\cM}_{g,n}} \psi_1^{d_1}\dots\psi_n^{d_n}\,.
\end{equation}

\subsection{Formula for the Masur--Veech volumes}
\label{ss:intro:Masur:Veech:volumes}

Following~\cite{AEZ:genus:0} we consider the following
linear operators $\cY(\boldsymbol{H})$ and $\cZ$ on the spaces of
polynomials in variables $b_1,b_2,\dots$, where $H_1, H_2, \dots$
are positive integers.
The operator $\cY(\boldsymbol{H})$ is defined on monomials as
\begin{equation}
\label{eq:cV}
\cY(\boldsymbol{H})\ :\quad
\prod_{i=1}^{k} b_i^{m_i} \longmapsto
\prod_{i=1}^{k} \frac{m_i!}{H_i^{m_i+1}}\,,
\end{equation}
and extended to arbitrary polynomials by linearity.
The operator $\cZ$ is defined on monomials as
\begin{equation}
\label{eq:cZ}
\cZ\ :\quad
\prod_{i=1}^{k} b_i^{m_i} \longmapsto
\prod_{i=1}^{k} \big(m_i!\cdot \zeta(m_i+1)\big)\,,
\end{equation}
and extended to arbitrary polynomials by linearity.
In the above
formula $\zeta$ is the Riemann zeta function
$$
\zeta(s) = \sum_{n \geq 1} \frac{1}{n^s}\,,
$$
so for any collection of strictly positive integers
$(m_1,\dots,m_k)$ one has
$$
\cZ\left(\prod_{i=1}^{k} b_i^{m_i}\right)
=\sum_{\boldsymbol{H}\in\N^k}
\cY(\boldsymbol{H})\left(\prod_{i=1}^{k} b_i^{m_i}\right)\,.
$$

\begin{Remark}
\label{rm:zeta:even:integers}
For even integers $2m$ we have
$$
\zeta(2m) = (-1)^{m+1} \frac{B_{2m} (2\pi)^{2m}}{2\, (2m)!}
$$
where $B_{2m}$ are the Bernoulli numbers. Consider a homogeneous polynomial in
$k$ variables of degree $2m-k$ with rational coefficients, such that all powers
of all variables in each monomial are odd.
The observation above implies that the value of $\cZ$ on
such polynomial is a rational number multiplied by
$\pi^{2m}$.
\end{Remark}

Given a stable graph $\Graph$ denote by $V(\Gamma)$ the set
of its vertices and by $E(\Gamma)$ the set of its edges. To
each stable graph $\Gamma\in\cG_{g,n}$ we associate the
following homogeneous polynomial $P_\Gamma$
of degree $6g-6+2n$. To
every edge $e\in E(\Gamma)$ we assign a formal variable
$b_e$. Given a vertex $v\in V(\Gamma)$ denote by $g_v$ the
integer number decorating $v$ and denote by $n_v$ the
valency of $v$, where the legs adjacent to $v$ are counted
towards the valency of $v$. Take a small neighborhood of
$v$ in $\Gamma$. We associate to each half-edge (``germ''
of edge) $e$ adjacent to $v$ the monomial $b_e$; we
associate $0$ to each leg. We denote by $\boldsymbol{b}_v$
the resulting collection of size $n_v$. If some edge $e$ is
a loop joining $v$ to itself, $b_e$ would be present in
$\boldsymbol{b}_v$ twice; if an edge $e$ joins $v$ to a
distinct vertex, $b_e$ would be present in
$\boldsymbol{b}_v$ once; all the other entries of
$\boldsymbol{b}_v$ correspond to legs; they are represented
by zeroes. To each vertex $v\in E(\Gamma)$ we associate the
polynomial $N_{g_v,n_v}(\boldsymbol{b}_v)$, where $N_{g,v}$
is defined in~\eqref{eq:N:g:n}. We associate to the stable
graph $\Gamma$ the polynomial obtained as the product
$\prod b_e$ over all edges $e\in E(\Graph)$ multiplied by
the product $\prod N_{g_v,n_v}(\boldsymbol{b}_v)$ over all
$v\in V(\Graph)$. We define $P_\Gamma$ as follows:
\begin{multline}
\label{eq:P:Gamma}
P_\Gamma(\boldsymbol{b})
=
\frac{2^{6g-5+2n} \cdot (4g-4+n)!}{(6g-7+2n)!}\cdot
\\
\frac{1}{2^{|V(\Graph)|-1}} \cdot
\frac{1}{|\operatorname{Aut}(\Graph)|}
\cdot
\prod_{e\in E(\Graph)}b_e\cdot
\prod_{v\in V(\Graph)}
N_{g_v,n_v}(\boldsymbol{b}_v)
\,.
\end{multline}

Table~\ref{tab:2:0} in
Section~\ref{ss:intro:Masur:Veech:volumes} and
Table~\ref{tab:1:2} in Appendix~\ref{a:1:2} list the polynomials associated to all stable
graphs in $\cG_{2,0}$ and in $\cG_{1,2}$.

\begin{Theorem}
\label{th:volume}
The Masur--Veech volume of the stratum of quadratic differentials with $4g-4+n$
simple zeros and $n$ simple poles has the following value:
\begin{multline}
\label{eq:square:tiled:volume}
\Vol \cQ_{g,n} = \Vol \QuadStrat(1^{4g-4+n}, -1^n)
\\= \sum_{\Graph \in \cG_{g,n}} \Vol(\Gamma)
= \sum_{\Graph \in \cG_{g,n}} \sum_{\mathbf{H} \in \N^{E(\Gamma)}} \Vol\left(\Gamma, \mathbf{H}\right)\,,
\end{multline}
where the contributions of individual stable graphs $\Gamma$ defined by~\eqref{eq:Vol:gamma:H} (respectively the contributions of pairs $(\Gamma,\boldsymbol{H})$ defined by~\eqref{eq:Vol:gamma}) are equal to
\begin{equation}
\label{eq:volume:contribution:of:stable:graph}
\Vol(\Gamma)=\cZ(P_\Gamma)
\quad \text{and} \quad
\Vol\left(\Gamma,\boldsymbol{H}\right) =\cY(\boldsymbol{H})(P_\Gamma).
\end{equation}
\end{Theorem}

Table~\ref{tab:2:0} below illustrates the computation of
the polynomials $P_\Gamma$ and of the contributions
$\Vol(\Gamma)$ to the Masur--Veech volume $\Vol\cQ_{g,n}$
in the particular case of $(g,n)=(2,0)$. To make the
calculation tractable, we follow the structure of
Formula~\eqref{eq:P:Gamma}.
The first numerical factor $\frac{128}{5}$ represents the factor
$\frac{2^{6g-5+2n} \cdot (4g-4+n)!}{(6g-7+2n)!}$. It is
common for all stable graphs in $\cG_{2,0}$. The second
numerical factor in the first line of each calculation
in Table~\ref{tab:2:0} is $\frac{1}{2^{|V(\Graph)|-1}}$, where $|V(\Graph)|$ is the
number of vertices of the corresponding stable graph
$\Gamma$ (equivalently ---
the number of connected components of the complement to the
associated reduced multicurve).
The third numerical factor is
$\frac{1}{|\Aut(\Graph)|}$. Recall that neither
vertices nor edges of $\Graph$ are labeled. We first evaluate the
order of the corresponding automorphism group
$|\Aut(\Graph)|$ (this group respects the decoration
of the graph), and only then we associate to edges of $\Graph$
variables $b_1, \dots, b_k$ in an arbitrary way.

\begin{table}[hbt]
$$
\hspace*{20pt}
\begin{array}{llr}


\includegraphics{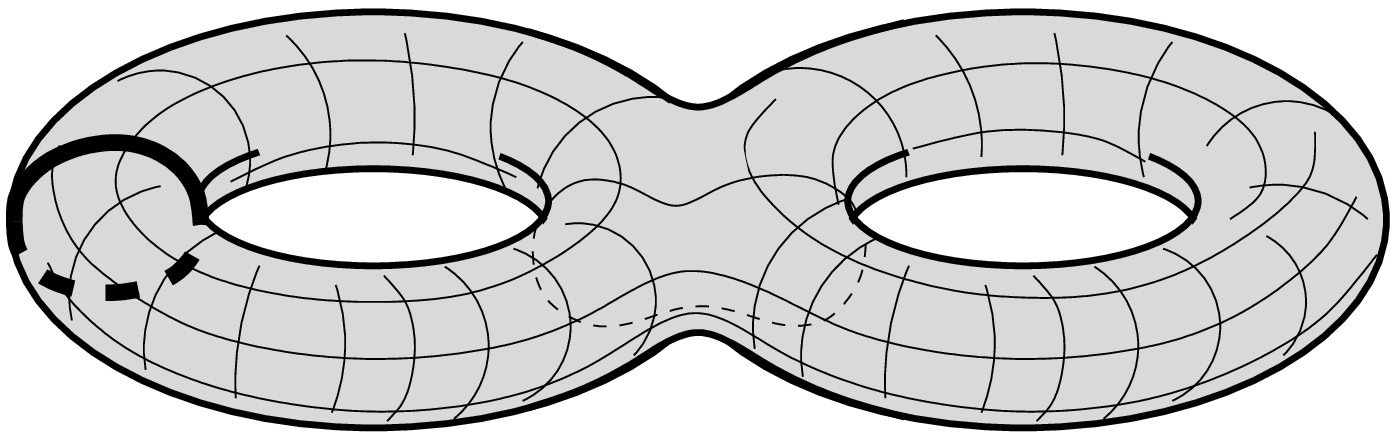}
\begin{picture}(0,0)(0,0)
\put(-24,0){$b_1$}
\end{picture}
\hspace*{68pt} 
\vspace*{5pt}

&
\frac{128}{5}\cdot
1\cdot
\frac{1}{2}\cdot
b_1 \cdot
N_{1,2}(b_1,b_1)=

&
\\


\includegraphics{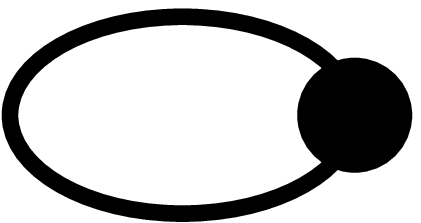}
\begin{picture}(0,0)(0,0)
\put(-10,-10.5){$b_1$}
\put(30,-10.5){$1$}
\end{picture}
\vspace*{5pt}

&
\hspace*{7pt}
=\frac{64}{5}\cdot b_1 \left(\frac{1}{384}(2b_1^2)(2b_1^2)\right)
=\frac{2}{15}\cdot b_1^5\xmapsto{\cZ}
&
\frac{2}{15}\cdot \big(5!\cdot \zeta(6)\big)
\\


&&
=\frac{16}{945}\cdot\pi^6

\vspace*{5pt}\\\hline&&\\

\includegraphics{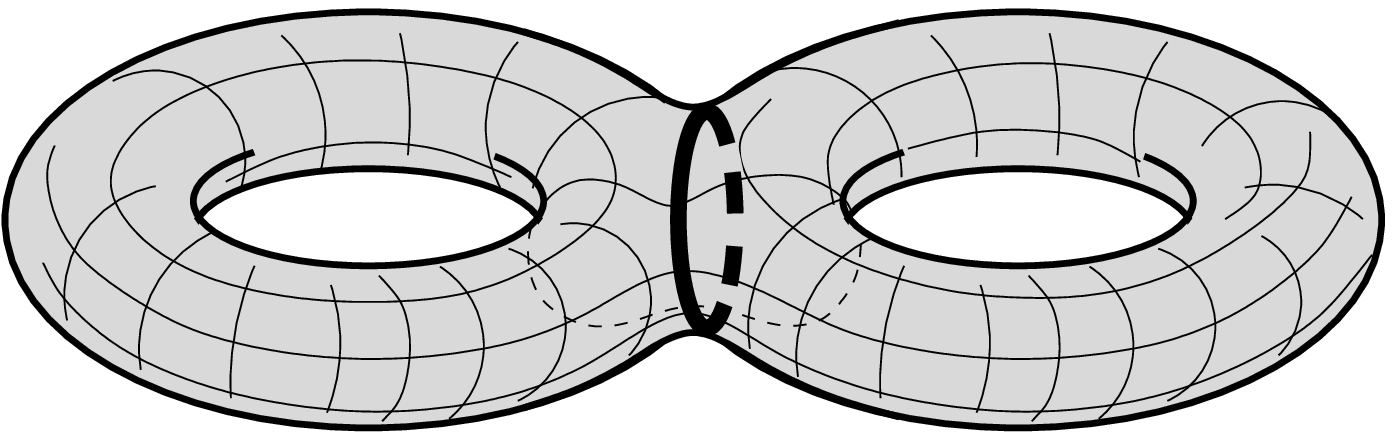}
\begin{picture}(0,0)(0,0)
\put(23,-14){$b_1$}
\end{picture}
\vspace*{5pt}

&
\frac{128}{5}\cdot
\frac{1}{2}\cdot
\frac{1}{2}\cdot
b_1 \cdot N_{1,1}(b_1)\cdot N_{1,1}(b_1)=

&
\\


\includegraphics{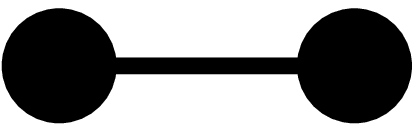}
\begin{picture}(0,0)(0,0)
\put(23,-20){$b_1$}
\put(6,-14){$1$}
\put(41,-14){$1$}
\end{picture}
\vspace*{5pt}

&
\hspace*{6pt}

=\frac{32}{5}\cdot b_1\cdot \left(\frac{1}{48}b_1^2\right)\left(\frac{1}{48}b_1^2\right)
=\frac{1}{360}\cdot b_1^5\xmapsto{\cZ}
&

\frac{1}{360}\cdot \big(5!\cdot \zeta(6)\big)

\\

&&

=\frac{1}{2835}\cdot\pi^6

\vspace*{5pt}\\\hline&&\\

\includegraphics{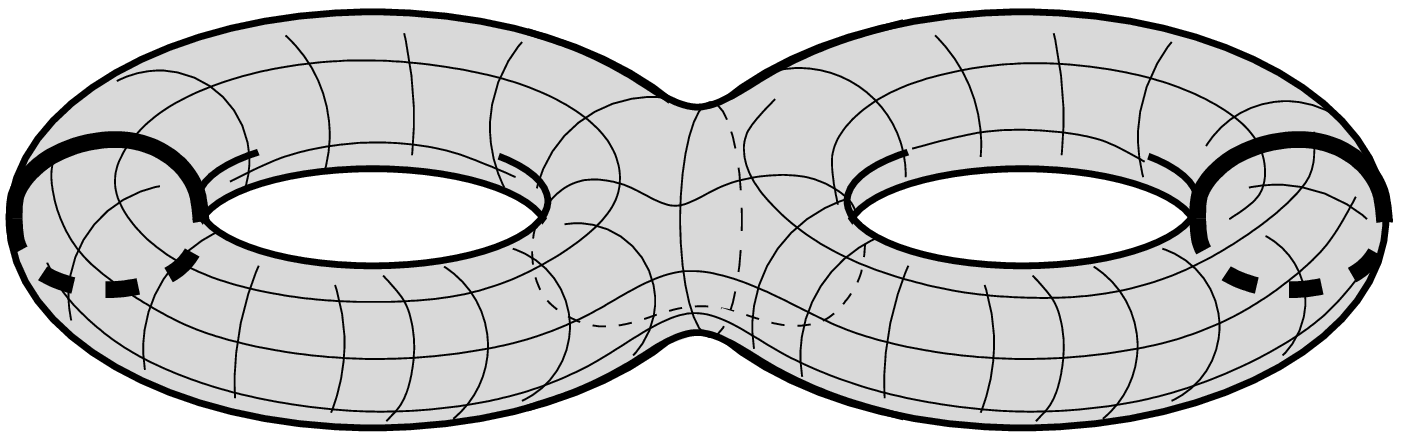}
\begin{picture}(0,0)(0,0)
\put(-24,0){$b_1$}
\put(66,0){$b_2$}
\end{picture}
\vspace*{5pt}

&
\frac{128}{5}\cdot
1 \cdot
\frac{1}{8}\cdot
b_1 b_2 \cdot
N_{0,4}(b_1,b_1,b_2,b_2)=

&

\\


\includegraphics{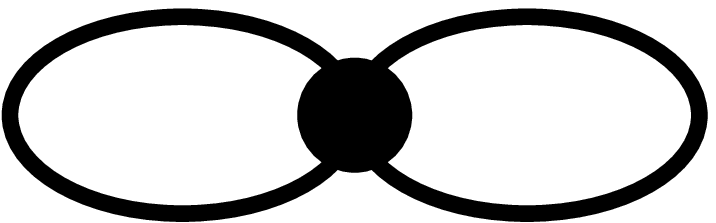}
\begin{picture}(0,0)(0,0)
\put(-9,-19){$b_1$}
\put(52,-19){$b_2$}
\put(23,-30){$0$}
\end{picture}
\vspace*{5pt}

&
\hspace*{20pt}
=\frac{16}{5}\cdot b_1 b_2\cdot\left(\frac{1}{4}(2b_1^2+2b_2^2)\right)=

&\\

&

\hspace*{80pt}
=\frac{8}{5}(b_1^3 b_2+b_1 b_2^3)\xmapsto{\cZ}

&

\frac{8}{5}\!\cdot\! 2\!\cdot\! 3!\zeta(4)\!\cdot\!1!\zeta(2)

\vspace*{5pt}\\ &&

=\frac{8}{225}\cdot \pi^6

\vspace*{5pt}\\\hline&&\\

\includegraphics{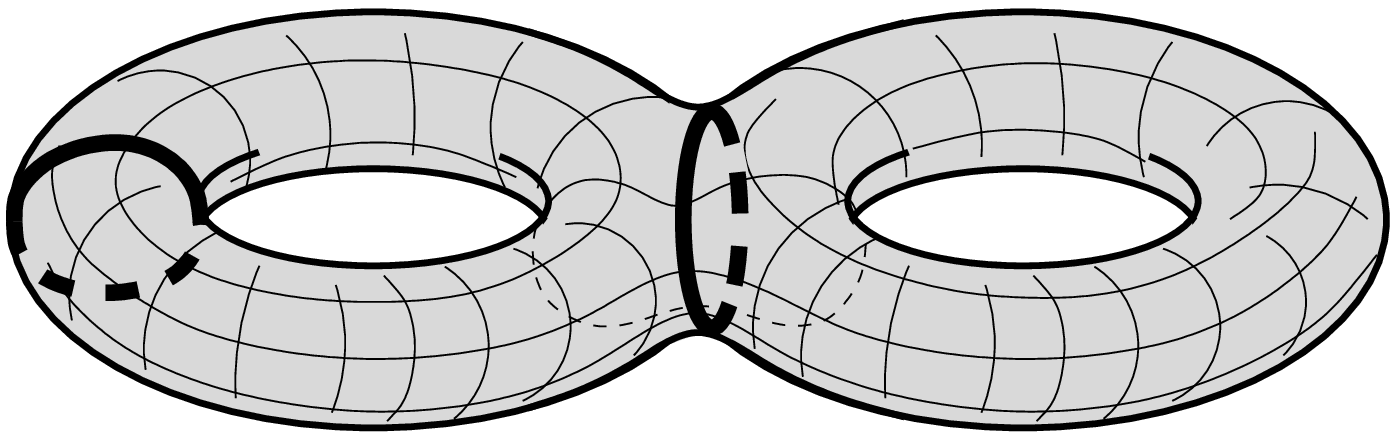}
\begin{picture}(0,0)(0,0)
\put(-24,0){$b_1$}
\put(23,-14){$b_2$}
\end{picture}
\vspace*{5pt}

&

\frac{128}{5}\!\cdot\!
\frac{1}{2}\!\cdot\!
\frac{1}{2}\!\cdot\!
b_1 b_2\!\cdot\! N_{0,3}(b_1,b_1,b_2)\!\cdot\! N_{1,1}(b_2)

&

\\


\includegraphics{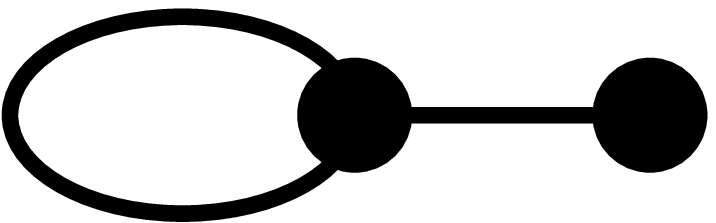}
\begin{picture}(0,0)(0,0)
\put(-19,-10){$b_1$}
\put(23,-16){$b_2$}
\put(6,-10){$0$}
\put(41,-10){$1$}
\end{picture}
\vspace*{5pt}

&
=\frac{32}{5}\cdot b_1 b_2\cdot\big(1\big) \cdot \left(\frac{1}{48} b_2^2\right)
=\frac{2}{15}\cdot b_1 b_2^3\xmapsto{\cZ}
&
\frac{2}{15}\!\cdot\! 1!\zeta(2)\!\cdot\!3!\zeta(4)
\\


&&
=\frac{1}{675}\cdot\pi^6

\vspace*{5pt}\\\hline&&\\

\includegraphics{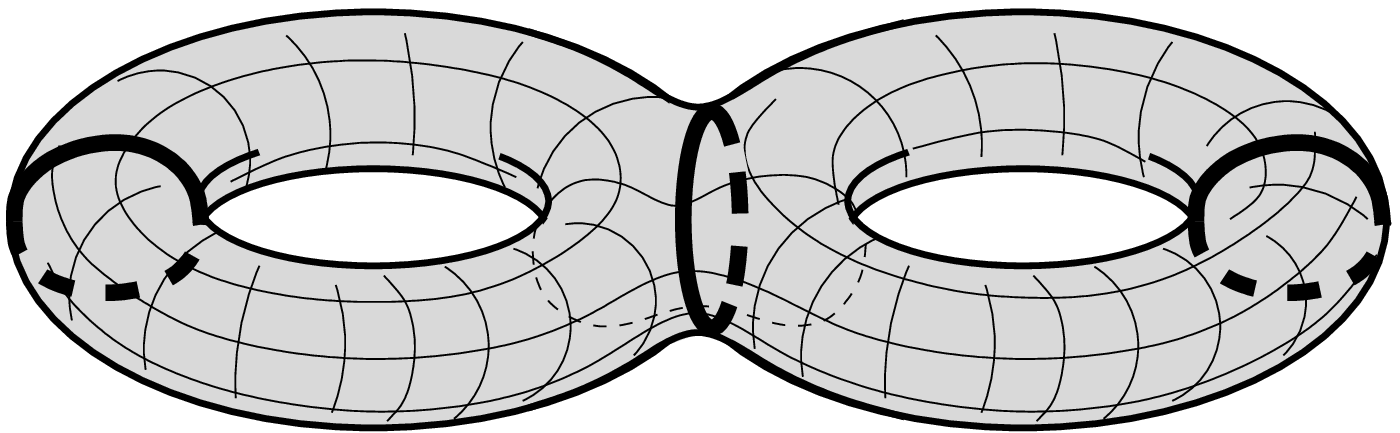}
\begin{picture}(0,0)(0,0)
\put(-24,0){$b_1$}
\put(23,-14){$b_2$}
\put(66,0){$b_3$}
\end{picture}
\vspace*{5pt}

&
\frac{128}{5}\!\cdot\!
\frac{1}{2}\!\cdot\!
\frac{1}{8}\!\cdot\!
b_1 b_2 b_3 \!\cdot\!
N_{0,3}(b_1,b_1,b_2)\cdot

&

\\


\includegraphics{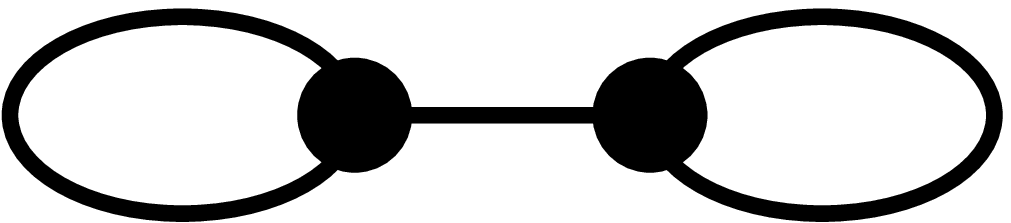}
\begin{picture}(0,0)(0,0)
\put(-18,-10){$b_1$}
\put(23,-16){$b_2$}
\put(63,-10){$b_3$}
\put(5.5,-10){$0$}
\put(42,-10){$0$}
\end{picture}
\vspace*{5pt}

&

\cdot N_{0,3}(b_2,b_3,b_3)
=\frac{8}{5}\!\cdot\! b_1 b_2 b_3\!\cdot\! (1) \!\cdot\! (1)
\xmapsto{\cZ}

&
\frac{8}{5}\cdot \left(1!\,\zeta(2)\right)^3
\\


&&
=\frac{1}{135}\cdot\pi^6

\vspace*{5pt}\\\hline&&\\

\includegraphics{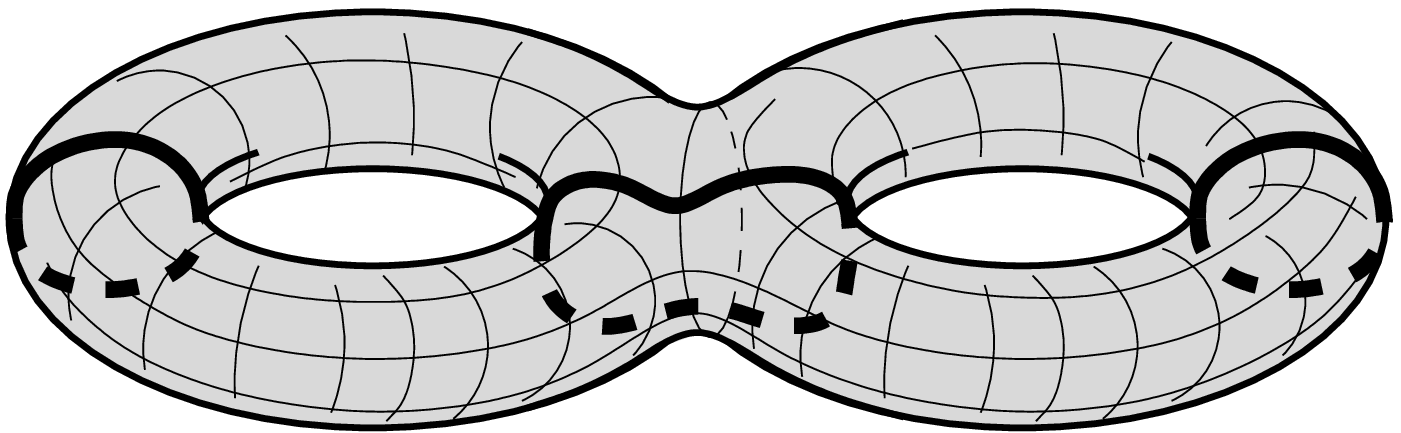}
\begin{picture}(0,0)(0,0)
\put(-24,0){$b_1$}
\put(20,12){$b_2$}
\put(66,0){$b_3$}
\end{picture}
\vspace*{5pt}

&

\frac{128}{5}\!\cdot\!
\frac{1}{2}\!\cdot\!
\frac{1}{12}\!\cdot\!
b_1 b_2 b_3 \!\cdot\!
N_{0,3}(b_1,b_1,b_2)\cdot

&

\\

\includegraphics{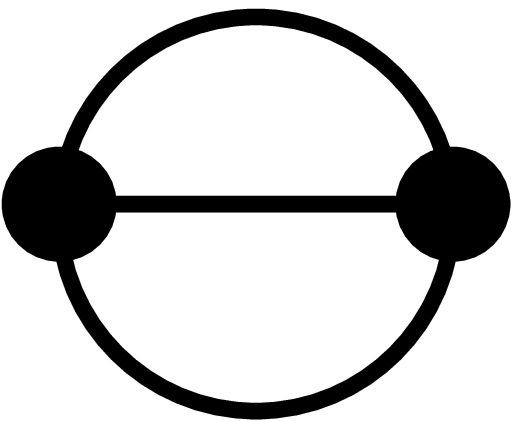}
\begin{picture}(0,0)(0,0)
\put(4,-16){$b_1$}
\put(27.5,-16){$b_2$}
\put(43,-16){$b_3$}
\put(31,2){$0$}
\put(31,-35){$0$}
\end{picture}
\vspace*{5pt}

&

\cdot N_{0,3}(b_2,b_3,b_3)
=\frac{16}{15}\!\cdot\! b_1 b_2 b_3\!\cdot\! (1) \!\cdot\! (1)
\xmapsto{\cZ}

&
\frac{16}{15}\cdot \left(1!\,\zeta(2)\right)^3
\\


&&
=\frac{2}{405}\cdot\pi^6

\vspace*{20pt}
\\
\end{array}
$$
\caption{
\label{tab:2:0}
Computation of $\Vol \cQ_{2,0}$.
The left column represents
the stable graphs $\Gamma$
and their associated multicurves;
the middle column gives
the polynomials
$P_\Gamma$; the right column
provides $\Vol(\Gamma)$.
}
\end{table}

Taking the sum of the six contributions
we obtain the answer, matching the value found in~\cite{Goujard:volumes}
by implementing the method of A.~Eskin and A.~Okounkov~\cite{Eskin:Okounkov:pillowcase}.
$$
\Vol\cQ_{2,0}=
\left(
\left(\tfrac{16}{945}+\tfrac{1}{2835}\right)
+
\left(\tfrac{8}{225}+\tfrac{1}{675}\right)
+
\left(\tfrac{1}{135}+\tfrac{2}{405}\right)
\right) \pi^6
=
\tfrac{1}{15} \pi^6\,.
$$
   %


\begin{Remark}
In genus $0$, the formula simplifies considerably.
It was conjectured by M.~Konstevich and proved by
J.~Athreya, A.~Eskin and A.~Zorich in~\cite{AEZ:genus:0}
that for all $n \geq 4$
\begin{equation}
\label{eq:vol:genus:0}
\Vol \cQ_{0,n} = \Vol \QuadStrat(1^{n-4}, -1^n) = \frac{\pi^{2n-6}}{2^{n-5}}\,.
\end{equation}

Note that in genus $0$ all correlators of $\psi$-classes
admit a closed explicit expression. Rewriting all polynomials
$N_{0,n_v}(\boldsymbol{b}_v)$ for all stable graphs in
$\cG_{0,n}$ in Formula~\eqref{eq:square:tiled:volume} for
$\Vol \cQ_{0,n}$ in terms of the corresponding multinomial
coefficients we get the formula originally obtained
in~\cite{AEZ:Dedicata}. A lot of techniques in this
article are borrowed from~\cite{AEZ:Dedicata}. Note,
however, that the proof of~\eqref{eq:vol:genus:0} is
indirect and is based on analytic
Riemann--Roch--Hierzebruch formula and on fine comparison
of asymptotics of determinants of flat and hyperbolic
Laplacians as $(X,q)$ in $\cQ_{0,n}$ approaches the
boundary.
\end{Remark}

\begin{Remark}
\label{rm:Chen:Moeller:Sauvaget}
In the recent paper~\cite{Chen:Moeller:Sauvaget}
D.~Chen, M.~M\"oller, A.~Sauvaget
proved
an alternative formula for $\Vol \QuadStrat(1^{4g-4+n}, -1^n)$
expressing it
as a weighted sum of certain linear Hodge integrals. Their approach is based on intersection theory.
Combined with the recursive formula for the
linear Hodge integrals obtained by M.~Kazarian in~\cite{Kazarian}, it allows to compute
the exact values of $\Vol \QuadStrat(1^{4g-4+n}, -1^n)$ for $g$ up to $250$
and, basically, for any $n$. In particular, it provides
an alternative proof of~\eqref{eq:vol:genus:0}.
\end{Remark}

Note the following important feature of
Formula~\eqref{eq:square:tiled:volume} which distinguishes
it from the approach of
Eskin--Okounkov~\cite{Eskin:Okounkov:Inventiones,
Eskin:Okounkov:pillowcase} based on the quasimodularity
of certain generating functions or from the approach of
Chen--M\"oller--Sauvaget~\cite{Chen:Moeller:Sauvaget}
based on intersection theory.
Formula~\eqref{eq:square:tiled:volume} allows us to analyze
the contribution of individual stable graphs to
$\Vol \cQ_{g,n}$. In particular, it allows us to study the
statistical geometry of random square-tiled surfaces as in Sections~\ref{ss:Combinatorial:geometry:of:random:st:surfaces}
and statistical properties of random simple closed hyperbolic geodesics as in Section~\ref{s:2:correlators}.

Formula~\eqref{eq:square:tiled:volume} also implies
the following asymptotic lower bound for the Masur--Veech
volume $\Vol \cQ_{g,0}$ and a conjectural asymptotic value.

\begin{Theorem}
\label{th:asymptotic:lower:bound}
The following inequality holds for any $g\ge 2$:
\begin{equation}
\label{eq:asymptotic:lower:bound}
\Vol \cQ_{g,0}
>\sqrt{\frac{2}{3\pi g}}
\cdot\left(\frac{8}{3}\right)^{4g-4}
\cdot\left(1-\frac{2}{6g-7}\right)\,.
\end{equation}
\end{Theorem}
Theorem~\ref{th:asymptotic:lower:bound} is proved in
Section~\ref{s:simple:closed:geodesics:asymptotics}.

\begin{Conjecture}
\label{conj:Vol:Qg}
The Masur--Veech volume of the moduli space of holomorphic
quadratic differentials has the following large genus
asymptotics\footnote{After a journal submission of the current paper
a stronger form of this Conjecture was proved by A.~Aggarwal
in~\cite{Aggarwal:correlators}.}:
\begin{equation}
\label{eq:Vol:Qg}
\Vol \cQ_{g,0}
\overset{?}{=}
\frac{4}{\pi}
\cdot\left(\frac{8}{3}\right)^{4g-4}
\cdot\left(1+O\left(\frac{1}{g}\right)\right)
\quad\text{as }g\to+\infty\,.
\end{equation}
\end{Conjecture}

The exact values of $\Vol \cQ_{g,0}$ obtained by combining
results~\cite{Chen:Moeller:Sauvaget} and~\cite{Kazarian}
corroborate the conjectural Formula~\eqref{eq:Vol:Qg}.
The work~\cite{Yang:Zagier:Zhang} also corroborates and
develops this conjecture suggesting several first terms
of the asymptotic expansion in the negative powers of $g$.

A conjectural generalization of
Formula~\eqref{eq:Vol:Qg} to other strata of meromorphic
quadratic differentials and numerical evidence beyond this
conjecture are presented in~\cite{ADGZZ:conjecture}.
The statistical geometry of random square-tiled surfaces of
large genus and of random simple closed multicurves on
surfaces of large genus is discussed in the separate
paper~\cite{DGZZ:statistics}.

\begin{Remark}
\label{rm:Vol:power:of:pi}
By the result of A.~Eskin and
A.~Okounkov~\cite{Eskin:Okounkov:Inventiones} the
Masur--Veech volume of any stratum of Abelian differentials
is a rational multiple of $\pi^{2g}$. For the strata of
meromorphic quadratic differentials the
results~\cite{Eskin:Okounkov:pillowcase} directly imply
slightly weaker property: the Masur--Veech volume is a
polynomial in powers of $\pi^2$ with rational coefficients,
see the proof of Proposition~5.10
in~\cite{Goujard:volumes}. Actually, one can derive from the
results~\cite{Eskin:Okounkov:pillowcase} a
stronger statement, namely, that the Masur--Veech volume of
any stratum of meromorphic quadratic differentials is a
rational multiple of $\pi^{2g_{\mathit{eff}}}$, where the
integer number $g_{\mathit{eff}}$ denotes the
\textit{effective genus} of the stratum. However, this
implication is already non-trivial and was never written down.
Alternative proofs of the latter statement were recently
obtained by D.~Chen, M.~M\"oller, and A.~Sauvaget for the
principal stratum in~\cite{Chen:Moeller:Sauvaget} and by
V.~Koziarz and \mbox{D.-M.~Nguyen} in~\cite{Koziarz:Nguyen} for
certain more general arithmetic affine $\GL$-invariant
submanifolds of the strata.

By construction, the polynomial $P_\Gamma(\boldsymbol{b})$
associated to a stable graph $\Graph\in\cG_{g,n}$ by
Expression~\eqref{eq:P:Gamma} is a homogeneous polynomial
of degree $6g-6+2n-|E(\Gamma)|$. Moreover, each variable
$b_e$ in each monomial appears with an odd power.
Thus, Formula~\eqref{eq:square:tiled:volume} implies
that the contribution $\Vol(\Gamma)$ of each stable graph
$\Gamma\in\cG_{g,n}$ to $\Vol \cQ_{g,n}$ is already a
rational multiple of $\pi^{6g-6+2n}$. Moreover, using the
refined version of $N_{g,n}(b_1, \ldots, b_n)$ due to
Norbury~\cite{Norbury} expressing the counting functions of
ribbon graphs as quasi-polynomials in $b_i$, one can even
show that the generating series of square-tiled surfaces
corresponding a given stable graph is a quasi-modular form.
This result develops the results of
Eskin-Okounkov~\cite{Eskin:Okounkov:pillowcase} that say
that in each stratum, the generating series for the count
of pillowcase covers (in the sense of A.~Eskin and
A.~Okounkov) is a quasimodular form and the analogous
result of P.~Engel~\cite{Engel1}, \cite{Engel2} for the
count of square-tiled surfaces.
\end{Remark}

\begin{Remark}
\label{rm:ref:to:Engel}
Up to a normalization constant given by the explicit
Formula~\eqref{eq:Vgn:through:Ngn} (depending only on $g$ and $n$),
the polynomial $N_{g,n}(b_1,\dots,b_n)$
coincides with the top homogeneous part of Mirzakhani's
volume polynomial $V_{g,n}(b_1,\dots,b_n)$ providing the
Weil--Petersson volume of the moduli space of bordered
Riemann
surfaces~\cite{Mirzakhani:simple:geodesics:and:volumes}.
The classical Weil--Petersson volume of
$\cM_{g,n}$ corresponds to the constant term of
$V_{g,n}(b_1,\dots,b_n)$ when the lengths of all boundary
components are equal to zero. This relation between correlators $\langle \psi_1^{d_1}
\ldots \psi_n^{d_n}\rangle$ and Weil--Petersson volumes
allowed Mirzakhani to provide an alternative proof
of Witten's conjecture.

In the paper~\cite{ABCDGLW},
J.~E.~Andersen, G.~Borot, S.~Charbonniery, V.~Delecroix, A.~Giacchetto,
D.~Lewa\'nski, and C.~Wheeler define a \textit{Masur--Veech}
polynomial whose top degree term is also (a rescaled) version of $N_{g,n}(b_1,
\ldots, b_n)$ and its constant term is the Masur--Veech volume.

Both the Weil--Petersson polynomial and the Masur--Veech
polynomial satisfy a topological recursion that allows
direct computations without relying on formulae such as~\eqref{eq:square:tiled:volume}
which involves the (huge) list of stable graphs.

Note that contrarily to the volume polynomial $V_{g,n}$, the topological
recursion for the Masur--Veech polynomial does not admit a geometric
interpretation yet.
\end{Remark}

\begin{Remark}
\label{rm:to:complete}
Formulae~\eqref{eq:P:Gamma}--\eqref{eq:volume:contribution:of:stable:graph}
admit a generalization allowing one to express Masur--Veech
volumes of those strata $\cQ(d_1,\dots, d_n)$, for which
all zeroes have odd degrees $d_i$, in terms of intersection
numbers of $\psi$-classes with the combinatorial cycles in
$\overline{\cM}_{g,n}$ associated to the strata (denoted by
$W_{m_*,n}$ in \cite{Arbarello:Cornalba}, where $m_*$ is the sequence of
multiplicities of the zeroes). In this more general case
the formula requires additional correction subtracting the
contribution of those quadratic differentials which
degenerate to squares of globally defined Abelian
differentials. This generalization is a work in progress.
\end{Remark}

\begin{Remark}
Note that the contribution of square-tiled surfaces having
a fixed number of cylinders to the Masur--Veech volume of
more general strata of quadratic differentials might have a
much more sophisticated arithmetic nature.
In~\cite{DGZZ:one:cylinder:Yoccoz:volume} we describe the contribution of square-tiled surfaces having a
single maximal horizontal cylinder to the Masur--Veech
volume of any stratum of Abelian or quadratic
differentials. We conjecture that the contribution
of $k$-cylinder square-tiled surfaces to the Masur--Veech volume of the ambient stratum is expressed as a polynomial with rational coefficients in multiple-zeta values.
\end{Remark}

\subsection{Siegel--Veech constants}
\label{ss:intro:Siegel:Veech:constants}
We now turn to a formula for the Siegel--Veech constants of
$\cQ_{g,n}$. We first recall the definition of
Siegel--Veech constants that involves the flat geometry of
quadratic differentials.

Let $(C,q)$ be a non-zero quadratic differential in
$\cQ_{g,n}$. It naturally defines a
Riemannian metric $|q|$ which is flat with conical
singularities exactly at the zeros and poles of $q$. This
metric allows to define geodesics and we say that a
geodesic is \textit{regular} if it does not pass through
the singularities of $q$. Closed regular flat geodesics
appear in families composed of parallel closed geodesics of
the same length. Each such family fills a maximal flat
cylinder $\mathit{cyl}$ having a conical singularity
(possibly the same) at each of the two boundary components.
The length of any regular geodesic in this family
is called the \textit{width} (or \textit{circumference}) of the cylinder .
The \textit{height} of the cylinder is the distance between its boundary components measured inside the cylinder. In
particular, the flat area of the cylinder is the product of its width
and its height.

For $S=(C,q)$ in $\cQ_{g,n} \setminus \{0\}$ and any $L\in\R$, the
number of maximal cylinders in $S$ filled with regular closed
geodesics of  bounded  length  $w(cyl)\le L$ is finite.
Thus the following quantity is well-defined:
\begin{equation}
\label{eq:N:area}
N_{\mathit{area}}(S,L):=
\frac{1}{\Area(S)}
\sum_{\substack{
\mathit{cyl}\subset S\\
w(\mathit{cyl})\le L}}
\Area(cyl)\,.
\end{equation}

For any pair of nonnegative integers $(g,n)$ satisfying $2g+n>3$,
choose the Masur--Veech volume element $\dVolMV$ in $\cQ_{g,n}$
which coincides with the Masur--Veech volume element on the principal
stratum $\cQ(1^{4g-4+n}, -1^n)$. This volume element induces a canonical volume element
$\dVolMV_1$ on any level hypersurface $\cQ^{\Area=a}_{g,n}$.
The following theorem is a special case
of the fundamental result of W.~Veech, \cite{Veech:Siegel}
developed by Y.~Vorobets in~\cite{Vorobets}.

\begin{NNTheorem}[W.~Veech; Ya.~Vorobets]
Let $(g,n)$ be a pair of non-negative integers such that
$2g+n>3$. There exists a strictly positive constant
$\carea(\cQ_{g,n})$ such that for any strictly
positive numbers
$a$ and $L$ the following holds:
\begin{equation}
\label{eq:SV:constant:definition}
\frac{a}{\pi L^2}\int_{\cQ^{\Area=a}_{g,n}}
N_{\mathit{area}}(S,L)\,d\Vol_1(S)
=
\Vol_1\cQ^{\Area=a}_{g,n}
\ \cdot\ \carea(\cQ_{g,n})\,.
\end{equation}
\end{NNTheorem}
This formula is called the \textit{Siegel--Veech formula}, and the
corresponding  constant $\carea(\cQ_{g,n})$ is called the
\textit{Siegel--Veech constant}. Note that $\carea(\cQ_{g,n})$, actually,
does not depend on the choice of the normalization of the Masur--Veech volume.

Eskin and Masur~\cite{Eskin:Masur} proved
that for almost all $S=(C,q)$ in $\cQ_{g,n}$
(with respect to the Masur--Veech measure)
\begin{equation}
\label{eq:SV:asymptotics}
\lim_{L\to+\infty}
\Area(S) \cdot \frac{N_{\mathit{area}}(S,L)}{\pi L^2}=
\carea(\cQ_{g,n})\,.
\end{equation}

\begin{Remark}
Beyond its geometrical relevance, let us mention that the
area Siegel--Veech constant is the most important
ingredient in the Eskin--Kontsevich--Zorich formula for the
sum of the Lyapunov exponents of the Hodge bundle along the
Teichm\"uller geodesic flow~\cite{Eskin:Kontsevich:Zorich}.
\end{Remark}

An edge of a connected graph is called a \textit{bridge} if
the operation of removing this edge breaks the graph into
two connected components.
We define the following function $\separatingQ: E(\Graph)
\to \{\tfrac{1}{2}, 1\}$ on the set of edges of
any connected graph $\Gamma$:
$$
\separatingQ(e)=
\begin{cases}
\tfrac{1}{2}&\text{if the edge $e$ is a bridge\,,}\\
1  &\text{otherwise}\,.
\end{cases}
$$

We define the following operator $\partial_{\Graph}$ on
polynomials $P$ in variables $b_e$ associated to the edges
of stable graphs $\Graph\in\cG_{g,n}$. For every
$e\in E(\Gamma)$ let
\begin{equation}
\label{eq:operator:D:e}
\partial^e_{\Graph} P
:=\separatingQ(e)\, b_e\, \left.\frac{\partial P}{\partial b_e}\right|_{b_e=0}\,,
\end{equation}
and let
\begin{equation}
\label{eq:operator:D}
\partial_{\Graph} P
:=
\sum_{e \in E(\Graph)}
\partial^e_{\Graph} P\,.
\end{equation}

\begin{Theorem}
\label{th:carea}
Let $g,n$ be non-negative integers satisfying $2g+n > 3$.
The Siegel--Veech constant $\carea(\cQ_{g,n})$
satisfies the following relation:
\begin{equation}
\label{eq:carea}
\Vol \cQ_{g,n}
 \cdot
\carea(\cQ_{g,n})
=
\frac{3}{\pi^2}
\cdot
\sum_{\Graph \in \cG_{g,n}}
\cZ\left(\partial_{\Graph}
P_\Gamma\right)\,.
\end{equation}
\end{Theorem}

As an illustration of the above Theorem we compute
$\carea(\cQ_{2,0})$ and $\carea(\cQ_{1,2})$ in
appendices~\ref{a:2:0} and~\ref{a:1:2} respectively.

\subsection{Masur--Veech Volumes and Siegel--Veech constants}
\label{ss:intro:Siegel:Veech:Masur:Veech}
The Siegel--Veech constant can be expressed in terms of the
Masur--Veech volumes of certain boundary strata. The
formula for strata of Abelian differentials was obtained
in~\cite{Eskin:Masur:Zorich} and for strata of quadratic
differentials in~\cite{Goujard:carea}. Before presenting a
reformulation of Corollary~1 in~\cite{Goujard:carea} we
introduce the following conventions for $(g,n)$ being
$(0,3)$ or $(1,1)$ (see Appendix~\ref{ss:Q03:and:Q11} for a
discussion)
\begin{eqnarray}
\label{eq:convention:Q03}
 &\Vol \cQ_{0,3} := 4\,,
 \\
\label{eq:convention:Q11}
 &\Vol \cQ_{1,1} :=\frac{2\pi^2}{3}\,.
\end{eqnarray}

\begin{NNTheorem}[\cite{Goujard:carea}]
Let $g,n$ be non-negative integers with $2g+n > 3$. Under
Conventions~\eqref{eq:convention:Q03}--\eqref{eq:convention:Q11}
the following formula is
valid:
\begin{align}
\label{eq:carea:Elise}
\carea(\cQ_{g,n}) \cdot \Vol \cQ_{g,n} =
\frac{1}{8}
 \sum_{\substack{g_1+g_2=g
 \\ n_1+n_2=n+2 \\
 g_i\geq 0, n_i\geq 1, d_i\geq 1}}
 \frac{\ell!}{\ell_1!\ell_2!}\frac{n!}{(n_1-1)!(n_2-1)!}\cdot
 \\
\notag
\cdot\frac{(d_1-1)!(d_2-1)!}{(d-1)!}
\Vol \cQ_{g_1, n_1} \times\Vol \cQ_{g_2,n_2}
\ +\hspace*{80pt}
\\+\
\notag
\frac{1}{16}\cdot\frac{(4g-4+n)n(n-1)}{(6g-7+2n)(6g-8+2n)}
\Vol \cQ_{0, 3} \times\Vol \cQ_{g,n-1}\ +
\\
\notag
+\ \frac{\ell!}{(\ell-2)!}\frac{(d-3)!}{(d-1)!}
 \Vol \cQ_{g-1,n+2}\,.
\end{align}
Here $d=\dim_{\C}\cQ_{g,n}=6g-6+2n$, $d_i=6g_i-6+2n_i$,
$\ell=4g-4+n$, $\zeroes_i=4g_i-4+n_i$\,.

For $g=0$ and any integer $n$ satisfying $n\ge 4$
the following formula is valid:
\begin{multline}
\label{eq:carea:g0:Elise}
\carea(\cQ_{0,n}) \cdot \Vol \cQ_{0,n}=
\frac{1}{8}
 \sum_{\substack{ n_1+n_2=n+2 \\  n_i\geq 4}}
 \frac{(n-4)!}{(n_1-4)!(n_2-4)!}
\cdot
\\
\cdot
 \frac{n!}{(n_1-1)!(n_2-1)!}
\cdot\frac{(2n_1-7)!(2n_2-7)!}{(2n-7)!}
\Vol \cQ_{0, n_1}\times\Vol \cQ_{0,n_2}
+\\+
\frac{1}{16}\cdot\frac{(n-4)n(n-1)}{(2n-7)(2n-8)}
\Vol \cQ_{0, 3}\times\Vol \cQ_{0,n-1}\,.
\end{multline}
\end{NNTheorem}

The terms which involve $(g,n)=(0,3)$ in the formulae above can be interpreted as particular cases of the
corresponding general terms under the following
convention. In the context of the formulae above
it is natural to define
\begin{equation}
\label{eq:minus:one:factorial}
\left.\frac{(d_i-1)!}{\ell_i!}\right|_{\substack{g=0\\n_i=3}}:=
\left.\frac{(2n_i-7)!}{(n_i-4)!}\right|_{n_i=3}:=
\lim_{n\to 3}\frac{\Gamma(2n-6)}{\Gamma(n-3)}=\frac{1}{2}\,.
\end{equation}

Note that the expressions appearing in the right-hand sides
of~\eqref{eq:carea}, \eqref{eq:carea:Elise}
and~\eqref{eq:carea:g0:Elise} can be seen as polynomials in
correlators. More precisely, in the definition of $N_{g_v,
n_v}(\boldsymbol{b}_v)$ one can keep the
correlators
$\langle\psi_1^{d_1} \ldots \psi_k^{d_k}\rangle
=\langle\tau_{d_1} \ldots \tau_{d_k}\rangle$
in~\eqref{eq:c:subscript:d}
without evaluation. We extend the operators $\cZ$
and $\partial_\Gamma$
to polynomials in the variables $b_e$ and
in ``unevaluated'' correlators
by linearity.
For example, under such convention one gets
\[
\Vol \cQ_{0,5}
=
\frac{\pi^2}{9}
\left( 5 \langle \tau_0^3 \tau_1\rangle\, \langle \tau_0^3 \rangle
+ 4 \langle \tau_0^3 \rangle^3 \right)
\quad \text{and} \quad
\frac{\pi^2}{3} \carea(\cQ_{0,5}) \Vol \cQ_{0,5} = \frac{5}{9} \langle \tau_0^3 \rangle^3.
\]
Numerical values of volumes $\Vol \cQ_{g,n}$
and of Siegel--Veech constants
$\carea(\cQ_{g,n})$
for small $g$ and $n$ are
presented in Table~\ref{tab:Vol:Q:g:n}
in Appendix~\ref{a:tables}. The corresponding expressions in terms of
the intersection numbers are available
in~\cite{DGZZ:tables}.

Viewed in this way, the right-hand sides
of~\eqref{eq:carea} and of~\eqref{eq:carea:Elise} in the
case of $g\ge 1$ (respectively, the right-hand sides
of~\eqref{eq:carea} and of~\eqref{eq:carea:g0:Elise}
in the case $g=0$)
provide identities between polynomials in intersection numbers.
We show that these identities are, actually, trivial.

\begin{Theorem}
\label{th:same:SV}
The right-hand sides of~\eqref{eq:carea} and
of~\eqref{eq:carea:Elise} for $g\ge 1$
(respectively, the right-hand sides of~\eqref{eq:carea} and
of~\eqref{eq:carea:g0:Elise} for $g=0$) considered
as polynomials in intersection numbers of $\psi$-classes
coincide.
\end{Theorem}

Theorem~\ref{th:same:SV} is proved in
Section~\ref{eq:proof:of:coincidence:of:two:SV:expressions}.

\subsection{Frequencies of multicurves (after M.~Mirzakhani)}
\label{ss:Frequencies:of:simple:closed:curves}

Let $g,n$ be non-negative integers with $2g+n>2$.
We say that two integral multicurves on the same
smooth surface of genus $g$ with $n$ punctures
\textit{have the same topological type}
if they belong to the same orbit of the mapping class group
$\Mod_{g,n}$. As we have already seen,
topological types of primitive multicurves (respectively
multicurves) are in bijection with stable graphs in
$\cG_{g,n}$ (respectively stable graphs in $\cG_{g,n}$
labeled with a height $\boldsymbol{H}_i$ at each edge).

Let $C$ be a complex curve, $C\in\cM_{g,n}$. We denote by $X$ the underlying
Riemann surface endowed with its hyperbolic metric of constant curvature $-1$.
Following M.~Mirzakhani, given an integral multicurve $\gamma$ in
$\cML_{g,n}(\Z)$ we define $s_X(L,\gamma)$ as the number
of simple closed geodesic multicurves on $X$ of length at
most $L$ having the same topological type as $\gamma$.
M.~Mirzakhani analyzed the asymptotic behavior of $s_X(L, \gamma)$
which involves several quantities that we define now.

The hyperbolic length function $\ell_X$ defined on multicurves
admits a continuous extension to $\cML_{g,n}$, see~\cite{Kerckhoff}. Hence, we can consider
the unit ball $B_X$ defined as
\begin{equation}
\label{eq:unit:ball}
B_X:= \{\gamma \in \cML: \ell_X(\gamma) \leq 1\}\,.
\end{equation}
in $\cML_{g,n}$ and the associated volume with
respect to Thurston's measure\footnote{Thurston's measure on $\cML_{g,n}$ admits two natural normalizations which differ by a constant factor. Following
W.~Thurston and M.~Mirzakhani,
we use the normalization under which the set
$\cML_{g,n}(\Z)$
of integral multicurves, playing the role of an
integer lattice, has covolume one in the ambient piecewise-linear
space $\cML_{g,n}$. The alternative normalization is
induced from the symplectic structure, see~\cite{Monin:Telpukhovskiy}, \cite{Erlandsson:Souto} and Corollary~\ref{cor:c:tilde} below for further details
and for the value of the constant factor relating the two normalizations.}
$$
B(X)=\mu_{\mathrm{Th}}(B_X)\,.
$$
Next, we define the number $b_{g,n}$ as the mean value of $B(X)$
\begin{equation}
\label{eq:b:g:n}
b_{g,n}:=\int_{\cM_{g,n}} B(X)\,dX\,.
\end{equation}
Here we integrate with respect to the Weil--Petersson volume form
$dX$ on $\cM_{g,n}$.

We can now state one of the main results of M.~Mirzakhani
from~\cite{Mirzakhani:grouth:of:simple:geodesics}.
\begin{NNTheorem}[M.~Mirzakhani]
Let $(g,n)$ be non-negative integers with $2g+n > 2$.
Let $\gamma$ be a multicurve in $\cML_{g,n}(\Z)$. Then there exists a positive
constant $c(\gamma)$ such that for any Riemann surface $X$
of genus $g$ with $n$ punctures we have
$$
s_X(L,\gamma)\sim B(X)\cdot\frac{c(\gamma)}{b_{g,n}}\cdot L^{6g-6+2n}\,,
$$
as $L\to+\infty$.
\end{NNTheorem}

Note that in this beautiful asymptotic formula all information about the
hyperbolic metric $X$ is carried by the factor $B(X)$ (which does not
depend on $\gamma$) and the topological information about $\gamma$ is carried
by the constant $c(\gamma)$ (which does not depend on $X$).
Mirzakhani showed furthermore that
\begin{equation}
\label{eq:b:g:n:as:sum:of:c:gamma}
b_{g,n}=\sum_{[\gamma]\in\cO} c(\gamma)\,,
\end{equation}
where the sum of $c(\gamma)$ is taken with respect to
representatives $[\gamma]$ of all orbits $\cO$ of
the mapping class group $\Mod_{g,n}$ in $\cML_{g,n}(\Z)$
as in the sum~\eqref{eq:Vol:Q:as:sum:of:Vol:gamma}.

This allows to interpret the ratio
$\tfrac{c(\gamma)}{b_{g,n}}$ as the probability to get a
multicurve of type $\gamma$ taking a ``long random''
multicurve (in the same sense as the probability that
the coordinates of a ``random'' point in $\Z^2$ are coprime
equals $\tfrac{6}{\pi^2}$). More precisely, M.~Mirzakhani
showed that the asymptotic frequency
$\frac{c(\gamma)}{b_{g,n}}$ represents the density of the
orbit $\Mod_{g,n}\cdot\gamma$ inside the set of all
integral simple closed multicurves $\cML_{g,n}(\Z)$.
This
density is analogous to the density $\frac{6}{\pi^2}$ of
integral points with coprime coordinates in $\Z^2$
represented by the $\SLZ$-orbit of the vector $(1,0)$.

We now present a consequence of the bridge between
square-tiled surfaces and multicurves that will appear in
the next section. Let $g \geq 2$. There is a single
topological type of a nonseparating simple closed curve in $\cML_{g,0}$ and
$\lfloor g / 2 \rfloor$ classes of separating simple closed
curves. We define
\[
c_{g,nonsep} := c(\gamma_{nonsep})\,,
\qquad
c_{g,sep} := \sum_{[\gamma_{sep}]} c(\gamma_{sep})\,,
\]
where $\gamma_{nonsep}$ is the non-separating simple closed curve and
the sum in the second term is over the $\lfloor g/2 \rfloor$ classes
of separating simple closed curves.
\begin{Theorem}
\label{th:separating:over:non:separating}
The frequency of separating simple closed geodesics
on a closed hyperbolic surface of large genus $g$ is exponentially
small with respect to the frequency of non-separating
simple closed geodesics:
\begin{equation}
\label{eq:log:sep:over:non:sep}
\frac{c_{g,sep}}{c_{g,nonsep}}
\sim
\sqrt{\frac{2}{3\pi g}}\cdot\frac{1}{4^g}\quad\text{as }g\to+\infty\,.
\end{equation}
\end{Theorem}

Theorem~\ref{th:separating:over:non:separating} follows from
analyzing some individual contributions of particularly simple
stable graphs to the Masur--Veech volume of $\cQ_{g,0}$ and is
proved in Section~\ref{ss:simple:closed:geodesics}. The proof is
based on the large genus asymptotic formulae for
$2$-correlators
$\langle \psi_1^{d_1}\psi_2^{d_2}\rangle$
uniform for all partitions $d_1+d_2=3g-1$.
This formula is
obtained in Section~\ref{s:2:correlators} using results
of~\cite{Zograf:2:correlators}.

\begin{Remark}
In order to go beyond the case of simple closed curve, one
has to carry a much more involved asymptotic analysis of
correlators that we perform in full generality
in~\cite{DGZZ:statistics}.
\end{Remark}

\begin{figure}[htb]
   %
   %
\includegraphics{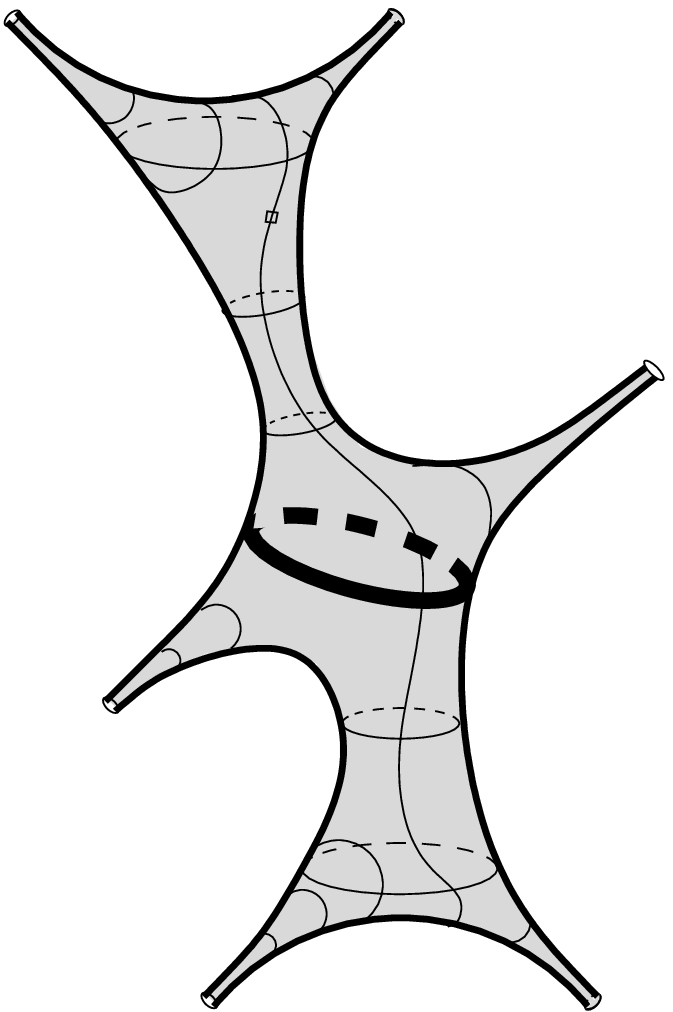}

\includegraphics{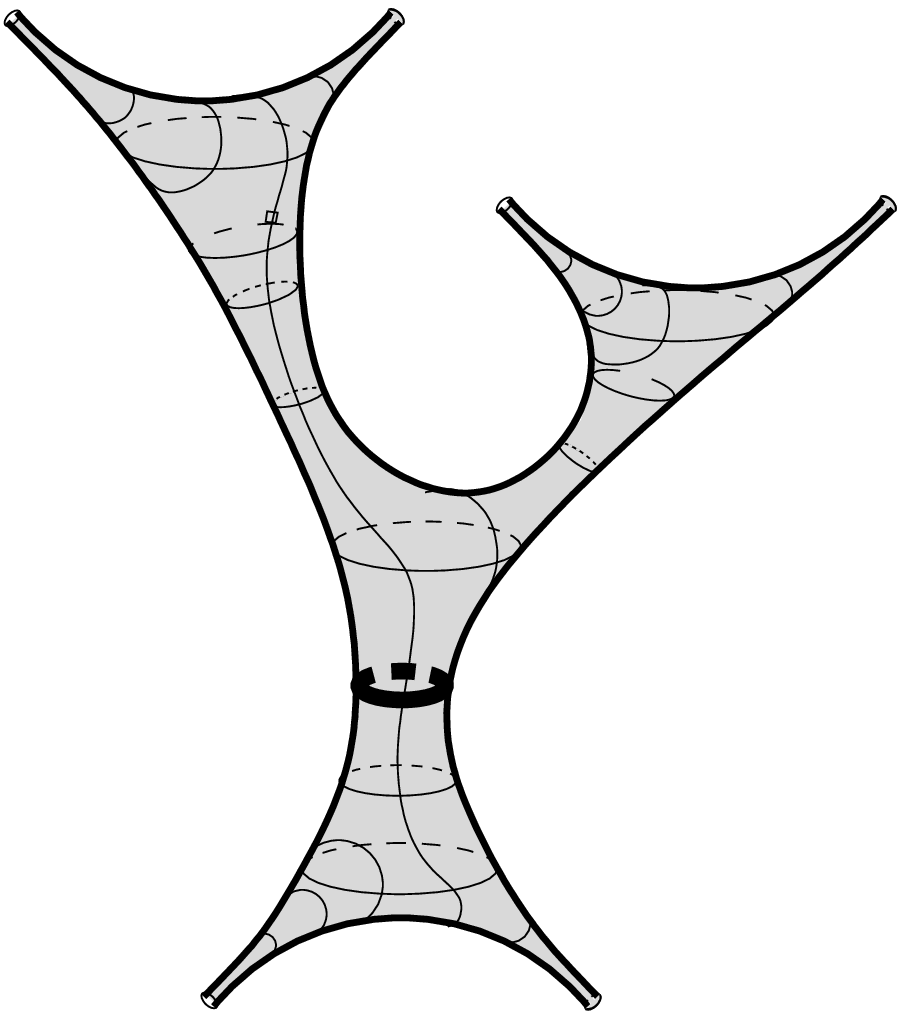}

\vspace{110pt}
\caption{
\label{fig:six:punctured:spheres}
Simple closed curves on a six-punctured sphere}
\end{figure}

Using hyperbolic geometry and the recursion relations for
Weil--Petersson volumes, M.~Mirzakhani found an explicit expression for the
coefficient $c(\gamma)$ and for the global normalization
constant $b_{g,n}$ in terms of the intersection numbers of
$\psi$-classes.

\begin{Example}
\label{ex:c:gamma:for:six:punctured:spheres} For any
hyperbolic metric $X$ on a sphere with $6$ cusps as in
Figure~\ref{fig:six:punctured:spheres}, a long simple
closed geodesic separates the cusps into groups of $3+3$
cusps with probability $\frac{4}{7}$ and into $2+4$ cusps
with probability $\frac{3}{7}$ (see (2) on page~123
in~\cite{Mirzakhani:grouth:of:simple:geodesics} for
calculation).
\end{Example}

\begin{Remark}
These values were confirmed experimentally in 2017 by
M.~Bell; see~\cite{Bell} for a description of the computations.
They were also confirmed by more implicit
independent computer experiment by V.~Delecroix.
\end{Remark}

\subsection{Frequencies of square-tiled surfaces of fixed
combinatorial type}
\label{ss:Frequencies:of:square:tiled:surfaces}

The following Theorem bridges flat and hyperbolic
count. Recall from Section~\ref{ss:MV:volume} that
weighted stable graphs are in bijection with topological
classes of multicurves.

\begin{Theorem}
\label{th:our:density:equals:Mirzakhani:density}
Let $(g,n)$ be a pair of non-negative integers satisfying $2g+n > 3$ and
different from $(2,0)$. Let $\gamma\in\cML_{g,n}(\Z)$ be
a multicurve and $(\Gamma, \boldsymbol{H})$ the associated
stable graph and weights. Then the volume contribution
$\Vol(\Gamma, \boldsymbol{H})$ to the Masur--Veech
volume $\Vol\cQ_{g,n}$ coincides with the Mirzakhani's
asymptotic frequency $c(\gamma)$ of closed geodesic
multicurves of topological type $\gamma$
up to an explicit factor depending only on $g$ and $n$:
\begin{equation}
\label{eq:Vol:gamma:c:gamma}
\Vol(\Gamma, \boldsymbol{H})
=
2\cdot(6g-6+2n)\cdot
(4g-4+n)!\cdot 2^{4g-3+n}\cdot
c(\gamma)\,.
\end{equation}
\end{Theorem}
Theorem~\ref{th:our:density:equals:Mirzakhani:density}
is proved in Section~\ref{s:comparison:with:Mirzakhani}.

\begin{Example}
\label{ex:Vol:gamma:for:six:punctured:spheres} A
one-cylinder square-tiled surface in the moduli space
$\cQ_{0,6}$ can have $3$ simple poles on each of the two
boundary component of the maximal horizontal cylinder or
can have $2$ simple poles on one boundary component and $4$
simple poles on the other boundary component. The
asymptotic frequency of
square-tiled surfaces of the first type is
$\frac{4}{7}$ and the asymptotic frequency of the
square-tiled surfaces of second type
is $\frac{3}{7}$; compare to
Example~\ref{ex:c:gamma:for:six:punctured:spheres}.
\end{Example}

Combining~\eqref{eq:Vol:Q:as:sum:of:Vol:gamma},
\eqref{eq:b:g:n:as:sum:of:c:gamma} and~\eqref{eq:Vol:gamma:c:gamma} we
get the following immediate Corollary:

\begin{Corollary}
\label{cor:Vol:as:b:g:n}
For any pair of non-negative integers satisfying $2g+n > 3$ and different from
$(2,0)$, the Masur--Veech volume $\Vol\cQ_{g,n}$ and the average
Thurston measure of a unit ball $b_{g,n}$ are related as
follows:

\begin{equation}
\label{eq:Vol:g:n:b:g:n}
\Vol\cQ_{g,n}
=2\cdot(6g-6+2n)\cdot
(4g-4+n)!\cdot 2^{4g-3+n}\cdot
b_{g,n}\,.
\end{equation}
\end{Corollary}

\begin{Remark}
In the case, when $n=0$, M.~Mirzakhani established in Theorem~1.4 in~\cite{Mirzakhani:earthquake} a relation
similar to~\eqref{eq:Vol:g:n:b:g:n} between $b_g=b_{g,0}$  (computed in Theorem~5.3 in~\cite{Mirzakhani:grouth:of:simple:geodesics}) and the volume of $\cQ_g$.
However, Mirzakhani does not give any formula for
the value of the normalization constant
presented in~\eqref{eq:Vol:g:n:b:g:n}.
This constant was recently computed by
F.~Arana--Herrera~\cite{Arana:Herrera:square:tiled} and by L.~Monin and I.~Telpukhovkiy~\cite{Monin:Telpukhovskiy} simultaneously
and independently of us by different methods. The same
value of the constant in~\eqref{eq:Vol:g:n:b:g:n} is obtained by V.~Erlandsson
and J.~Souto in~\cite{Erlandsson:Souto} through an approach different from all
the ones mentioned above.

Despite the fact that our main formula in
Theorem~\ref{th:volume} is obtained by completely different
method, it has, basically, the same structure as
Mirzakhani's formula for $b_{g,n}$. We provide a detailed
comparison of these two formulae in
Section~\ref{s:comparison:with:Mirzakhani}.

Our announcement of the explicit relation between flat and
hyperbolic counts described in the current paper
inspired F.~Arana--Herrera to suggest
in~\cite{Arana:Herrera:square:tiled} an alternative
geometric proof of these results in the spirit of
M.~Mirzakhani.
\end{Remark}

\begin{Remark}
\label{rm:11:and:20}
The pairs $(0,3)$, $(1,1)$ and $(2,0)$ of $(g,n)$ are
exceptional by several reasons, which affect, in
particular, Expression~\eqref{eq:b:g:n} for $b_{g,n}$. The
cases of $(0,3)$ and $(1,1)$ are discussed in
appendix~\ref{ss:Q03:and:Q11} and in
Remark~\ref{rm:Kazarian}; see also the footnotes~1 and~2 on
pages~9 and~12--13 respectively in~\cite{Wright}.

The following observation distinguished the case
$(g,n)=(2,0)$. Any complex curve $C$ in $\cM_2$ admits a
hyperelliptic involution $\tau$. Any holomorphic quadratic
differential $q$ is invariant under this involution,
$\tau^\ast q= q$. Suppose that $q$ belongs to the principal
stratum $\cQ(1^4)$, i.e., suppose that $q$ has four
distinct simple zeroes. Then, the zeroes of $q$ are
arranged into two groups of two zeroes in each group, and
the hyperelliptic involution interchanges the zeroes in
each of the two groups. Thus, there are $4!/2$ ways to
label these $4$ zeroes and not $4!$ ways as suggests the
factor $(4g-4+n)!$ in the general
Formulae~\eqref{eq:Vol:gamma:c:gamma}
and~\eqref{eq:Vol:g:n:b:g:n}.
\end{Remark}

Formula~\eqref{eq:square:tiled:volume} from
Theorem~\ref{th:volume}
allows to compute $\Vol\cQ_{g,n}$ for
all sufficiently small values of $(g,n)$.
Since recently, more efficient methods are now known (see Remark~\ref{rm:Chen:Moeller:Sauvaget}).
Corollary~\ref{cor:Vol:as:b:g:n} thus provides explicit values of $b_{g,n}$
for all such pairs.

When $g=0$ the value $\Vol\cQ_{0,n}$ admits closed
Formula~\eqref{eq:vol:genus:0} obtained
in~\cite{AEZ:genus:0}. Corollary~\ref{cor:Vol:as:b:g:n}
translates this formula into the following explicit
expression for $b_{0,n}$.

\begin{Corollary}
The quantity $b_{0,n}$ defined in~\eqref{eq:b:g:n}
is equal to:
\begin{equation}
b_{0,n}=
\frac{1}{(n-3)!}
\cdot
\left(\frac{\pi}{2}\right)^{2(n-3)}\,.
\end{equation}
\end{Corollary}
By Stirling's formula we get the following asymptotics for large
$n$:
\begin{equation}
b_{0,n}
\sim\frac{1}{\sqrt{2\pi n}}\cdot
\left(\frac{\pi^2 e}{4n}\right)^{n-3}\quad\text{as}\ n\to+\infty\,.
\end{equation}

M.~Mirzakhani notes in~\cite{Mirzakhani:grouth:of:simple:geodesics}
with a reference to~\cite{Mirzakhani:thesis}
that the frequency of simple closed curves of any fixed topological type $\gamma$ can be described in a purely topological way as
$$
\lim_{N\to+\infty}
\frac{\card\big(\{\lambda \text{ an integral multi-curve }| \iota(\lambda,\gamma)\le N\}/\operatorname{Stab}(\gamma)\big)}{N^{6g-6+2n}}
=\tilde{c}(\gamma)\,.
$$
J.~Souto attracted our attention to the fact that
Theorem~\ref{th:our:density:equals:Mirzakhani:density}
implies the following Corollary.

\begin{Corollary}[J.~Souto]
\label{cor:c:tilde}
The frequencies $c(\gamma)$ and $\tilde{c}(\gamma)$ are related by the following scaling factor:
\begin{equation}
\label{eq:c:c:tilda}
c(\gamma)=2^{2g-3+n}\cdot \tilde{c}(\gamma)\,.
\end{equation}
\end{Corollary}
\begin{proof}
The subset of the set
$\{\text{an integral multi-curve }\lambda\, |\, \iota(\lambda,\gamma)\le N\}/\operatorname{Stab}(\gamma)$ for which the pair $(\lambda,\gamma)$ is not filling has cardinality
$o\big(N^{6g-6+2n}\big)$. The complimentary subset for which
the pair $(\lambda,\gamma)$ is filling
are images under the morphism that forgets the labelling
of square-tiled surfaces in $\cST_\Gamma(\cQ(1^{4g-4+n}, -1^n))$
defined in Section~\ref{ss:Square:tiled:surfaces:and:associated:multicurves}.
Hence~\eqref{eq:Vol:gamma} implies that
$$
\Vol(\Gamma,\boldsymbol{H})=2(6g-6+2n)\cdot(4g-4+n)!\cdot 2^{6g-6+2n}\cdot\tilde{c}(\gamma)\,,
$$
where $(\Gamma, \boldsymbol{H})$ is the weighted stable graph corresponding
to $\gamma$. The factor $(4g-4+n)!$ comes from
possible ways to label the $(4g-4+n)$ zeroes
and the factor $2^{6g-6+2n}$ comes from
$$
\card\big(\cST_\gamma(\QuadStrat(1^{4g-4+n}, -1^n), 2N)\big)\sim
2^{6g-6+2n}\cdot\card\big(\cST_\gamma(\QuadStrat(1^{4g-4+n}, -1^n), N)\big)\,.
$$
It remains to apply~\eqref{eq:Vol:gamma:c:gamma}.
\end{proof}

An alternative proof based on the result~\cite{Monin:Telpukhovskiy}
of L.~Monin and I.~Telpukhovskiy is suggested in~\cite{Erlandsson:Souto}.

\subsection{Statistical geometry of square-tiled surfaces}
\label{ss:Combinatorial:geometry:of:random:st:surfaces}

Theorem~\ref{th:volume} provides a detailed description of
the statistical geometric properties of square-tiled surfaces
in $\cQ_{g,n}$ tiled with large number of squares. It has the
same spirit as the result~\cite[Theorem
 1.2]{Mirzakhani:statistics} of M.~Mirzakhani describing
statistics of lengths of simple closed geodesics in random
pants decomposition. More precisely, she fixes a
reduced multicurve $\gamma=\gamma_1+\dots+\gamma_{3g-3}$
decomposing the surface of genus $g$ into pairs of pants and
considers its $\Mod_g$-orbit. For any hyperbolic metric,
she describes the asymptotic distribution of
(normalized) lengths of simple closed geodesics represented
by the components $\gamma_i$, $i=1,\dots,3g-3$, of the
multicurve $\gamma$.

Our result concerns, in particular, the asymptotic
statistics of (normalized) perimeters of a random
square-tiled surface corresponding to a given stable graph
$\Gamma$ tiled with large number of squares. The
resulting statistics disclose the geometric meaning of the
coefficients of the polynomials $P_\Gamma$ associated to a
stable graph $\Gamma$ appearing in our formulae for the
Masur--Veech volumes as in Theorem~\ref{th:volume} and for
the Siegel--Veech constants as in Theorem~\ref{th:carea}.

As we have seen in
Section~\ref{ss:Frequencies:of:square:tiled:surfaces},
asymptotic statistical properties of random square-tiled
surfaces can be translated into asymptotic statistical
properties of geodesic multicurves on random hyperbolic
surfaces and vice versa. This general correspondence
translates the results mentioned above into analogs of a
mean version of Theorem 1.2 in~\cite{Mirzakhani:statistics},
in the sense that we obtain the average of her
statistics over all hyperbolic surfaces in
$\cM_g$, where the average is computed using the
Weil--Petersson measure on $\cM_g$.

Let us define the operator
$\cZt(\boldsymbol{x},\boldsymbol{H})$ on polynomials $\Q[b_1, \ldots, b_k]$ as
follows. We define it on monomials
\[
\cZt(\boldsymbol{x},\boldsymbol{H}) \left( \prod_{i=1}^k b_i^{m_i} \right)
=(m_1+\dots+m_k+k)!\cdot
\prod_{i=1}^k \frac{x_i^{m_i}}{H_i^{m_i+1}}
\]
and extend it by $\Q$-linearity. We denote by
$$
\Delta^k :=
\{(x_1, \ldots, x_k)\,\vert\, x_i \geq 0;\, x_1 + \ldots + x_k \leq
1\}
$$
the standard $k$-dimensional simplex. The operator $\cZt$ generalizes
both $\cY(\boldsymbol{H})$ and $\cZ$
(defined by~\eqref{eq:cV} and~\eqref{eq:cZ} respectively)
in the sense that we can recover
$\cY(\boldsymbol{H})$ and $\cZ$ by integration
\begin{align}
\label{eq:recover:Y}
\cY(\boldsymbol{H}) (P) &= \int_{\Delta_k}
\cZt(\boldsymbol{x},\boldsymbol{H})(P)\, dx_1 \ldots dx_k\,,
\\
\label{eq:recover:Z}
\cZ(P) &= \sum_{H_i \geq 1} \int_{\Delta_k}
\cZt(\boldsymbol{x},\boldsymbol{H})(P)\, dx_1 \ldots dx_k\,.
\end{align}

To any square-tiled surface $S$ in $\cST(\QuadStrat(1^{4g-4+n}, -1^n))$ we associate
the following data
\[
\left(\Gamma, \boldsymbol{H},
\frac{\boldsymbol{b}}{2N}\right)
\in \cG_{g,n} \times \N^k \times\Delta^k\,,
\]
where $\Gamma$ is the stable graph associated to the
horizontal cylinder decomposition,
$k$ is the number of maximal horizontal cylinders
(i.e., number of edges of $\Gamma$),
and
$\boldsymbol{H}=(H_1, \ldots, H_k)$ and
$\boldsymbol{b} = (b_1, \ldots, b_k)$
are respectively the heights and perimeters
of the maximal horizontal cylinders
measured in those units, in which the square of the tiling
has unit sides.
Consider the set
$\cST(\QuadStrat(1^{4g-4+n}, -1^n), 2N)$ of all square-tiled
surfaces in $\QuadStrat(1^{4g-4+n}, -1^n)$
tiled with at most $2N$ squares as in~\eqref{eq:Vol:sq:tiled}.
Recall that $\cST_{\Gamma, \boldsymbol{H}}(\QuadStrat(1^{4g-4+n}, -1^n), 2N)$
denotes the set of square-tiled surfaces associated to the stable graph $\Gamma$ and having the
vector of heights $\boldsymbol{H}$ with at most $2N$ squares.

For each stable graph $\Gamma\in\cG_{g,n}$, each
$\boldsymbol{H}\in\N^k$,
where $k=|E(\Gamma)|$,
and each $N\in\N$, we can define the
following measure $\mu^{\gamma(\Gamma,\boldsymbol{H})}_{g,n,N}$ on the simplex
$\Delta^k$:
\[
\mu^{\gamma(\Gamma,\boldsymbol{H})}_{g,n,N}
:= 2 (6g-6+2n)\cdot \frac{1}{N^d}
\sum_{S \in \cST_{\Gamma,\boldsymbol{H}}(\QuadStrat(1^{4g-4+n}, -1^n), 2N)} \frac{1}{|\Aut(S)|}
\,\delta_{\boldsymbol{b}(S)/(2N)}\,,
\]
where $\delta_{\boldsymbol{b}(S)/(2N)}$ is the Dirac measure concentrated at the point $\frac{\boldsymbol{b}(S)}{2N}\in\Delta^k$.
We can disintegrate the discrete
part $(\Gamma,\boldsymbol{H})$ defining the following measure on $\Delta^k$:
\[
\mu^\Gamma_{g,n,N} = \sum_{\boldsymbol{H}}
\mu^{\gamma(\Gamma,\boldsymbol{H})}_{g,n,N}\,.
\]

Let $g,n$ be non-negative integers with $2 g + n > 3$. Let
$\mu^{\gamma(\Gamma,\boldsymbol{H})}_{g,n,N}$
and $\mu^\Gamma_{g,n,N}$ be the measures defined above.

\begin{Theorem}
\label{thm:statistics}
For each stable graph $\Gamma\in\cG_{g,n}$ and each
$\boldsymbol{H}\in\N^k$ we have weak convergence of
measures:
\begin{equation}
\label{eq:measure:mu:Gamma:H}
\left(\mu^{\gamma(\Gamma,\boldsymbol{H})}_{g,n,N}\right)_N \to
\cZt(\boldsymbol{x},\boldsymbol{H})(P_\Gamma)\, dx
\quad\text{as }N\to+\infty\,.
\end{equation}
Here $P_\Gamma = P_\Gamma(b_1, \ldots, b_k)$
is the global polynomial associated to the stable graph
$\Gamma$ by Formula~\eqref{eq:P:Gamma} and $dx=dx_1\dots dx_k$ is the Lebesgue measure on the simplex $\Delta^k$.

Similarly, we have the weak convergence of measures:
\begin{equation}
\label{eq:measure:mu}
(\mu^\Gamma_{g,n,N})_{N} \to \sum_{\boldsymbol{H}}
\cZt(x,\boldsymbol{H})(P_\Gamma)\,d x\,.
\end{equation}
\end{Theorem}

Theorem~\ref{thm:statistics}
is proved in Section~\ref{ss:Proof:of:the:volume:formula}.

Comparing~\eqref{eq:recover:Y}
and~\eqref{eq:recover:Z}
with~\eqref{eq:volume:contribution:of:stable:graph}
we conclude that the total masses of the
limiting measures are finite and have the following geometric
meaning:
\begin{align*}
\int_{\Delta^k}&
\cZt(x,\boldsymbol{H})(P_\Gamma)\, dx
=\cY(\boldsymbol{H})(P_\Gamma)
=\Vol\big(\Gamma,\boldsymbol{H}\big)
\,,
\\
\int_{\Delta^k}&
\sum_{\boldsymbol{H}}
\cZt(x,\boldsymbol{H})(P_\Gamma)\,d x
=
\sum_{\boldsymbol{H}}
\int_{\Delta^k}
\cZt(x,\boldsymbol{H})(P_\Gamma)\,d x
=\cZ(P_\Gamma)
=\Vol(\Gamma)
\,.
\end{align*}

The above theorem allows us to describe some statistical properties of
random square-tiled surfaces. For example, it allows us to compute
the asymptotic probability that a random square-tiled
surface tiled with a large number of squares corresponds to a
given stable graph $\Gamma$. Considering only square-tiled
surfaces associated to a given stable graph $\Gamma$, we
can compute asymptotic distributions of the heights
$\boldsymbol{H}$ of the maximal horizontal cylinders and
asymptotic distribution of their areas normalized by the
area of the surface. We can also compute asymptotic
statistics of perimeters of the cylinders under appropriate
normalization; for example statistics of the ratios of any
two perimeters. Note, that for the ratios of length
variables, the unit of measurement becomes irrelevant, in
particular,
$$
\frac{H_i}{H_j}=\frac{h_i}{h_j}
\qquad\text{and}\quad
\frac{b_i}{b_j}=\frac{w_i}{w_j}\,.
$$

We will use the notation $\E_\Gamma$
(respectively $\E_{\Gamma,\boldsymbol{H}}$) to
denote the asymptotic expectation values of quantities
evaluated on square-tiled surfaces with given cylinder
decomposition associated to $\Gamma$ (respectively
associated to $\Gamma$ and given heights $\boldsymbol{H}$).
Let us consider several simple examples.

\begin{Example}
Consider the following stable graph $\Phi_1$ in $\cG_{2,0}$
and the associated reduced multicurve:

\includegraphics{genus_two_graph_22.eps}
\includegraphics{genus_two_22.eps}
\begin{picture}(0,0)(-48,1)
\put(-19,-10){$b_1$}
\put(23,-16){$b_2$}
\put(6,-10){$0$}
\put(41,-10){$1$}
\put(-40,-10){$\Phi_1$}
\put(117,-10){$b_1$}
\put(164,-23){$b_2$}
\end{picture}
\vspace*{30pt}

\noindent
From Table~\ref{tab:2:0} we see that
$\Vol(\Phi_1)=\tfrac{1}{675}\cdot\pi^6$ and
$\Vol\cQ_{2,0}=\tfrac{1}{15}\cdot\pi^6$.
Thus, a random square-tiled surface in $\cQ_{2,0}$
(tiled with very large number of squares)
corresponds to the stable graph $\Phi_1$
with (asymptotic) probability
$\frac{\Vol(\Phi_1)}{\Vol\cQ_{2,0}}=\tfrac{1}{45}$.

We have also computed in Table~\ref{tab:2:0}
the polynomial
$P_{\Phi_1} = \tfrac{2}{15} b_1 b_2^3$.
For any given $\boldsymbol{H}= (H_1, H_2)$ we get
\begin{multline*}
\E_{\Phi_1,\boldsymbol{H}}\left(\frac{b_1}{b_2}\right)
\ =\
\frac{\int_{\Delta^2}\cZt(\boldsymbol{x},\boldsymbol{H})\left(\tfrac{b_1}{b_2}P_{\Phi_1}(b_1,b_2)\right)dx_1 dx_2}
{\int_{\Delta^2}\cZt(\boldsymbol{x},\boldsymbol{H})\left(P_{\Phi_1}(b_1,b_2)\right)dx_1 dx_2}
\ =\\=\
\frac{\int_{\Delta^2}\left(\frac{x_1^2 x_2^2}{H_1^3 H_2^3}\right)dx_1 dx_2}
{\int_{\Delta^2}\left(\frac{x_1 x_2^3}{H_1^2 H_2^4}\right)dx_1 dx_2}
\ =\
\frac{2! \cdot 2!}{1! \cdot 3!}\ \frac{H_2}{H_1}
\ =\
\frac{2}{3} \cdot \frac{H_2}{H_1}\,.
\end{multline*}
In other words, if we impose to a square-tiled surface
``of type $\Phi_1$'' to have cylinders of the same height,
then the perimeter $b_2$ of the
``separating''
cylinder is in average slightly longer
than the perimeter $b_1$ of the ``non-separating'' cylinder. However, if
we impose a large height $H_2$ to the ``separating'' cylinder,
its perimeter becomes proportionally shorter in average,
which is quite natural.

What might seem somehow counterintuitive is that
if we do not fix $H$, we obtain
\begin{multline*}
\E_{\Phi_1}\left(\frac{b_1}{b_2}\right)
\ =\
\frac{\int_{\Delta^2}
\sum_{\boldsymbol{H}}
\cZt(\boldsymbol{x},\boldsymbol{H})\left(\tfrac{b_1}{b_2}P_{\Phi_1}(b_1,b_2)\right)\,dx_1 dx_2}
{\int_{\Delta^2}
\sum_{\boldsymbol{H}}
\cZt(\boldsymbol{x},\boldsymbol{H})\left(P_{\Phi_1}(b_1,b_2)\right)\,dx_1 dx_2}
\ =\\=\
\frac{\int_{\Delta^2}
\sum_{\boldsymbol{H}}
\left(\frac{x_1^2 x_2^2}{H_1^3 H_2^3}\right)dx_1 dx_2}
{\int_{\Delta^2}
\sum_{\boldsymbol{H}}
\left(\frac{x_1 x_2^3}{H_1^2 H_2^4}\right)dx_1 dx_2}
\ =\
\frac{
\left(\sum_{\boldsymbol{H}\in\N^2}
\frac{1}{H_1^3 H_2^3}\right)
\cdot
\int_{\Delta^2}
x_1^2 x_2^2\,dx_1 dx_2}
{\left(\sum_{\boldsymbol{H}\in\N^2}
\frac{1}{H_1^2 H_2^4}\right)
\cdot
\int_{\Delta^2}
x_1 x_2^3\,dx_1 dx_2}
\ =\\=\
\frac{\zeta(3)\zeta(3)}{\zeta(2)\zeta(4)}
\cdot
\frac{2!\cdot 2!}{1!\cdot 3!}
\approx
0.811605 \cdot \frac{2}{3}
\approx 0.54107\,,
\end{multline*}
while
\begin{multline*}
\E_{\Phi_1}\left(\frac{b_2}{b_1}\right)
\ =\
\frac{\int_{\Delta^2}
\sum_{\boldsymbol{H}}
\cZt(\boldsymbol{x},\boldsymbol{H})\left(\tfrac{b_2}{b_1}P_{\Phi_1}(b_1,b_2)\right)\,dx_1 dx_2}
{\int_{\Delta^2}
\sum_{\boldsymbol{H}}
\cZt(\boldsymbol{x},\boldsymbol{H})\left(P_{\Phi_1}(b_1,b_2)\right)\,dx_1 dx_2}
\ =\\=\
\frac{\int_{\Delta^2}
\sum_{\boldsymbol{H}}
\left(\frac{x_2^4}{H_1 H_2^5}\right)dx_1 dx_2}
{\int_{\Delta^2}
\sum_{\boldsymbol{H}}
\left(\frac{x_1 x_2^3}{H_1^2 H_2^4}\right)dx_1 dx_2}
\ =\
\frac{
\left(\sum_{\boldsymbol{H}\in\N^2}
\frac{1}{H_1 H_2^5}\right)
\cdot
\int_{\Delta^2}
x_2^4\,dx_1 dx_2}
{\left(\sum_{\boldsymbol{H}\in\N^2}
\frac{1}{H_1^2 H_2^4}\right)
\cdot
\int_{\Delta^2}
x_1 x_2^3\,dx_1 dx_2}
\ =\\=\
\frac{\zeta(1)\zeta(5)}{\zeta(2)\zeta(4)}
\cdot
\frac{0!\cdot 4!}{1!\cdot 3!}
= +\infty\,.
\end{multline*}
\end{Example}

\begin{Example}
Now consider the following graph $\Phi_2$ with two edges:

\includegraphics{genus_two_graph_21.eps}
\begin{picture}(0,0)(-40,0)
\put(1,-19){$b_1$}
\put(62,-19){$b_2$}
\put(33,-30){$0$}
\put(-30,-19){$\Phi_2$}
\end{picture}

\includegraphics{genus_two_21.eps}
\begin{picture}(0,0)(-189,9)
\put(-24,0){$b_1$}
\put(66,0){$b_2$}
\end{picture}
\vspace*{30pt}

\noindent
It was computed in Table~\ref{tab:2:0} that
$P_{\Phi_2} = \tfrac{8}{5} (b_1^3 b_2 + b_1 b_2^3)$. Then
for any given $\boldsymbol{H} = (H_1, H_2)$ we have
\[
\E_{\Phi_2,\boldsymbol{H}}\left(\frac{b_1}{b_2}\right) =
\frac{\frac{4!\cdot 0!}{H_1^5 H_2} + \frac{2!\cdot 2!}{H_1^3 H_2^3}}
{\frac{3!\cdot 1!}{H_1^4 H_2^2} + \frac{1!\cdot 3!}{H_1^2 H_2^4}}
=
\frac{2 H_2 \left(H_1^2+6 H_2^2\right)}{3 H_1
   \left(H_1^2+H_2^2\right)}\,.
\]
In particular, in the symmetric case, when $H_1=H_2$, we have
$$
\E_{\Phi_2(2),\boldsymbol{H}}\left(\frac{b_1}{b_2}\right) =
\E_{\Gamma,\boldsymbol{H}}\left(\frac{b_2}{b_1}\right) =
\frac{7}{3}\,.
$$
\end{Example}

\begin{Example}
Note also, that we get for free the averaged version of
\cite[Theorem~1.2]{Mirzakhani:statistics}. Namely, when the
stable graph $\Gamma\in\cG_{g,0}$ has the maximal possible
number $3g-3$ of vertices (i.e., when the corresponding
multicurve provides a pants decomposition of the surface),
Equation~\eqref{eq:P:Gamma} for $P_\Gamma$
takes the following form:
$$
P_\Gamma(\boldsymbol{b})
=(\text{numerical factor})\cdot
b_1\dots b_{3g-3}\cdot\prod_{v\in V(\Gamma)}
N_{0,3}(\boldsymbol{b}_v)\,.
$$
Since $N_{0,3}=1$ identically, we conclude that for
$H_1=H_2=\dots=H_{3g-3}$ the density function of
statistics of the normalized
lengths is the product $x_1\cdot\dots\cdot x_{3g-3}$ up
to a constant normalization factor depending only on the genus $g$
\begin{multline*}
\mu^{\gamma(\Gamma,(1,\dots,1))}_{g,n,N} \to
\cZt(\boldsymbol{x},(1,\dots,1))(C_1(g)\cdot b_1\dots b_{3g-3})\,dx
=\\=C_2(g)\cdot x_1\cdot\dots\cdot x_{3g-3}\, dx_1\dots dx_{3g-3}
\quad\text{as }N\to+\infty\,.
\end{multline*}

Mirzakhani proved in
\cite[Theorem~1.2]{Mirzakhani:statistics} that the same
asymptotic length statistics is valid for any individual
hyperbolic surface in $\cM_g$ (and not only in average, as
we do). F.~Arana~Herrera and M.~Liu
independently proved in~\cite{Arana:Herrera:Equidistribution:of:horospheres}, \cite{Arana:Herrera:Counting:multi:geodesics}
and in~\cite{Mingkun}
a generalization of this result
to arbitrary multicurves. Namely, for any stable graph $\Gamma\in\cG_{g,n}$,
any associated collection of positive integer weights $\boldsymbol{H}$ and any hyperbolic surface $X\in\cM_{g,n}$, the asymptotic statistics of normalized lengths of components of hyperbolic geodesic multicurves in $\Mod_{g,n}\cdot \gamma(\Gamma,\boldsymbol{H})$
coincides
(up to a global normalization constant depending only on $g$ and $n$) with
$\cZt(\boldsymbol{x},\boldsymbol{H})(P_\Gamma)\, dx$. In particular, it does not depend on the hyperbolic metric $X$.
\end{Example}

We complete this section considering two examples
describing statistics of heights of the maximal
horizontal cylinders of a random square-tiled surfaces.
We start with the following elementary lemma.

\begin{Lemma}
\label{lm:height:of:one:cylinder}
Consider a random square-tiled surface in $\cQ_{g,n}$
having a single maximal horizontal cylinder. The asymptotic
probability that this cylinder is represented by a single band
of squares (i.e. that $H_1=1$) equals
$\cfrac{1}{\zeta(6g-6+2n)}$.
\end{Lemma}
\begin{proof}
When a stable graph
$\Gamma\in\cG_{g,n}$ has a single edge $b_1$,
Formula~\eqref{eq:P:Gamma} gives
$P_\Gamma=(\text{numerical factor})
\cdot b_1^{6g-7+2n}$,
and the Lemma follows.
\end{proof}

In terms of multicurves this means that a random
single-component integral multicurve $n\gamma$ in $\cML(\Z)$ (where
$n\in\N$ and $\gamma$ is a simple closed curve) is reduced
(i.e. $n=1$) with asymptotic probability
$\frac{1}{\zeta(6g-6+2n)}$.

Note that $\zeta(x)$ tends to $1$ exponentially rapidly as
the real-valued argument $x$ grows. Thus, our result
implies, that when at least one of $g$ or $n$ is large
enough, a random one-cylinder square-tiled surface is tiled
with a single horizontal band of squares with a very large
probability, and a random single-component integral
multicurve is just a simple closed curve with
a very large probability.

The polynomial $P_\Gamma$ enables to compute analogous
probabilities for any given stable graph $\Gamma$. For
example, a random square-tiled surface in $\cQ_2$
associated to the stable graph $\Phi_1$, considered earlier
in this section, has both cylinders of height $H_1=H_2=1$
with probability
$$
\frac{\cY(1,1)(P_{\Phi_1})}{\cZ(P_{\Phi_1})}
=\frac{1}{\zeta(2)\zeta(4)}
=\frac{540}{\pi^6}\approx 0.561687\,.
$$
A random square-tiled surface in $\cQ_2$ associated to the
stable graph $\Phi_2$, considered earlier in this
section, has heights $H_i$ of both horizontal cylinders
bounded by $2$ with probability
\begin{multline*}
\frac{\cY(1,1)(P_{\Phi_2})
+\cY(1,2)(P_{\Phi_2})
+\cY(2,1)(P_{\Phi_2})
+\cY(2,2)(P_{\Phi_2})}{\cZ(P_{\Phi_2})}
\ =\\=\
\frac{2+\frac{2}{16}+\frac{2}{4}+\frac{2}{64}}{2\zeta(2)\zeta(4)}
=\frac{\frac{85}{64}}{\zeta(2)\zeta(4)}
\approx 0.745991\,.
\end{multline*}

\subsection{Structure of the paper}

In Section~\ref{s:proofs} we prove Theorem~\ref{th:volume}
stated in Section~\ref{ss:intro:Masur:Veech:volumes}
providing the formula for the Masur--Veech volume
$\Vol\cQ_{g,n}$ and Theorem~\ref{th:carea} stated in
Section~\ref{ss:intro:Siegel:Veech:constants} providing the
formula for the area Siegel--Veech constant
$\carea(\cQ_{g,n})$.

In Section~\ref{s:comparison:with:Mirzakhani} we compare
our Formula~\ref{eq:square:tiled:volume} for
$\Vol\cQ_{g,n}$ with Mirzakhani's formula for $b_{g,n}$ and
our Formula~\eqref{eq:volume:contribution:of:stable:graph}
for $\Vol(\Gamma)$ for a stable graph $\Gamma$ with
Mirzakhani's formula for the associated $c(\gamma)$ for the
associated multicurve $\gamma$. We elaborate the translation
between the two languages and prove
Theorem~\ref{th:our:density:equals:Mirzakhani:density}
stated in
Section~\ref{ss:Frequencies:of:square:tiled:surfaces}
evaluating the normalization constant~\ref{eq:Vol:gamma:c:gamma}
between the corresponding quantities.

In Section~\ref{s:2:correlators} we state a uniform
asymptotic Formula~\eqref{eq:a:g:k:difference} for
correlators $\langle\tau_{d_1}\tau_{d_2}\rangle$, which
has independent interest. We apply it to the computation of
the asymptotic frequencies $c_{g,sep}$ and
$c_{g,nonsep}$ of separating and of non-separating
simple closed hyperbolic geodesics on a hyperbolic surface
of large genus $g$ thus proving
Theorem~\ref{th:separating:over:non:separating} stated in
Section~\ref{ss:Frequencies:of:simple:closed:curves}.

In Appendix~\ref{a:proof:2:correlators} we present the proof of the asymptotic
Formula~\eqref{eq:a:g:k:difference} used in Section~\ref{s:2:correlators}. This
proof consists of a sequence of combinatorial manipulations with expressions in
binomial coefficients.
For the sake of completeness, we present a detailed formal
definition of a stable graph in
Appendix~\ref{s:stable:graphs}.
Appendix~\ref{s:explicit:calculations} provides examples of
explicit calculations of the Masur--Veech volume
$\Vol\cQ_{g,n}$ and of the Siegel--Veech constant
$\carea(\cQ_{g,n})$ for small $g$ and $n$.
Appendix~\ref{a:tables} presents tables of
$\Vol\cQ_{g,n}$ and of $\carea(\cQ_{g,n})$ for small $g$
and $n$.

\subsection{Acknowledgements}
Numerous results of this paper were directly or indirectly
inspired by beautiful ideas of Maryam~Mirzakhani. Working
on this paper we had a constant feeling that we are
following her steps. This concerns in particular the
relation between Masur--Veech volume $\Vol\cQ_{g,n}$ and
frequencies of hyperbolic multicurves and relation between
large genus asymptotics of the volumes of moduli spaces and
intersection numbers of $\psi$-classes.

A.~Eskin planned to use Kontsevich formula for computation
of Masur--Veech volumes before the invention of the
approach of
Eskin--Okounkov~\cite{Eskin:Okounkov:Inventiones}. We thank
him for useful conversations and for indicating to us that
our technique of evaluation of the Masur--Veech volumes
admits generalization to strata with only odd zeroes, see
Remark~\ref{rm:to:complete}.

We are grateful to M.~Kazarian for his
computer code evaluating intersection numbers which we used
on the early stage of this project.
We are particularly grateful to him for the recursion for linear Hodge integrals
obtained in~\cite{Kazarian}.
Combination of recent results from~\cite{Chen:Moeller:Sauvaget}
and from~\cite{Kazarian} confirmed all numerical values
of the Masur--Veech volumes $\Vol\cQ_{g,n}$ and of the Siegel--Veech constants
$\carea(\cQ_{g,n})$ which appear in the current paper and provided very serious
numerical evidence for the Conjecture~\ref{conj:Vol:Qg} (stated prior to these results).

We highly appreciate the computer experiments of M.~Bell~\cite{Bell} which
provided computer evidence independent of theoretical
predictions of frequencies of multicurves on surfaces of
low genera.

We are extremely grateful to A.~Aggarwal and to J.~Athreya
who carefully read the manuscripts of the current paper and
made numerous helpful remarks.

We thank F.~Arana--Herrera, S.~Barazer, A.~Borodin, C.~Ball, G.~Borot, D.~Chen, E.~Duryev, M.~Liu, M.~M\"oller, L.~Monin, B.~Petri, K.~Rafi, A.~Sauvaget, J.~Souto, I.~Telpukhovky, S.~Wolpert, A.~Wright, D.~Zagier for useful discussions.

We are grateful to MPIM in Bonn, where a considerable part of
this work was performed,
and to MSRI in Berkeley for providing us with friendly and
stimulating environment.

We thank anonymous referees for attentive reading of
the manuscript and for important suggestions which helped us
to improve the presentation.

\section{Proofs of the formulae for the Masur--Veech volumes and for the Siegel--Veech constants}
\label{s:proofs}

We start this section by recalling the necessary background
and normalization conventions that are used in the
subsequent sections of the paper.

\subsection{The principal stratum and Masur--Veech measure}
\label{ss:background:strata}
In this section we recall the canonical construction of the
Masur--Veech measure on $\QuadStrat(1^{4g-4+n}, -1^n)$ and its link with the
integral structure given by square-tiled surfaces.

Consider a compact nonsingular complex curve $C$ of genus
$g$ endowed with a meromorphic quadratic differential $q$
with $\zeroes=4g-4+n$ simple zeroes and with $n$ simple
poles. Any such pair $(C,q)$ defines a canonical ramified
double cover $\pi:\hat C\to C$ such that $\pi^\ast
q=\hat\omega^2$, where $\hat\omega$ is an Abelian
differential $\widehat\omega$ on the double cover $\widehat
C$. The ramification points of $\pi$ are exactly the zeroes
and poles of $q$. The double cover $\widehat C$ is endowed
with the canonical involution $\iota$ interchanging the two
preimages of every regular point of the cover. The stratum
$\QuadStrat(1^{\zeroes},-1^n)$ of such differentials is modelled
on the subspace of the relative cohomology of the double
cover $\widehat C$, anti-invariant with respect to the
involution $\iota$. This anti-invariant subspace is denoted
by $H^1_-(\widehat C,\{\widehat P_1,\dots,\widehat
P_{\zeroes}\};\C)$, where $\{\widehat P_1,\dots,\widehat
P_{\zeroes}\}$ are zeroes of the induced Abelian
differential $\widehat\omega$. For $(g,n)$ with $2g+n>3$,
the image of the stratum $\QuadStrat(1^{\zeroes}, -1^n)$ is
open and dense in $\cQ_{g,n}$ (its complement in $\cQ_{g,n}$
is closed and of positive codimension). In what follows we
always work with the principal stratum.

We define a lattice in $H^1_-(\widehat
S,\{\widehat{P}_1,\ldots,\widehat{P}_{\zeroes}\};\C)$ as
the subset of those linear forms which take values in
$\Z\oplus i\Z$ on $H^-_1(\widehat S,\{\widehat
P_1,\dots,\widehat P_{\zeroes}\};\Z)$. The integer points
in $\cQ_{g,n}$ are exactly those quadratic differentials
for which the associated flat surface with the metric $|q|$
can be tiled with $1/2 \times 1/2$ squares. In this way the
integer points in $\QuadStrat(1^{\zeroes}, -1^n)$ are represented by
\textit{square-tiled surface} as defined in
Section~\ref{ss:Square:tiled:surfaces:and:associated:multicurves}.

We define the Masur--Veech volume
element $d\!\Vol$ on $\QuadStrat(1^{\zeroes},-1^n)$ as the linear volume
element in the vector space
$H^1_-(S,\{\widehat{P}_1,\ldots,\widehat{P}_{\zeroes}\};\C{})$ normalized in such a
way that the fundamental domain of the above lattice has unit
volume. The Masur--Veech volume element $\dVolMV$ in $\QuadStrat(1^{\zeroes}, -1^n)$
induces a volume element on the level sets of the $\Area$ function. In particular
on the level hypersurface $\cQ^{\Area=\frac{1}{2}}_{g,n}$ we get
\begin{equation}
\label{eq:def:Vol:Q:g:n}
\Vol \QuadStrat(1^{\zeroes}, -1^n)
:= \Vol_1 \QuadStrat^{\Area=\frac{1}{2}}(1^{\zeroes}, -1^n)
= 2 \cdot d \cdot \Vol \QuadStrat^{\Area\le\frac{1}{2}}(1^{\zeroes}, -1^n)\,.
\end{equation}
Here $d = 6g-6+2n = \dim_{\C}\cQ_{g,n}$.
The right-hand side term in the above formula can be considered as the definition of the middle term and of the left-hand-side term.

By construction, the volume element $d\!\Vol^{symplectic}$
in $\cQ_{g,n}$ induced by the canonical symplectic
structure considered in Section~\ref{ss:MV:volume} and the
linear volume element $d\!\Vol^{period}$ in period
coordinates defined in this section
belong to the same Lebesgue measure class. It was proved by
H.~Masur in~\cite{Masur:Hamiltonian} that the Teichm\"uller
flow is Hamiltonian, in particular, that
$d\!\Vol^{symplectic}$ is preserved by the
Teichm\"uller flow. By the results of
H.~Masur~\cite{Masur:82} and
W.~Veech~\cite{Veech:Gauss:measures}, the volume element
$d\!\Vol^{period}$ is also preserved by the Teichm\"uller
flow. Ergodicity of the Teichm\"uller flow now implies that
the two volume forms are pointwise proportional with
constant coefficient.

We postpone evaluation of this constant factor to another
paper. Throughout this paper we consider the normalization
of the Masur--Veech volume element
$d\!\Vol=d\!\Vol^{period}$ as defined in the current
section and then define $\Vol\cQ_{g,n} = \Vol \QuadStrat(1^{\zeroes}, -1^n)$ by
means of~\eqref{eq:def:Vol:Q:g:n}. This definition
incorporates the conventions on the choice of the lattice,
on the choice of the level of the area function, and the
convention on the dimensional constant. We
follow~\cite{AEZ:Dedicata}, \cite{AEZ:genus:0},
\cite{Goujard:volumes},
\cite{DGZZ:meanders:and:equidistribution},
\cite{DGZZ:one:cylinder:Yoccoz:volume} in the choice of
these conventions; see Section~4.1 in~\cite{AEZ:genus:0} and
Appendix~A in~\cite{DGZZ:meanders:and:equidistribution} for
the arguments in favor of this normalization.

\subsection{Jenkins--Strebel differentials and stable graphs}
\label{ss:Jenkins:Strebel}
A quadratic differential $q$ in $\cQ_{g,n}$ is called
\textit{Jenkins-Strebel} if its horizontal foliation
contains only closed leaves. Any Jenkins--Strebel
differential can be decomposed into maximal horizontal
cylinders with zeroes and simple poles located on the
boundaries of these cylinders. We  call these boundaries
\textit{singular layers}. Each singular layer defines a
metric ribbon graph representing an oriented surface with
boundary. When the quadratic differential belongs to the
principal stratum $\cQ(1^{4g-4+n},-1^n)$, the ribbon graph has
vertices of valence three at simple zeroes of $q$, vertices
of valence one at simple poles of $q$ and no vertices of
any other valency. Throughout this paper we always assume
that the quadratic differential $q$ belongs to the
principal stratum.

Every ribbon graph $\ribbongraph$ considered as an oriented
surface with boundary has a certain genus
$g(\ribbongraph)$, number $n(\ribbongraph)$ of
boundary components, and number $p(\ribbongraph)$
of univalent vertices often called \textit{leaves} of the
graph. The number $m(\ribbongraph)$ of trivalent vertices
can be expressed through these quantities as
\begin{equation}
\label{eq:m:g:p:n}
m(\ribbongraph)=4g(\ribbongraph)-4+2n(\ribbongraph)+p(\ribbongraph)
\end{equation}

\begin{figure}[htb]
  %
\includegraphics{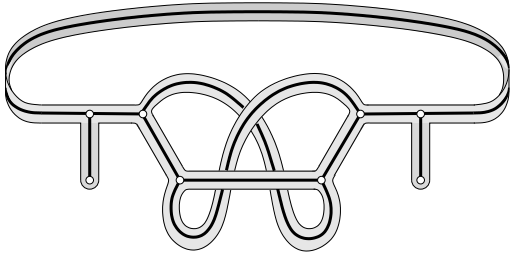}
\vspace{85pt}
\caption{
\label{fig:ribbon:graph}
This ribbon graph $\ribbongraph$ has
genus $g(\ribbongraph)=1$; it has $n(\ribbongraph)=2$ boundary components
and $p(\ribbongraph)=2$ univalent vertices.
}
\end{figure}

To every Jenkins--Strebel differential $q$ as above we
associate a \textit{stable graph} $\Graph =(V, H, \alpha,
\iota, \boldsymbol{g}, L)$ in the same way as we did it in
Section~\ref{ss:Square:tiled:surfaces:and:associated:multicurves}
in the particular case when Jenkins--Strebel differential
$q$ represents a square-tiled surface. (A formal definition
of a stable graph can be found, for example,
in~\cite{Okounkov:Pandharipande}; we reproduce it in
Appendix~\ref{s:stable:graphs} for completeness.) We recall
briefly the construction of $\Gamma$.

The vertices of $\Graph$ encode the singular layers. The
set of all vertices (singular layers) is denoted by $V$. A
tubular neighborhood of a singular layer $v$ is a surface
with boundary, which has some genus $g_v$. The
\textit{genus decoration} $\boldsymbol{g} = (g_v)_v$
associates to each $v\in V$ the non-negative integer $g_v$.
Any maximal horizontal cylinder of $q$ has two boundary
components which are canonically identified with
appropriate boundary components of tubular neighborhoods of
appropriate singular layers $v_i, v_j$ (where $v_i$ and
$v_j$ coincide when the cylinder goes from the singular layer
to itself). In this way, each maximal horizontal
cylinder defines an edge of $\Graph$ joining the
boundary layers $v_i, v_j$. Finally, simple poles of
$q$ are encoded by the \textit{legs} of $\Graph$. By
convention the $n$ simple poles are labeled, so the legs of
$\Gamma$ inherit the labeling $L$.
Relation~\eqref{eq:m:g:p:n}
implies the stability condition $2 g_v - 2 + n_v > 0$ for
every vertex $v$ of $\Graph$.

\begin{Remark}
Consider the following trivial stable graph: it has a
unique vertex decorated by the integer $g$; it has $n$
legs; it has no edges. Such graph does not correspond to
any (non-zero) Strebel differential. For $(g,n) \not= (0,3)$
it provides zero contribution in
Formulae~\eqref{eq:square:tiled:volume}
for $\Vol \cQ_{g,n}$ and~\eqref{eq:carea} for $\carea(\cQ_{g,n})$.
\end{Remark}

\subsection{Conditions on the lengths of the waist curves of the cylinders}
\label{eq:conditions:on:lengths}

Having a square-tiled surface and its associated stable
graph $\Graph$, we denote by $k=|E(\Graph)|$ the number of
maximal cylinders filled with closed horizontal
trajectories. Denote by $w_1, \dots, w_k$ the lengths of
the waist curves of these cylinders. Since every edge of
any singular layer $v$ is followed by the boundary of the
corresponding ribbon graph twice, the sum of the lengths of
all boundary components of each singular layer $V$ is
integral (and not only half-integral).

Let $\Graph$ be a stable graph and let us consider the collection of linear
forms $l_v=\sum_{e\in E_v(\Graph)}w_e$ in variables $w_1,\dots, w_k$, where
$k=|E(\Graph)|$, $v$ runs over the vertices $V(\Graph)$, and $E_v(\Graph)$ is
the set of edges adjacent to the vertex $v$ (ignoring legs).
It is immediate to see that the $(\Z/2\Z)$-vector space
spanned by all such linear forms has dimension $|V(\Graph)|-1$.

Let us make a change of variables passing from half-integer
to integer parameters
$b_i:=2w_i$ where $i=1,\dots,k$. Consider the integer sublattice
$\mathbb{L}_\Graph\subset\Z^k$ defined by the linear relations
\begin{equation}
\label{eq:sublattice:L}
l_{v}(b_1,\dots,b_k)=\sum_{e\in E_v(\Graph)}b_{e}=0\ (\operatorname{mod} 2)
\end{equation}
for all vertices $v\in V(\Graph)$. By the above
remark, the sublattice $\mathbb{L}_\Graph$ has index $2^{|V(\Graph)|-1}$
in $\Z^k$. We summarize the observations of this section in the following
criterion.

\begin{Corollary}
\label{cor:criterion}
A collection of strictly positive numbers
$w_1,\dots,w_{k}$,
where $w_i\in\tfrac{1}{2}\N$ for $i=1,\dots,k$,
corresponds to a square-tiled surface
realized by a stable graph $\Graph\in\cG_{g,n}$ if and only if $k=|E(\Graph)|$ and the
corresponding vector $\boldsymbol{b}=2\boldsymbol{w}$ belongs to the sublattice
$\mathbb{L}_\Graph$. This sublattice has index
$|\Z^{k}:\mathbb{L}_\Graph|=2^{|V(\Graph)|-1}$
in the integer lattice $\Z^k$.
\end{Corollary}

We complete this section with a generalization
of Lemma~3.7 in~\cite{AEZ:Dedicata} which would be used in
the proof of our main
Formula~\eqref{eq:volume:contribution:of:stable:graph} for
the Masur--Veech volume $\Vol\cQ_{g,n}$.

\begin{Lemma}
\label{lm:evaluation:for:monomial}
Let $\mathbb{L}$ be a sublattice of finite index
$|\Z^k:\mathbb{L}|$ in the integer lattice $\Z^k$
and let $m_1,\dots,m_k\in\N$ be any positive integers.
The following formula holds
\begin{multline}
\label{eq:3.7}
\lim_{N\to+\infty}
\frac{1}{N^{|m|+k}}
\sum_{\substack{\boldsymbol{b}\cdot\boldsymbol{H}\le N\\b_i, H_i\in\N\\ \boldsymbol{b}\in\mathbb{L}}}
b_1^{m_1}\cdots b_k^{m_k}
\ =\\=\
\frac{1}{|\Z^k:\mathbb{L}|}\cdot
\frac{1}{(|m|+k)!}
\cdot
\prod_{i=1}^k \Bigg(m_i!\cdot \zeta(m_i+1)\Bigg)
\ =\\=\
\frac{1}{|\Z^k:\mathbb{L}|}\cdot
\frac{1}{(|m|+k)!}
\cdot
\cZ(b_1^{m_1}\cdots b_k^{m_k})
\,.
\end{multline}

Moreover, the sum in $\boldsymbol{H}$ and the limit commute:
$$
\lim_{N\to+\infty}
\frac{1}{N^{|m|+k}}
\sum_{\substack{\boldsymbol{b}\cdot\boldsymbol{H}\le N\\b_i, H_i\in\N\\ \boldsymbol{b}\in\mathbb{L}}}
b_1^{m_1}\cdots b_k^{m_k}
\ =\
\sum_{\substack{\boldsymbol{H}\\H_i \in \N}}
\lim_{N\to+\infty}
\frac{1}{N^{|m|+k}}
\sum_{\substack{\boldsymbol{b}\cdot\boldsymbol{H}\le N\\b_i \in\N\\ \boldsymbol{b}\in\mathbb{L}}}
b_1^{m_1}\cdots b_k^{m_k}
$$
and we have
$$
\lim_{N\to+\infty}
\frac{1}{N^{|m|+k}}
\sum_{\substack{\boldsymbol{b}\cdot\boldsymbol{H}\le N\\b_i \in\N\\ \boldsymbol{b}\in\mathbb{L}}}
b_1^{m_1}\cdots b_k^{m_k}
\ =\
\frac{1}{|\Z^k:\mathbb{L}|}\cdot
\frac{1}{(|m|+k)!}
\cdot
\cY(\boldsymbol{H})(b_1^{m_1}\cdots b_k^{m_k})\,.
$$
\end{Lemma}
\begin{proof}
The limit with omitted restriction
$\boldsymbol{b}\in\mathbb{L}$ is computed in Lemma~3.7
in the paper~\cite{AEZ:Dedicata}. Note that the inversion of sum and
limits here is valid by virtue of the dominated convergence
theorem. More precisely, under the substitution $x_i =
\frac{b_i H_i}{N}$ the sum approximates the integral
from below.

The restriction $\boldsymbol{b}\in\mathbb{L}$ rescales
the volume element in the corresponding integral sum by the
index $|\Z^k:\mathbb{L}|$ of the sublattice $\mathbb{L}$
in $\Z^k$ which
produces the extra factor $|\Z^k:\mathbb{L}|^{-1}$.

Finally, the last equality in~\eqref{eq:3.7}
and the last equality in the last line
of the assertion of Lemma~\ref{lm:evaluation:for:monomial}
are just the
Definitions~\eqref{eq:cZ} and~\eqref{eq:cV} of $\cZ$
and of $\cY$ respectively.
\end{proof}

\subsection{Counting trivalent metric ribbon graphs with leaves}
\label{ssec:trivalent:graphs:with:leaves}
We need the following elementary generalization of the
theorem of M.~Kontsevich stated in
section~\ref{ss:intro:volume:polynomials} allowing to our
metric ribbon graph have univalent vertices
(\textit{leaves}) in addition to trivalent vertices.

We use the letter $p$ to denote the number of leaves.
Consider a collection of positive integers $b_1,\dots, b_n$
such that $\sum_{i=1}^n b_i$ is even. Similarly to
$\cN_{g,n}(b_1, \ldots, b_n)$ defined in
section~\ref{ss:intro:volume:polynomials}, let us denote by
$\cN_{g,n,p}(b_1,\dots,b_n)$ the weighted count of
connected metric ribbon graphs $G$  of genus $g$ with $n$
labeled boundary components of integer lengths
$b_1,\dots,b_n$ and $p$ univalent vertices. In other words
$$
\cN_{g,n,p}(b_1,\dots,b_n):=\sum_{G \in \cR_{g,n,p}} \frac{1}{|\Aut(G)|} N_G(b_1,\dots,b_n)\,,
$$
where $\cR_{g,n,p}$ is the set of equivalence classes of
ribbon graphs of genus $g$, with $n$ boundary components
and $p$ univalent vertices. The counting function
$\cN_{g,n,p}$ generalizes Kontsevich's polynomials.

\begin{Proposition}
\label{pr:count:with:leaves}
Consider $n$-tuples of large positive integers
$b_1,\dots,b_n$ such that
$\sum_{i=1}^n b_i$ is even. The following relation holds:
\begin{align}
\cN_{g,n,p}(b_1,\dots,b_n)
&= \cN_{g,n+p,0}(b_1,\dots,b_n,\underbrace{0,\dots,0}_p) \\
&= N_{g,n+p}(b_1,\dots,b_n,\underbrace{0,\dots,0}_p)
+\text{lower order terms}
\,,
\end{align}
where the Kontsevich polynomials $N_{g,n}$ are defined by Formula~\eqref{eq:N:g:n}.
\end{Proposition}

The proof of Proposition~\ref{pr:count:with:leaves}
is the combination of the following two Lemmas.

\begin{NNLemma}
   %
Suppose that for some $g,n,p$ the leading term of
$\cN_{g,n,p}(b_1,\dots,b_n)$ is a homogeneous polynomial
$N^{top}_{g,n,\poles}(b_1,\dots, b_n)$ in $b_1,\dots,b_n$.
Then the leading term of $\cN_{g,n,\poles+1}(b_1,\dots,
b_n)$ is also a homogeneous polynomial
$N^{top}_{g,n,\poles+1}(b_1,\dots, b_n)$ in
$b_1,\dots,b_n$. Moreover, it satisfies the relation
\begin{equation}
\label{eq:g:m:k}
N^{top}_{g,n,\poles+1}=\frac{1}{2}\cdot\cI(N^{top}_{g,n,\poles})\,,
\end{equation}
where  $\cI=\sum_{i=1}^n  \cI_{b_i}$, and the operators  $\cI_{b_i}$
are defined on monomials  by
$$
\cI_{b_i}(b_1^{j_1}\dots b_{i-1}^{j_{i-1}}
\cdot b_i^{j_i}\cdot
b_{i+1}^{j_{i+1}}  \dots b_n^{j_n})
=
b_1^{j_1}\dots b_{i-1}^{j_{i-1}} \cdot
\frac{b_i^{j_i+2}}{j_i+2}\cdot
b_{i+1}^{j_{i+1}}  \dots b_n^{j_n}
$$
and  are  extended  to arbitrary polynomials by linearity.
\end{NNLemma}
\begin{proof}
This Lemma mimics Lemma 3.5 in~\cite{AEZ:Dedicata}. Formally
speaking, in~\cite{AEZ:Dedicata} the corresponding statement
is formulated only for $g=0$, but it is immediate to see that the
inductive proof is applicable without any changes to any genus as
soon as the base of induction, corresponding to $\poles=0$ is valid.
\end{proof}

\begin{NNLemma}
The polynomials $N_{g,n}$
defined by Equation~\eqref{eq:N:g:n} satisfy the relations:
\begin{equation}
\label{eq:g:m:plus:k}
N_{g,n+\poles+1}(b_1,\dots,b_n,0,\dots,0)=
\frac{1}{2}\cdot\cI(N_{g,n+\poles}(b_1,\dots,b_n,0,\dots,0))\,.
\end{equation}
\end{NNLemma}
\begin{proof}
Using the explicit expressions of $N_{g,n+\poles+1}$ and of
$N_{g,n+\poles}$ in terms of $\psi$-classes
we see that the above
relation is equivalent to the following one:
\begin{multline*}
\langle \psi_1^{d_1}\psi_2^{d_2}\dots \psi_n^{d_n} \rangle_{g,n+p+1}
=
\langle \psi_1^{d_1-1}\psi_2^{d_2}\dots \psi_n^{d_n} \rangle_{g,n+p}
+\\+
\langle \psi_1^{d_1}\psi_2^{d_2-1}\dots \psi_n^{d_n} \rangle_{g,n+p}
+\dots+
\langle \psi_1^{d_1}\psi_2^{d_2}\dots \psi_n^{d_n-1} \rangle_{g,n+p}\,,
\end{multline*}
which is the well-known identity for the intersection
numbers known as ``string equation'', see~\cite{Witten}.
\end{proof}

\begin{proof}[Proof of Proposition~\ref{pr:count:with:leaves}]
For $\poles=0$ (when there are no poles at all) the statement
corresponds to the original theorem of Kontsevich stated in
Section~\ref{ss:intro:volume:polynomials}. We use this as the base of
induction in $p$ for any fixed pair $(g,n)$. It remains to notice that
Equations~\eqref{eq:g:m:k} and~\eqref{eq:g:m:plus:k} recursively
define the corresponding polynomials for any $(g,n,p)$ starting from
$(g,n,0)$.
\end{proof}

\subsection{Proof of the volume formula}
\label{ss:Proof:of:the:volume:formula}

A square-tiled surface corresponding to a fixed stable graph
$\Graph$ can be described by three groups of
parameters. The parameters in different groups can
be varied independently. The parameters in the first group are
responsible for the lengths of horizontal saddle connections. In this
group we fix only the lengths $w_1, \dots, w_k$ of the waist curves
of the cylinders filled with closed horizontal trajectories, where $k$
is the number of edges in $\Graph$. This
leaves a certain freedom for the choice of the lengths of horizontal
saddle connections. The criterion of admissibility of a given
collection $\boldsymbol{w}=(w_1,\dots,w_k)$ is given by
Corollary~\ref{cor:criterion}. The count for the number of choices of
the lengths of all individual saddle connections for a fixed choice
of $\boldsymbol{w}$ is given in
Proposition~\ref{pr:count:with:leaves}.

There are no restrictions on the choice of strictly positives integer
or half-integer heights $h_1,\dots, h_k$ of the cylinders.

Having chosen the widths $w_1, \dots, w_k$ of all maximal cylinders
(i.e. the lengths of the closed horizontal trajectories) and the
heights $h_1, \dots, h_k$ of the cylinders, the flat area of the
entire surface is already uniquely determined as the sum
$\boldsymbol{w}\cdot\boldsymbol{h}=w_1 h_1+\dots +w_k h_k$ of flat areas of
individual cylinders.

However, when the lengths of all horizontal saddle connections and
the heights $h_i$ of all cylinders are fixed, there is still a freedom
in the third independent group of parameters. Namely, we can twist
each cylinder by some twist $\phi_i\in\frac{1}{2}\N$ before attaching
it to the layer. Applying, if necessary, appropriate Dehn twist we
can assume that $0\le\phi_i<w_i$, where $w_i$ is the perimeter
(length of the waist curve) of the corresponding cylinder. Thus, for
any choice of lengths of horizontal saddle connections realizing some
square-tiled surface with the stable graph $\Graph$ and for any
choice $h_1,\dots, h_k$ of heights of the cylinders we get
$(2w_1)\cdot\ldots\cdot(2w_k)$ square-tiled surfaces sharing the same
lengths of the horizontal saddle connections and same heights of the
cylinders.

In Proposition~\ref{pr:count:with:leaves} we assume that the lengths
of the edges of the metric ribbon graph are integer. Clearly, if we
allow these lengths to be also half-integer, we get
$N_{g,n+p}\big(2w_1,\dots,2w_n,0,\dots,0\big)$ as the leading
term of the new count. The realizability condition
$2\boldsymbol{w}\in\mathbb{L}_\Graph$ from Corollary~\ref{cor:criterion}
translates as the compatibility condition of the parity of the sum of
the lengths of the boundary components of each individual connected
ribbon graph as in Proposition~\ref{pr:count:with:leaves}.

We are ready to write a formula for the leading term in
the number of all square-tiled surfaces tiled with at most $2N$
squares represented by the stable graph $\Graph$ when the integer
bound $N$ becomes sufficiently large:
\begin{multline*}
\card(\cST_\Graph(\cQ(1^{4g-4+n}, 2N)))
\sim\\
\sim
(4g-4+n)!\cdot
\frac{1}{|\operatorname{Aut}(\Graph)|}\cdot
\sum_{\substack{\boldsymbol{w}\cdot\boldsymbol{h}\le N/2\\w_i, h_i\in\frac{1}{2}\N\\ 2\boldsymbol{w}\in\mathbb{L}_\Graph}}
(2w_1) \cdots (2w_k)\cdot
\prod_{v\in V(\Graph)}
N_{g_v,n_v}(2\boldsymbol{w}_v)\,,
\end{multline*}
The notation in the above expression mimic notation
in~\eqref{eq:square:tiled:volume}, namely $k=|E(\Graph)|$,
and $\boldsymbol{w}_v$ is defined analogously to
$\boldsymbol{b}_v$ in~\eqref{eq:square:tiled:volume}.
The factor $(4g-4+n)!$ represents the number of ways
to label the $4g-4+n$ trivalent vertices of the ribbon graphs,
which correspond to $\ell=4g-4+n$ simple zeroes of the corresponding
Strebel quadratic differential $q$. Note that by convention
the univalent vertices (leaves) (corresponding to simple poles
of $q$ and also to $n$ marked points) are already labeled.

Making a change of variables $b_h:=(2w_h)\in\N$ and
$H_i:=(2h_i)\in\N$ we can rewrite the above expression as
\begin{multline}
\label{eq:last}
\card(\cST_\Gamma(\QuadStrat(1^{4g-4+n}, -1^n), 2N))
\sim\\
\sim
(4g-4+n)!\cdot
\frac{1}{|\operatorname{Aut}(\Graph)|}\cdot
\sum_{\substack{\boldsymbol{b}\cdot\boldsymbol{H}\le 2N\\b_i, H_i\in\N\\ \boldsymbol{b}\in\mathbb{L}^k}}
b_1 \cdots b_k
\cdot
\prod_{v\in V(\Graph)}
N_{g_v,n_v}(\boldsymbol{b}_v)\,.
\end{multline}

The expression above is a homogeneous polynomial of degree
$6g-6+2n-k$. For any individual monomial the corresponding
sum was evaluated in
Lemma~\ref{lm:evaluation:for:monomial}. It remains to adapt
Formula~\eqref{eq:3.7} to our specific context.

By Corollary~\ref{cor:criterion} the sublattice
$\mathbb{L}_\Graph$ in the above formula has index
$2^{|V(\Graph)|-1}$ in $\Z^k$. The corresponding
factor $1/2^{|V(\Graph)|-1}$ appears
as the first factor in the second line of
Definition~\eqref{eq:P:Gamma}
of $P_\Gamma(\boldsymbol{b})$.

The degree of the homogeneous polynomial denoted by $|m|$ in
Formula~\eqref{eq:3.7} equals in our case to $6g-6+2n-k$, so
$$
|m|+k=6g-6+2n=\dim_{\C}\cQ(1^{4g-4+n},-1^n)=\dprinc\,.
$$
Note also that
in~\eqref{eq:3.7} we perform the summation under the condition
$\boldsymbol{b}\cdot\boldsymbol{H}\le N$ while in the above formula
we sum over the region $\boldsymbol{b}\cdot\boldsymbol{H}\le 2N$.
This provides an extra factor $2^{\dprinc}$.
Finally, passing from
$\card(\cST_\Gamma(\QuadStrat(1^{4g-4+n}, -1^n), 2N))$
in~\eqref{eq:last} to $\Vol\cQ_{g,n}$
by~\eqref{eq:Vol:sq:tiled}
we introduce the extra factor
$2(6g-6+2n)$.
The resulting product factor
$$
2(6g-6+2n)\cdot\frac{2^{\dprinc}}{\dprinc !} \cdot (4g-4+n)!
=
\frac{2^{6g-5+2n} \cdot (4g-4+n)!}{(6g-7+2n)!}
$$
is the factor in the first line of
Definition~\eqref{eq:P:Gamma} of
$P_\Gamma(\boldsymbol{b})$.

Theorem~\ref{th:volume} is proved.

The Lemma below is a straightforward combination of
Lemma~\ref{lm:evaluation:for:monomial} stated in terms of
densities and of Theorem~\ref{th:volume}. This Lemma
corresponds to Relation~\eqref{eq:measure:mu} from
Theorem~\ref{thm:statistics}. The other statements of
Theorem~\ref{thm:statistics} are proved analogously.

\begin{Lemma}
\label{lm:general:evaluation:for:monomial}
Let $F: \Delta^k \times \N^k \to \R$ be a continuous
function integrable with respect to the density
$\cZt(b_1^{m_1} \ldots b_k^{m_k}) \delta_{\boldsymbol{H}}
d\boldsymbol{x}$ defined in~\eqref{eq:measure:mu}.
Let $\mathbb{L}$ be a lattice of finite index in $\Z^k$.
Then
\begin{multline*}
\lim_{N\to+\infty}
\frac{1}{N^{|m|+k}}
\sum_{\substack{\boldsymbol{b}\cdot\boldsymbol{H}\le 2N
  \\b_i, H_i\in\N
  \\ \boldsymbol{b}\in\mathbb{L}}}
F\left( \left(\frac{b_1 H_1}{2N},\dots,\frac{b_k H_k}{2N}\right),
\boldsymbol{H}\right)
\,b_1^{m_1}\cdots b_k^{m_k}
\ =\\= \
\frac{1}{|\Z^k:\mathbb{L}|}\cdot
\sum_{\boldsymbol{H}}
\int_{\Delta^k} F(\boldsymbol{x}, \boldsymbol{H})\,
\cZt(b_1^{m_1}\cdots b_k^{m_k})\, d\boldsymbol{x}\,.
\end{multline*}
\end{Lemma}

\subsection{Yet another expression for the Siegel--Veech constant}

For any square-tiled surface $S$ define the following quantity. Suppose that
$S$ has $k$ maximal cylinders filled with closed horizontal trajectories.
Denote as usual by $w_i$ the length of the closed horizontal trajectory (length of the
waist curve) of the $i$-th cylinder and denote by $h_i$ its height.
Define
$$
M(S):=
\sum_{i=1}^k
\frac{h_{i}}{w_{i}}\,.
$$

The $\GL$-orbit of any square-tiled surface $S$ is a closed
invariant submanifold $\cL(S)$ in the ambient stratum of
quadratic (or Abelian) differentials. In the same way in
which we defined in
Section~\ref{ss:intro:Siegel:Veech:constants} the area
Siegel--Veech constant $\carea(\cQ_{g,n})$ for $\cQ_{g,n}$,
we can define the area Siegel--Veech constant $\carea(\cL(S))$
for $\cL(S)$. It satisfies, in particular, analogs
of~\eqref{eq:SV:constant:definition}
and~\eqref{eq:SV:asymptotics}. Theorem 4 in~\cite{Eskin:Kontsevich:Zorich} proves
the following assertion.

\begin{NNTheorem}[\cite{Eskin:Kontsevich:Zorich}]
For any square-tiled surface $S$, its Siegel--Veech constant satisfies:
\begin{equation}
\label{eq:SVconstant:for:square:tiled}
\carea(\cL(S))=
\cfrac{3}{\pi^2}\cdot
\cfrac{1}{\card(\SLZ\cdot S)}\
\sum_{S_i\in\SLZ\cdot S} M(S)
\end{equation}
\end{NNTheorem}

Recall that we denote by $\cST(\QuadStrat(1^{4g-4+n}, -1^n), N)$
the set of square-tiled surfaces in the principal stratum
$\QuadStrat(1^{4g-4+n}, -1^n)$ tiled with at most
$N$ squares. Define now the quantity:
$$
\cD\cST(\QuadStrat(1^{4g-4+n}, -1^n), N) :=
\sum_{S \in \cST(\QuadStrat(1^{4g-4+n}, -1^n), N)} M(S)\,.
$$

The latter quantity has the following geometric interpretation.
Consider all square-tiled surfaces obtained from square-tiled
surfaces as above by cutting exactly one cylinder $\cyl_i$ along the
closed horizontal trajectory at the level $h$, where $0\le h< h_i$
and $h$ is integer in the case of Abelian differentials and
half-integer in the case of quadratic differentials. In other words,
we do not chop the squares along the cut. The above quantity
enumerates bordered square-tiled surfaces obtained in this way.
Indeed, we lose the twist parameter $w_i$ (correspondingly $2w_i$)
along the cylinder which is now cut open, but we gain the new height
parameter $h_i$ (correspondingly $2h_i$) responsible for the level of
the cut.

As a corollary of the above theorem, D.~Chen and A.~Eskin
proved the following result.

\begin{NNTheorem}[D.~Chen, A.~Eskin, Appendix A in~\cite{Chen}]
The area Siegel--Veech constant $\carea(\cQ_{g,n})=\carea(\QuadStrat(1^{4g-4+n},-1^n))$ has the following value:
\begin{equation}
\label{eq:SVconstant:for:stratum}
\carea(\QuadStrat(1^{4g-4+n},-1^n))=
\cfrac{3}{\pi^2}\cdot
\lim_{N\to\infty}
\frac{\cD\cST(\QuadStrat(1^{4g-4+n}, -1^n), 2N)}
{\card(\cST(\QuadStrat(1^{4g-4+n}, -1^n), 2N))}\,.
\end{equation}
\end{NNTheorem}
Formally speaking, the original Theorem is proved only for
the components of the strata of Abelian differentials.
However, all the arguments are applicable to any arithmetic
$\GL$-invariant submanifold, in particular
to the loci induced from strata of quadratic differentials by
the ramified double covering construction.

An alternative way to
derive~\eqref{eq:SVconstant:for:stratum} is to use the
following more elaborate technique. Neglecting exceptional
orbits of square-tiled surfaces in
$\QuadStrat(1^{4g-4+n},-1^n)$ containing negligibly small
number of square-tiled surfaces, we can arrange other
orbits in a sequence of affine invariant manifolds for
which the natural $\SL$-invariant measures supported on the
orbits converge to the invariant measure of the ambient
stratum. By Theorem~2.8 from~\cite{Bonatti:Eskin:Wilkinson}
all individual Lyapunov exponents of affine invariant
manifolds converge to the Lyapunov exponents of the ambient
stratum. Results from~\cite{Eskin:Kontsevich:Zorich} now
imply that the area Siegel--Veech constants of the
corresponding arithmetic Teichm\"uller discs converge to
the area Siegel--Veech constant of the ambient stratum,
which implies~\eqref{eq:SVconstant:for:stratum}.

\subsection{Proof of the formula for the area Siegel--Veech constant}
We have already evaluated the denominator
in~\eqref{eq:SVconstant:for:stratum}. Evaluating the numerator
following the lines of the initial computation we reduce the problem to
the evaluation of the sum~\eqref{eq:last} counted with the weight
$\frac{h_i}{w_i}=\frac{H_i}{b_i}$, for each $i=1,\dots,k$,
where we use the notation of Formula~\eqref{eq:last}:
\begin{equation}
\label{eq:last:carea:1}
\sum_{\substack{\boldsymbol{b}\cdot\boldsymbol{H}\le 2N\\b_j, H_j\in\N\\ \boldsymbol{b}\in\mathbb{L}_\Graph}}
\left(b_1 \cdots b_k
\cdot
\prod_{v\in V(\Graph)}N_{g_v,n_v}(\boldsymbol{b}_v)\right)\cdot\frac{H_i}{b_i}
\end{equation}
The numerator in~\eqref{eq:SVconstant:for:stratum}
is the sum of the above expressions with respect to
the summation index $i$ varying from $1$ to $k$.

Denote by $P(b_1,\dots,b_k)$ the homogeneous polynomial
$\prod N_{g_v,n_v}(\boldsymbol{b}_v)$
in the formula above.
It is easy to see that the condition
$b_i H_i<\boldsymbol{b}\cdot\boldsymbol{H}\le 2N$ implies that
the contribution of any monomial of $P(b_1,\dots,b_k)$
containing the variable $b_i$ to the above sum
is of order $o(N^{\dprinc})$, so it does not contribute to the limit~\eqref{eq:SVconstant:for:stratum}.
Thus, up to lower order terms the sum~\eqref{eq:last:carea:1}
coincides with the sum
\begin{equation}
\label{eq:last:carea:2}
\sum_{\substack{\boldsymbol{b}\cdot\boldsymbol{H}\le 2N\\b_j, H_j\in\N\\ \boldsymbol{b}\in\mathbb{L}_\Graph}}
b_1 \cdots b_{i-1}\cdot H_i\cdot b_{i+1}\cdots b_k
\cdot
P(b_1,\dots,b_{i-1},0,b_{i+1},\dots, b_k)\,.
\end{equation}
It is sufficient to interchange the notation $b_i$ and $H_i$
to see that
\begin{multline*}
\sum_{\substack{\boldsymbol{b}\cdot\boldsymbol{H}\le 2N\\b_i, H_i\in\N}}
b_1 \cdots b_{i-1}\cdot H_i\cdot b_{i+1}\cdots b_k
\cdot
P(b_1,\dots,b_{i-1},0,b_{i+1},\dots, b_k)
=\\=
\sum_{\substack{\boldsymbol{b}\cdot\boldsymbol{H}\le 2N\\b_i, H_i\in\N}}
b_1 \cdots b_k
\cdot
P(b_1,\dots,b_{i-1},0,b_{i+1},\dots, b_k)\,.
\end{multline*}
We already know how to evaluate the latter sum, so it remains to
study the impact of the extra condition $\boldsymbol{b}\in\mathbb{L}_\Graph$
present in the sum~\eqref{eq:last:carea:2}.

Recall the strategy of evaluation of the sum~\eqref{eq:last:carea:2}
(see the proof of analogous Lemma~3.7 in~\cite{AEZ:Dedicata} for
reduction to integral sums and the proof of
Lemma~\ref{lm:evaluation:for:monomial} for the impact of the
sublattice condition). The variables $H_1,\dots, H_{i-1}, b_i,
H_{i+1},\dots, H_k$ are considered as parameters. For each collection
of such parameters we evaluate the corresponding integral sum over a
simplex in the $k$-dimensional space with coordinates
$b_1,\dots,b_{i-1},H_i, b_{i+1},\dots, b_k$. After that we perform
summation with respect to the parameters $H_1,\dots, H_{i-1}, b_i,
H_{i+1},\dots, H_k$,

When the edge of the graph $\Graph$ corresponding to the
variable $b_i$ is a \textit{bridge} (i.e. when this edge is
separating), the \textit{parameter} $b_i$ is always
even. The space of integration now has coordinates
$b_1,\dots,b_{i-1},H_i, b_{i+1},\dots, b_k$; the sublattice
$\mathbb{L}_\Graph$ in it is defined by the system of
Equations~\eqref{eq:sublattice:L} where we let $b_i=0$.
Such sublattice has index $2^{|V(\Graph)|-2}$ and not
$2^{|V(\Graph)|-1}$ as before. Thus, on the level of
integration we gain the factor $2$ with respect to the initial
count. However, since the \textit{parameter} $b_i$ is now
always even, evaluating the corresponding sum with respect
to possible values of this parameter we get the sum
$$
\frac{1}{2^2}+\frac{1}{4^2}+\frac{1}{6^2}\dots=
\frac{1}{4}\cdot \zeta(2)
$$
instead of the original sum
$$
\frac{1}{1^2}+\frac{1}{2^2}+\frac{1}{3^2}\dots=\zeta(2)\,.
$$
Thus, when $b_i$ corresponds to a bridge
(i.e. to a separating edge), we get
\begin{multline*}
\sum_{\substack{\boldsymbol{b}\cdot\boldsymbol{H}\le 2N\\b_j, H_j\in\N\\ \boldsymbol{b}\in\mathbb{L}_\Graph}}
b_1 \cdots b_{i-1}\cdot H_i\cdot b_{i+1}\cdots b_k
\cdot
P(b_1,\dots,b_{i-1},0,b_{i+1},\dots, b_k)
=\\=
\frac{1}{2}\cdot\sum_{\substack{\boldsymbol{b}\cdot\boldsymbol{H}\le 2N\\b_j, H_j\in\N\\ \boldsymbol{b}\in\mathbb{L}_\Graph}}
b_1 \cdots b_k
\cdot
P(b_1,\dots,b_{i-1},0,b_{i+1},\dots, b_k)\,.
\end{multline*}

When $b_i$ corresponds to a non separating edge, the \textit{parameter}
$b_i$ in the sum~\eqref{eq:last:carea:2} can take even and odd values.
The new space of integration has coordinates
$b_1,\dots,b_{i-1},H_i, b_{i+1},\dots, b_k$;
where the sublattice in it is defined by the
system of Equations~\eqref{eq:sublattice:L}
in which we substitute $b_i=0$ or $b_i=1$
depending on the parity of the value of the parameter $b_i$.
The sublattice is linear in
the first case and affine in the second case.
Such a sublattice has index $2^{|V(\Graph)|-1}$
as before. Thus, when $b_i$ corresponds to a non-separating edge, we get
\begin{multline*}
\sum_{\substack{\boldsymbol{b}\cdot\boldsymbol{H}\le 2N\\b_j, H_j\in\N\\ \boldsymbol{b}\in\mathbb{L}_\Graph}}
b_1 \cdots b_{i-1}\cdot H_i\cdot b_{i+1} b_k
\cdot
P(b_1,\dots,b_{i-1},0,b_{i+1},\dots, b_k)
=\\=
\sum_{\substack{\boldsymbol{b}\cdot\boldsymbol{H}\le 2N\\b_j, H_j\in\N\\ \boldsymbol{b}\in\mathbb{L}_\Graph}}
b_1 \cdots b_k
\cdot
P(b_1,\dots,b_{i-1},0,b_{i+1},\dots, b_k)\,.
\end{multline*}

We have proved that
\begin{multline*}
\sum_{\substack{\boldsymbol{b}\cdot\boldsymbol{H}\le 2N\\b_j, H_j\in\N\\ \boldsymbol{b}\in\mathbb{L}_\Graph}}
\left(b_1 \cdots b_k
\cdot \prod_{v\in V(\Graph)}
N_{g_v,n_v}(\boldsymbol{b}_v)
\right)
\cdot\left(\sum_{i=1}^k\frac{H_i}{b_i}\right)
=\\=
\sum_{\substack{\boldsymbol{b}\cdot\boldsymbol{H}\le 2N\\b_j, H_j\in\N\\ \boldsymbol{b}\in\mathbb{L}_\Graph}}
b_1 \cdots b_k\cdot\cD_{\Graph}
\left(
\prod_{v\in V(\Graph)}
N_{g_v,n_v}(\boldsymbol{b}_v)
\right)
+\text{lower order terms}\,,
\end{multline*}
where operator $\cD_{\Graph}$ is defined in
Formula~\eqref{eq:operator:D}. Applying to the latter sum the same
technique as in the end of the proof of Theorem~\ref{th:volume} we
complete the proof of Theorem~\ref{th:carea}.

\subsection{Equivalence of two expressions for the Siegel--Veech constant}
\label{eq:proof:of:coincidence:of:two:SV:expressions}

In this section we prove Theorem~\ref{th:same:SV}.

We start the proof by establishing a natural correspondence
between the summands of the two expressions. For any stable
graph $\Gamma\in\cG_{g,n}$ and any edge $e$ of $\Gamma$
we define a combinatorial surgery
$\Cut_e\Graph$ of $\Gamma$. We describe it separately
in the case when $e$ is a bridge (i.e. a separating edge),
and when it is not.

We start with the case when $e$ is a bridge. Cut the edge
$e$ transforming it into two legs. Assign index $1$ to one
of the resulting graphs, and index $2$ to the remaining
one. We do not modify the genus decoration of the vertices.
The set of vertices $V(\Gamma)$ gets naturally partitioned
into two complementary subsets $V=V_1\sqcup V_2$. Define
$g_i=\sum_{v\in V_i} g_v$ for $i=1,2$. Similarly, the $n$
original legs are partitioned into $n_1$ legs which go to
$\Gamma_1$ and into $n_2$ legs which go to $\Gamma_2$. For
$i=1,2$, relabel the $n_i$ legs of $\Gamma_i$ to the
consecutive labels $1,2,\dots,n_i$ preserving the order
of labels. Assign the label $n_i+1$ to the new leg of
$\Gamma_i$ created during the surgery. The stability
condition $2 g_v - 2 + n_v > 0$, which is valid for every
vertex $v$ of $\Graph$, implies that we get two stable
graphs $\Gamma_i\in\cG_{g_i,n_i+1}$.

The only ambiguity in this construction is the choice of
the label ($1$ or $2$) for one of the components $\Gamma_i$
of the graph $\Gamma$ with removed bridge $e$. In general,
there are two choices, except the case when there is a
symmetry of $\Gamma$ acting on the edge $e$ as a flip (i.e.
a symmetry which sends $e$ to itself exchanging its two
ends).

Note that the surgery is reversible in the following sense.
Given two stable graphs $\Gamma_i\in\cG_{g_i,n_i+1}$ we can
glue the endpoint of the leg with index $n_1+1$ of
$\Gamma_1$ to the endpoint of the leg with index $n_2+1$ of
$\Gamma_2$ creating a connected graph with $n=n_1+n_2$ legs
and with an extra bridge joining $\Gamma_1$ to $\Gamma_2$.
The only ambiguity in this construction is in relabeling
the $n=n_1+n_2$ legs to a consecutive list $(1,2,\dots,n)$;
there are $\binom{n}{n_1}$ ways to do it.

We describe now the surgery $\Cut_e\Gamma$ in the
remaining case when the edge $e$ of $\Gamma$ is not a
bridge (i.e. is not separating). Cutting such edge we
transform it into two legs. We keep the same labels for the
preexisting legs and we associate labels $n+1$ and $n+2$ to
the two created legs. In general, there are two ways to do
that, except when there is a symmetry of $\Gamma$ acting on
the edge $e$ as a flip (i.e. a symmetry which sends $e$ to
itself exchanging its two ends). We get a stable graph
$\Gamma'\in\cG_{g-1,n+2}$.

The inverse operation (applicable to stable graphs with at
least two legs) consists in gluing the two legs of higher
index together transforming them into an edge and keeping
the same labeling for the other legs. Once again we do not
modify the genus decoration of the vertices.

Note that the operator $\partial_\Gamma$ is defined
in~\eqref{eq:operator:D} as a sum
$\sum \partial^e_{\Graph}$
over edges $e\in E(\Gamma)$ of a
stable graph $\Gamma$. Thus, our key sum $\sum_{\Graph \in
\cG_{g,n}} \cZ\left(\partial_{\Graph} P_\Gamma\right)$ in
the right-hand side of~\eqref{eq:carea} in
Theorem~\ref{th:carea} can be seen as the sum over all
pairs $(\Gamma,e)$, where $\Gamma\in\cG_{g,n}$ and $e\in
E(\Gamma)$. We show below that for every such pair
$(\Gamma,e)$, the corresponding term of the resulting sum
has simple expression in terms of the product
$\cZ(P_{\Gamma_1})\cZ(P_{\Gamma_2})$ when $e$ is a bridge
and in terms of $\cZ(P_{\Gamma'})$ when $e$ is not a
bridge, where $\Gamma_1\sqcup\Gamma_2$ (respectively $\Gamma'$)
are the stable graphs obtained under applying the surgery
$\Cut_e\Gamma$.

Having a stable graph $\Graph$ we associate to it the
polynomial
$$
\Pi_{\Graph}(\boldsymbol{b}):=
\prod_{v\in V(\Graph)}
N_{g_v,n_v}(\boldsymbol{b}_v)\,.
$$
By Definition~\eqref{eq:P:Gamma} of
$P_{\Graph}(\boldsymbol{b})$ we have
$$
P_{\Graph}(\boldsymbol{b})
=(\text{combinatorial factor})
\cdot\left(\prod_{e\in E(\Gamma)} b_e\right)
\cdot\Pi_{\Graph}(\boldsymbol{b})\,.
$$
A pair $(\Graph,e_0)$ provides a nonzero contribution
to the sum in the right-hand side of~\eqref{eq:carea}
if and only if the term
$\partial^{e_0}_{\Graph}=\chi_{\Graph}(e_0)b_{e_0}
\left.\frac{\partial}{\partial b_{e_0}}\right|_{b_{e_0}=0}$
in the operator $\partial_{\Graph}$
applied to
$\left(\prod_{e\in E(\Gamma)} b_e\right)
\cdot\Pi_{\Graph}(\boldsymbol{b})$
does not identically vanish
(see given by~\eqref{eq:operator:D:e}).
The latter is
equivalent to the condition that the polynomial
$\left.\Pi_{\Graph}\right|_{b_{e_0}=0}$ does not identically
vanish.

If the edge $e_0$ is a bridge, consider the stable
graphs $\Gamma_1, \Gamma_2$ obtained under the surgery
$\Cut_e\Gamma$. The polynomial
$\left.\Pi_{\Graph}\right|_{b_{e_0}=0}$
splits naturally into the product:
$
\left.\Pi_{\Graph}\right|_{b_j=0}
=\Pi_{\Graph_1}\Pi_{\Graph_2}\,,
$
so when $e_0$ is a bridge,
and when
$\left.\Pi_{\Graph}\right|_{b_{e_0}=0}$ does not identically
vanish, we get
\begin{multline}
\label{eq:separating}
\cZ\left(\partial_{\Graph}^{e_0}\left( \prod_e b_e \cdot
\Pi_{\Graph}\right)\right)
=\cZ\left(\frac{1}{2}\cdot \prod_e b_e
 \cdot  \left.\Pi_{\Graph}\right|_{b_{e_0}=0}\right)
\\
=\cZ\left(\frac{1}{2}\cdot b_{e_0}\cdot
 \prod_{e\in E(\Graph_1)}b_e\cdot \Pi_{\Graph_1}
 \cdot \prod_{e\in E(\Graph_2)}b_e\cdot \Pi_{\Graph_2}\right)\\
=\\=
\frac{1}{2}\cdot\frac{\pi^2}{6}
 \cdot\cZ\left(\prod_{e\in E(\Graph_1)}b_e\cdot \Pi_{\Graph_1}\right)
 \cdot\cZ\left(\prod_{e\in E(\Graph_2)}b_e\cdot \Pi_{\Graph_2}\right)\,.
\end{multline}

If the edge $e_0$ is
not a bridge
and
$\left.\Pi_{\Graph}\right|_{b_{e_0}=0}$ does not identically
vanish, we get
$\left.\Pi_{\Graph}\right|_{b_{e_0}=0}=\Pi_{\Graph'}$
so
\begin{multline}
\label{eq:non:separating}
\cZ\left(\partial_{\Graph}^{e_0}
\left( \prod_e b_e \cdot  \Pi_{\Graph}\right)\right)
=\\=
\cZ\left(b_{e_0}\cdot\prod_{e\in E(\Graph')}b_e\cdot \Pi_{\Graph'}\right)
=\frac{\pi^2}{6}\cdot\cZ\left(\prod_{e\in E(\Graph')}b_e\cdot \Pi_{\Graph'}\right)\,.
\end{multline}

Rewrite the right-hand side
of~\eqref{eq:carea} as
$$
\frac{3}{\pi^2}
\cdot
\sum_{\Graph \in \cG_{g,n}}
\cZ\left(\partial_{\Graph}
P_\Gamma\right)
=
\frac{3}{\pi^2}
\cdot
\sum_{\Graph \in \cG_{g,n}}\sum_{e\in E(\Gamma)}
\cZ\left(\partial^e_{\Graph}
P_\Gamma\right)
$$
Apply~\eqref{eq:separating}
and~\eqref{eq:non:separating}
to the resulting sum, keeping the
intersection numbers as formal expressions
and simplify the product of the factors
$\frac{3}{\pi^2}$ and
$\frac{\pi^2}{6}$.

Suppose now that $g\ge 1$ (the consideration in the case
$g=0$ is completely analogous). Consider the sum in the
right-hand side of~\eqref{eq:carea:Elise} and replace
$\Vol\cQ_{g_i,n_i}$ for $i=1,2$ and $\Vol\cQ_{g-1,n+2}$ in
it by the corresponding sums~\eqref{eq:square:tiled:volume}
over $\cG_{g_1,n_1}\times\cG_{g_2,n_2}$ and $\cG_{g-1,n+2}$
respectively (where we keep the intersection numbers as
formal expressions. It is easy to see, that we get
term-by-term the same sum as above.
Theorem~\ref{th:same:SV} is proved.


\section{Comparison with Mirzakhani's formula for $b_{g,n}$}
\label{s:comparison:with:Mirzakhani}

\subsection{Mirzakhani's expression for the volume of
$\cQ_g$} Consider a pair $(X,\lambda)$, where $X$ is a
hyperbolic surface of genus $g$ without punctures and
$\lambda$ is a measured lamination on $X$. In the
paper~\cite{Mirzakhani:earthquake} M.~Mirzakhani associates
to almost any such pair $(X,\lambda)$ a unique holomorphic
quadratic differential $q=F(\lambda,X)$ on the complex
curve $C=C(X)$ corresponding to the hyperbolic metric $X$.

Consider the measure $\mu_{\mathrm{WP}}$ on the moduli
space $\cM_g$ coming from the Weil--Petersson volume
element and consider Thurston measure $\mu_{\mathrm{Th}}$
on the space of measured laminations $\cM\cL_g$. Using
ergodicity arguments, M.~Mirzakhani proves
in~\cite{Mirzakhani:earthquake} that the pushforward
measure
\begin{equation}
\label{eq:F:ast:WP:times:Th}
\mu_g:= F_\ast(\mu_{\mathrm{WP}} \otimes \mu_{\mathrm{Th}})
\end{equation}
on $\cQ_g$ under the map
$$
F: \cM_g\times\cM\cL_g \to \cQ_g
$$
is proportional to the Masur--Veech measure.
Equation~\eqref{eq:F:ast:WP:times:Th} should be considered
as the definition of Mirzakhani's normalization of the
Masur---Veech measure on $\cQ_g$.

\begin{Remark}
Note that under such implicit definition of the
Masur--Veech measure it is not clear at all why its density
should be constant in period coordinates.
Definition~\eqref{eq:F:ast:WP:times:Th} does not provide
any distinguished lattice in period coordinates either.
Thus, though we know, by results of Mirzakhani, that the
density of the Masur--Veech measure defined
by~\eqref{eq:F:ast:WP:times:Th} differs from the
Masur--Veech volume element defined in
section~\ref{ss:background:strata} by a constant numerical
factor which depends only on $g$, evaluation of this factor
is not straightforward.
\end{Remark}

M.~Mirzakhani proves in~\cite{Mirzakhani:earthquake} that
the map $F$ identifies the length
$\ell_\lambda(X)$ of the measured lamination $\lambda$
evaluated in the hyperbolic metric $X$ with the norm
$\|q(\lambda,X)\|=\int_C |q|$ of the quadratic differential
$q$ (defined as the area of the flat surface associated to
the pair $(C,q)$):
\begin{equation}
\label{eq:norm:q:length:lambda}
\int_{C(X)} |q(\lambda,X)|=\ell_\lambda(X)\,.
\end{equation}

Relation~\eqref{eq:norm:q:length:lambda} implies that as
the image of $F$ restricted to the total space of the
bundle of ``unit balls'' $B_X$ over $\cM_g$,
defined by Equation~\eqref{eq:unit:ball},
one gets the
total space $\cQ^{\le 1}_g$ of the bundle of ``unit balls''
in $\cQ_g$, where
\begin{equation}
\label{eq:ball:of:radius:1:in:Q:g:n}
\cQ^{\Area\le 1}_g
=
\left\{(C,q)\in\cQ_g\,|\, \Area(C,q) \le 1\right\}\,.
\end{equation}
We have seen in Section~\ref{ss:MV:volume} that the real
hypersurface
\begin{equation}
\label{eq:sphere:of:radius:1:in:Q:g:n}
\cQ^{\Area=1}_g
=
\left\{(C,q)\in\cQ_{g,n}\,|\, \Area(C,q) = 1\right\}\,.
\end{equation}
can be seen as the unit cotangent bundle to $\cM_g$
(denoted by $\cQ^1\cM_g$ in~\cite{Mirzakhani:earthquake}).
M.~Mirzakhani defines the Masur--Veech volume of $\cQ_g$ as
\begin{equation}
\label{eq:Vol:Q:1:Mirzakhani:definition}
\Vol_{\mathrm{Mir}}\cQ_g:=\mu_g(\cQ^{\Area\le 1}_g)\,.
\end{equation}
The above observations imply the following formula
for $\Vol_{\mathrm{Mir}}\cQ_g$
(see Theorem~1.4 in~\cite{Mirzakhani:earthquake}):
\begin{NNTheorem}[Mirzakhani]
   %
The Masur--Veech volume of the moduli space of holomorphic
quadratic differentials on complex curves of genus $g$
defined by~\eqref{eq:Vol:Q:1:Mirzakhani:definition}
under normalizations~\eqref{eq:F:ast:WP:times:Th}
satisfies the following relation
\begin{equation}
\label{eq:MV:Volume:as:bg}
\Vol_{\mathrm{Mir}}\cQ_g = b_g.
\end{equation}
where $b_g = b_{g,0}$ is defined by~\eqref{eq:b:g:n}.
\end{NNTheorem}

Formula~\eqref{eq:Vol:g:n:b:g:n} from Corollary~\ref{cor:Vol:as:b:g:n}
relating $\Vol\cQ_g =\Vol\cQ_{g,0}$ and $b_g=b_{g,0}$
implies that
the Masur--Veech volume $\Vol_{\mathrm{Mir}}\cQ_g$ in
Mirzakhani's normalization~\eqref{eq:MV:Volume:as:bg}
and the Masur--Veech volume $\Vol\cQ_g$
in normalization of Formula~\eqref{eq:square:tiled:volume}
are related by the following factor:
\begin{equation}
\label{eq:Vol:Q:g:us:and:Mirzakhani}
\Vol\cQ_g
=\big((12g-12)\cdot
(4g-4)!\cdot 2^{4g-3}\big)\cdot
\Vol_{\mathrm{Mir}} \cQ_g\,.
\end{equation}

\subsection{Mirzakhani's formulae for the Masur--Veech volume of $\cQ_g$
and for asymptotic frequencies of simple closed geodesic multicurves}

As we have seen, $\Vol_{\mathrm{Mir}}\cQ_g$ that appears in
Formula~\eqref{eq:MV:Volume:as:bg} is a particular case of the more general
quantity $b_{g,n}$ defined in~\eqref{eq:b:g:n}.
The quantity $b_{g,n}$ is computed in Theorem~5.3 on page~118
in~\cite{Mirzakhani:grouth:of:simple:geodesics}.
To reproduce the corresponding formula and closely related
formula for the asymptotic frequencies $c(\gamma)$
of simple closed geodesic multicurves $\gamma$
of fixed topological type we need to remind the
notation from~\cite{Mirzakhani:grouth:of:simple:geodesics}.
\medskip

\noindent
\textbf{Simple closed multicurves.}
Depending on the context, we denote by the same symbol
$\gamma$ a collection of disjoint, essential, nonperipheral
simple closed curves, no two of which are in the same
homotopy class; a disjoint union of such curves; the
corresponding primitive multicurve; and the corresponding
orbit in the space $\cM\cL_{g,n}$ under the action of the
mapping class group $\Mod_{g,n}$,
$$
\gamma
=\gamma_1\sqcup\dots\sqcup\gamma_k
=\gamma_1+\dots+\gamma_k
=(\gamma_1,\dots,\gamma_k)\,.
$$
(M.~Mirzakhani uses
in~\cite{Mirzakhani:grouth:of:simple:geodesics} symbols
$\gamma$, $\hat\gamma$ and $\tilde\gamma$ for these objects
depending on the context.) To every such multicurve
$\gamma$ M.~Mirzakhani associates
in~\cite{Mirzakhani:grouth:of:simple:geodesics} a
collection of quantities $N(\gamma)$, $M(\gamma)$,
$\Sym(\gamma)$, $b_\gamma, \Vol_{\mathrm{WP}}(\cM_{g,n}(\gamma,x))$
involved in the formula for $b_{g,n}$. For the sake of
completeness, and to simplify formulae comparison we
reproduce the definitions of these quantities; see the
original paper~\cite{Mirzakhani:grouth:of:simple:geodesics}
of M.~Mirzakhani for details.
\medskip

\noindent
\textbf{The collection $\boldsymbol{\cS_{g,n}}$ of all topological types of primitive multicurves.}
Recall that by $S_{g,n}$ we denote a smooth orientable
topological surface of genus $g$ with $n$ punctures.
Consider the finite set $\cS_{g,n}$ defined as
\begin{equation}
\label{eq:S:cal:g:n}
\cS_{g,n}:=
\{\gamma\,|\,\gamma\ \text{is a union of simple closed curves on}
\ S_{g,n}\}/\Mod_{g,n}\,.
\end{equation}
(see formula~(5.4) on page~118
in~\cite{Mirzakhani:grouth:of:simple:geodesics}). It is
immediate to see that $\cS_{g,n}$ is in a canonical
bijection with the set $\cG_{g,n}$ of stable graphs defined
in section~\ref{ss:intro:Masur:Veech:volumes},
\begin{equation}
\label{eq:cS:cG}
\cS_{g,n}\simeq \cG_{g,n}\,.
\end{equation}

\noindent
\textbf{Symmetries $\boldsymbol{\Stab(\gamma)}$ and $\boldsymbol{N(\gamma)}$ of a primitive multicurve.}
For any set $A$ of homotopy classes of simple closed curves
on $S_{g,n}$, Mirzakhani defines $\Stab(A)$ as
$$
\Stab(A):=\{g\in\Mod{g,n}\,|\, g\cdot A=A\}\subset\Mod_{g,n}\,.
$$
Having a multicurve $\gamma$ on $S_{g,n}$ as above,
Mirzakhani defines
$$
\Sym(\gamma):=\Stab(\gamma)/\cap_{i=1}^k\Stab(\gamma_i)\,,
$$
(see the beginning of section~4 on page~112
of~\cite{Mirzakhani:grouth:of:simple:geodesics} for both
definitions). For any single connected simple closed curve
$\gamma_i$ one has $|\Sym(\gamma_i)|=1$.

For each connected simple closed curve
$\gamma_i$ define
$\Stab_0(\gamma_i)\subset\Stab(\gamma_i)$ as
the subgroup consisting of elements which
preserve the orientation of $\gamma_i$.
Define $N(\gamma)$ as
$$
N(\gamma):=\left|
\bigcap_{i=1}^k\Stab(\gamma_i)/
\bigcap_{i=1}^k\Stab_0(\gamma_i)\right|\,.
$$
(see page~113
of~\cite{Mirzakhani:grouth:of:simple:geodesics}).

Consider a stable graph $\Gamma(\gamma)$ associated
to a primitive multicurve $\gamma$.
It follows from definitions of $\Stab(\gamma)$, $N(\gamma)$
and $\Aut(\gamma)$ that
\begin{equation}
\label{eq:Sym:N:Aut}
|\Aut(\Gamma(\gamma))|=|\Sym(\gamma)|\cdot N(\gamma)\,.
\end{equation}

\noindent
\textbf{Number $\boldsymbol{M(\gamma)}$ of one-handles.}
Consider now the closed surface
\begin{equation}
\label{eq:decomposition:of:S:g:n}
S_{g,n}(\gamma)=\bigsqcup_{j=1}^s S_{g_j,n_j}
\end{equation}
obtained from $S_{g,n}$ by cutting along all
$\gamma_1,\dots,\gamma_k$. Here $S_{g_1,n_1},\dots,
S_{g_s,n_s}$ are the connected components of the resulting
surface $S_{g,n}(\gamma)$. Define
$$
M(\gamma):=|\{i\,|\, \gamma_i\
\text{separates off a one-handle from}\ S_{g,n}\}|\,,
$$
(see this formula in the statement of Theorem~4.1 on
page~114 of~\cite{Mirzakhani:grouth:of:simple:geodesics}).
By definition, a ``one-handle'' is a surface of genus one
with one boundary component, i.e., a surface of type
$S_{1,1}$.

\begin{Remark}
\label{rm:one:handle}
There is one very particular case,
when the index $i$ in the definition of $M(\gamma)$ should
be counted with multiplicity $2$. Namely, when $\gamma$ is
a connected separating simple closed curve on a surface
$S_2=S_{2,0}$ it simultaneously separates
off \textit{two} one-handles, and thus should be
counted with multiplicity $2$. In other words, the quantity
$M(\gamma)$ might be defined as
\begin{equation}
\label{eq:M:gamma}
M(\gamma):=\text{number of surfaces of type }\ S_{1,1}\
\text{in decomposition~\eqref{eq:decomposition:of:S:g:n}}
\end{equation}
without any exceptions and multiplicities.
\end{Remark}

A.~Wright suggested an alternative way to fix this issue; see footnote~2 on pages 12--13 in~\cite{Wright}.
\medskip

\noindent
\textbf{Volume polynomials $\boldsymbol{V_{g,n}}$.}
In~\cite{Mirzakhani:simple:geodesics:and:volumes}
and~\cite{Mirzakhani:volumes:and:intersection:theory}
M.~Mirzakhani
proves the following statement, that we reproduce from
Theorems~4.2 and~4.3
in~\cite{Mirzakhani:grouth:of:simple:geodesics}.

\begin{NNTheorem}[Mirzakhani]
The Weil--Petersson volume
$\Vol_{\textrm{WP}}\cM_{g,n}(b)$
of the moduli space
of bordered
hyperbolic surfaces of genus $g$ with hyperbolic
boundary components of lengths $b_1,\dots,b_n$
is a polynomial $V_{g,n}(b_1,\dots,b_n)$
in even powers of $b_1,\dots,b_n$; that is,
\begin{equation*}
\Vol_{\mathrm{WP}}\cM_{g,n}(b)
=V_{g,n}(b)=\sum_{\substack{\alpha\\|\alpha|\le 3g-3+n}} C_\alpha\cdot b^{2\alpha}\,,
\end{equation*}
where $C_\alpha>0$ lies in $\pi^{6g-6+2n-2|\alpha|}\cdot\Q$.
The coefficient $C_\alpha$ is given by
\begin{equation*}
C_\alpha=\frac{1}{2^{|\alpha|}\cdot\alpha!\cdot(3g-3+n-|\alpha|)!}
\int_{\overline{\cM}_{g,n}}
\psi_1^{\alpha_1}\cdots\psi_n^{\alpha_n}
\cdot\omega^{3g-3+n-|\alpha|}\,,
\end{equation*}
where $\omega$ is the Weil--Petersson symplectic form,
$\alpha!=\prod_{i=1}^n \alpha_i!$, and $|\alpha|=\sum_{i=1}^n\alpha_i$.
\end{NNTheorem}

\begin{Remark}[M.~Kazarian]
\label{rm:Kazarian}
For one particular pair $(g,n)$, namely
for $(g,n)=(1,1)$, Mirzakhani's normalization
\begin{equation}
\label{eq:Vol:M11:Mirzakhani}
\Vol^{\mathit{Mir}}_{\textrm{WP}}\cM_{1,1}(b)
=\frac{1}{24}(b^2+4\pi^2)\,,
\end{equation}
(see equation~(4.5) on page~116
in~\cite{Mirzakhani:grouth:of:simple:geodesics}
or~\cite{Mirzakhani:volumes:and:intersection:theory}) is
twice bigger than the normalization in many other papers.
Topologically $\overline{\cM}_{1,1}$ is homeomorphic to
$\CP$. However, since every elliptic curve with a marked
point admits an involution, the fundamental class of the
orbifold $\overline{\cM}_{1,1}$ used in the integration
equals
$[\overline{\cM}_{1,1}]=\frac{1}{2}\left[\CP\right]\in
H_2(\CP)$ which gives the value
\begin{equation}
\label{eq:Vol:M11:Kazarian}
\Vol_{\textrm{WP}}\cM_{1,1}(b)
=\frac{1}{48}(b^2+4\pi^2)\,.
\end{equation}

\end{Remark}

Denote by $V^{top}_{g,n}(b)$ the homogeneous part of the
top degree $6g-6+2n$ of $V_{g,n}(b)$. It follows from
the definition of the volume polynomial $V_{g,n}(b)$ that
\begin{equation}
\label{eq:Vgn}
V^{top}_{g,n}(b)=\sum_{|\alpha|= 3g-3+n} C_\alpha\cdot b^{2\alpha}\,,
\end{equation}
where $C_\alpha$ is given by
\begin{equation}
\label{eq:C:alpha}
C_\alpha=\frac{1}{2^{3g-3+n}\cdot\alpha!}
\int_{\overline{\cM}_{g,n}}
\psi_1^{\alpha_1}\cdots\psi_n^{\alpha_n}\,.
\end{equation}
Comparing the definition of $V^{top}_{g,n}(b)$
with the Definition~\eqref{eq:N:g:n}--\eqref{eq:correlator}
of the polynomial $N_{g,n}(b)$ we get the following result.
\begin{Lemma}
The homogeneous parts of top degree of Mirzakhani's
volume polynomial $V_{g,n}$ and of Kontsevich's polynomial
$N_{g,n}$ coincide up to the constant factor:
\begin{equation}
\label{eq:Vgn:through:Ngn}
V^{top}_{g,n}(b)
=\begin{cases}
\hspace*{12pt}
2^{2g-3+n}\cdot N_{g,n}(b)
&\text{for}\ (g,n)\neq(1,1)\,;
\\
2\cdot 2^{2g-3+n}\cdot N_{g,n}(b)
&\text{for}\ (g,n)=(1,1)\,.
\end{cases}
\end{equation}
\end{Lemma}
In the exceptional case $(g,n)=(1,1)$ the polynomial
$V^{top}_{1,1}(b)=\frac{1}{24} b^2$ used by Mirzakhani is
twice larger than $N_{1,1}(b)=\frac{1}{48} b^2$. The origin
of this extra factor $2$ is explained in
Remark~\ref{rm:Kazarian}.
\medskip

\noindent \textbf{Volume polynomial
$\boldsymbol{\Vol_{\mathrm{WP}}(\cM_{g,n}(\gamma,x))}$
associated to a multicurve $\boldsymbol{\gamma}$.}
Assuming that the initial surface $S_{g,n}$
is endowed with a hyperbolic metric, and
that the simple closed curves $\gamma_1,\dots,\gamma_k$
are realized by simple closed hyperbolic geodesics
of hyperbolic lengths $x_1=\ell_{\gamma_1}(X),\dots,
x_k=\ell_{\gamma_k}(X)$,
the boundary $\partial S_{g,n}(\gamma)$
gets $k$ pairs of distinguished boundary
components of lengths $x_1,\dots,x_k$.
We assume that the $n$ marked points
of $S_{g,n}$ are represented by hyperbolic cusps,
i.e. by hyperbolic boundary components of zero length.
Denote by $g_j$ the genus and by $n_j$ the number of boundary
components (including cusps) of each surface $S_j$,
where $j=1,\dots,s$. To each boundary component
of each surface
$S_j$ we have assigned a length variable which is
equal to $x_i$ if the boundary component comes from
the cut along $\gamma_i$, or is equal to $0$ if the boundary
component comes from a cusp (one of the $n$ marked points).
Consider the corresponding Mirzakhani--Weil--Petersson
volume polynomial $V_{g_j,n_j}(x)$
and define
\begin{equation}
\label{eq:Vol:Mgn:gamma:x}
\Vol_{\mathrm{WP}}(\cM_{g,n}(\gamma,x))
:=\frac{1}{N(\gamma)}
\prod_{j=1}^s V_{g_j,n_j}(x)\,,
\end{equation}
(see formula~(4.1) on page~113 of~\cite{Mirzakhani:grouth:of:simple:geodesics}).
Denote by $(2d_1,\dots,2d_k)_\gamma$ the coefficient
of $x^{2d_1}\cdots x^{2d_k}$ in this polynomial and
let
\begin{equation}
\label{eq:b:gamma:5:3}
b_\gamma(2d_1,\dots,2d_k):= (2d_1,\dots,2d_k)_\gamma
\frac{\prod_{i=1}^k(2d_i+1)!}{(6g-6+2n)!}\,,
\end{equation}
(see equation~(5.3) on page~118 in~\cite{Mirzakhani:grouth:of:simple:geodesics}).

Now everything is ready to state Mirzakhani's result and to
prove Theorem~\ref{th:our:density:equals:Mirzakhani:density}.
\medskip

\noindent
\textbf{Mirzakhani's results and proof of comparison theorems.}
Let $\gamma=\gamma_1+\dots+\gamma_k$ be a
primitive simple closed multicurve as above. Let
$\gamma_a=\sum_{i=1}^k a_i\gamma_i$, where $a_i\in\N$ for
$i=1,\dots,k$. (To follow original notation of Mirzakhani,
we denote integer weights of components of a multicurve by
$a_i$ and not by $H_i$ as before.)

\begin{NNTheorem}[Mirzakhani~\cite{Mirzakhani:grouth:of:simple:geodesics}]
The frequency $c(\gamma_a)$ of a multi-curve
$\gamma_a=\sum_{i=1}^k a_i\gamma_i$ is equal to
\begin{equation}
\label{eq:c:gamma}
c(\gamma_a)
=\frac{2^{-M(\gamma)}}{|\Sym(\gamma)|}
\cdot\sum_{\substack{d\\|d|=3g-3+n-k}}
\frac{b_\gamma(2d_1,\dots,2d_k)}{a_1^{2d_1+2}\dots a_k^{2d_k+2}}\,.
\end{equation}
The expression $b_{g,n}$ given by the
integral~\eqref{eq:b:g:n} can be represented as
the sum
\begin{equation}
\label{eq:bgn:as:sum:B:gamma}
b_{g,n}=\sum_{\gamma\in\cS_{g,n}} B_\gamma\,,
\end{equation}
where for $\gamma=\bigsqcup_{i=1}^k \gamma_i$
one has
\begin{equation}
\label{eq:B:gamma}
B_\gamma
=\frac{2^{-M(\gamma)}}{|\Sym(\gamma)|}
\cdot\sum_{|d|=3g-3+n-k}
b_\gamma(2d_1,\dots,2d_k)\prod_{i=1}^k\zeta(2s_i+2)\,.
\end{equation}
\end{NNTheorem}

\begin{proof}[Proof of Theorem~\ref{th:our:density:equals:Mirzakhani:density}]
Let $\gamma=\gamma_1+\dots+\gamma_k$
be a primitive simple closed multicurve in $\cS_{g,n}$.
Let $\Gamma$ be the associated decorated graph.
The sets $\cS_{g,n}$ and $\cG_{g,n}$ are
in the canonical one-to-one correspondence.

Defining $M(\gamma)$ as in~\eqref{eq:M:gamma} we see that
$M(\gamma)$ is the number of factors $V_{1,1}$ in the
product~\eqref{eq:Vol:Mgn:gamma:x}, i.e. the number of indices
$j$ in the range $\{1,\dots,s\}$ such that
$(g_j,n_j)=(1,1)$. Thus, the factor $2^{-M(\gamma)}$ in
Equations~\eqref{eq:c:gamma} and~\eqref{eq:B:gamma}
compensates the difference in
normalizations~\eqref{eq:Vol:M11:Mirzakhani}
and~\eqref{eq:Vol:M11:Kazarian}, see
Remark~\ref{rm:Kazarian}. Hence,
applying~\eqref{eq:Vgn:through:Ngn} to the right-hand side
of~\eqref{eq:Vol:Mgn:gamma:x} we get
$$
2^{-M(\gamma)}\prod_{j=1}^s V_{g_j,n_j}(x)
=
\left(\prod_{j=1}^s N_{g_j,n_j}(x)\right)\cdot
\left(\prod_{j=1}^s 2^{2g_j-3+n_j}\right)\,.
$$
For a decorated graph $\Gamma$ one has
\begin{multline*}
g=1-\chi(\Gamma)+\sum_{v_j\in V(\Gamma)} g(v_j)
=1-|V(\Gamma)|+|E(\Gamma)|+\sum_{j=1}^s g_j
=\\=
1+\sum_{j=1}^s \left(g_j -1 +\frac{n_j}{2}\right)
-\frac{(\text{total number of half-edges (legs)})}{2}
\,.
\end{multline*}
(By convention $E(\Gamma)$ denotes the set of ``true''
edges of $\Gamma$ versus $n$ half-edges (legs)
corresponding to $n$ marked points on $S_{g,n}$.) Hence,
$$
\sum_{j=1}^s (2g_j-3+n_j)
=2g-2+n
-|V(\Gamma)|
=2g-3+n
-(|V(\Gamma)|-1)\,.
$$

Note also that by definition
$$
|\Aut(\Gamma)|=|\Sym(\gamma)|\cdot N(\gamma)\,.
$$
We have proved that
\begin{multline}
\label{eq:two:polynomials}
\frac{2^{-M(\gamma)}}{|\Sym(\gamma)|\cdot N(\gamma)}
\prod_{j=1}^s V_{g_j,n_j}(x)
=\\=
2^{2g+n-3}\cdot
\left(\frac{1}{2^{|V(\Graph)|-1}} \cdot
\frac{1}{|\operatorname{Aut}(\Graph)|}
\cdot
\prod_{v\in V(\Graph)}
N_{g_v,n_v}(x)\right)\,.
\end{multline}

Let us prove now Relation~\eqref{eq:Vol:gamma:c:gamma}.
Expressions~\eqref{eq:b:gamma:5:3} and~\eqref{eq:c:gamma}
represent $c(\gamma_{\boldsymbol{H}})$ as the sum of terms
constructed using the polynomial in the left-hand side
of~\eqref{eq:two:polynomials}.
Expression~\eqref{eq:volume:contribution:of:stable:graph}
represents $\Vol\big(\Gamma,\boldsymbol{H}\big)$ as the
sum of the corresponding terms constructed using the
proportional polynomial in the brackets on the right-hand
side of~\eqref{eq:two:polynomials}. The only difference
between the two sums comes from the global normalization
factors shared by all terms of the sums. Namely,
Expression~\eqref{eq:b:gamma:5:3} has an extra factor $(6g-6+2n)!$ in
the denominator, while the first line of
Expression~\eqref{eq:P:Gamma} has the extra global factor
$\cfrac{2^{6g-5+2n} \cdot (4g-4+n)!}{(6g-7+2n)!}$. Taking
into consideration the coefficient of proportionality
$2^{2g+n-3}$ relating the two polynomials
in~\eqref{eq:two:polynomials} we get
$$
(6g-6+2n)!\cdot c(\gamma_a)=2^{2g+n-3}\cdot
\left(\cfrac{2^{6g-5+2n} \cdot (4g-4+n)!}{(6g-7+2n)!}\right)^{-1}
\cdot\Vol\big(\Gamma,\boldsymbol{H}\big)\,,
$$
and~\eqref{eq:Vol:gamma:c:gamma} follows.

It remains to notice that by Definitions~\eqref{eq:B:gamma}
and~\eqref{eq:c:gamma} of $B_\gamma$ and $c(\gamma_a)$
respectively one has
$$
B_\gamma=\sum_{\boldsymbol{H}\in\N^k} c(\gamma_{\boldsymbol{H}})\,.
$$
Similarly, by~\eqref{eq:Vol:Q:as:sum:of:Vol:gamma}
one has
$$
\Vol(\Gamma)=\sum_{\boldsymbol{H}\in\N^k} \Vol\left(\Gamma,\boldsymbol{H})\right)\,.
$$
Since the coefficient of proportionality
$const_{g,n}$ between $\Vol\left(\Gamma,\boldsymbol{H}\right)$
and $c(\gamma_{\boldsymbol{H}})$ in~\eqref{eq:Vol:gamma:c:gamma}
is common for all $\boldsymbol{H}\in\N^k$,
Relation~\eqref{eq:Vol:gamma:c:gamma}
for the individual terms
implies the analogous relation
$\Vol(\Gamma)=const_{g,n}\cdot B_\gamma$
for the above sums of all terms
over all $\boldsymbol{H}\in\N^k$
and for the corresponding sums
$\Vol\cQ_{g,n}=const_{g,n}\cdot b_{g,n}$
(see~\eqref{eq:square:tiled:volume}
and~\eqref{eq:bgn:as:sum:B:gamma}) over all stable graphs
$\Gamma\in\cG_{g,n}$.
Theorem~\ref{th:our:density:equals:Mirzakhani:density}
and Corollary~\ref{cor:Vol:as:b:g:n} are proved.
\end{proof}


\section{Large genus asymptotics for frequencies
of simple closed curves and of one-cylinder
square-tiled surfaces}
\label{s:2:correlators}

\subsection{Universal bounds for $2$-correlators}

Following Witten~\cite{Witten} define
\begin{equation}
\label{eq:tau:correlator}
\langle \tau_{d_1}\dots\tau_{d_n}\rangle=\int_{\overline{\cM}_{g,n}}\psi^{d_1}\dots\psi^{d_n}\,,
\end{equation}
where $d_1+\dots+d_n=3g-3+n$.

Consider the following normalization of the $2$-correlators
$\langle\tau_k\tau_{3g-1-k}\rangle_g$
introduced in~\cite{Zograf:2:correlators}:
\begin{equation}
\label{eq:a:g:k}
a_{g,k}
=\frac{(2k+1)!!\cdot(6g-1-2k)!!}{(6g-1)!!}
\cdot 24^g\cdot g!
\cdot \langle\tau_k\tau_{3g-1-k}\rangle_g\,.
\end{equation}
By~(7) in~\cite{Zograf:2:correlators} under such
normalization the differences of $2$-correlators admit the
following explicit expression:
\begin{multline}
\label{eq:a:g:k:difference}
a_{g,k+1}-a_{g,k}
=\\=
\cfrac{(6g-3-2k)!!}{(6g-1)!!}
\cdot
\begin{cases}
\cfrac{(6j-1)!!}{j!}
\cdot\cfrac{(g-1)!}{(g-j)!}
\cdot(g-2j)\,,
\quad&\text{ for }k=3j-1\,,
\\ \ & \ \\
-2\cdot\cfrac{(6j+1)!!}{j!}
\cdot\cfrac{(g-1)!}{(g-1-j)!}\,,
\quad&\text{ for }k=3j\,,
\\ \ & \ \\
2\cdot\cfrac{(6j+3)!!}{j!}
\cdot\cfrac{(g-1)!}{(g-1-j)!}\,,
\quad&\text{ for }k=3j+1\,,
\end{cases}
\end{multline}
where $k=0,1,\dots,\left[\cfrac{3g-1}{2}\right]-1$, and
$a_{g,0}=1$.

It follows from Definition~\eqref{eq:a:g:k} that
$a_{g,k}=a_{g,3g-1-k}$, so
Relations~\eqref{eq:a:g:k:difference} determine $a_{g,k}$
for all $g\ge 1$ and for all $k$ satisfying $0\le k\le
3g-1$.

\begin{Proposition}
\label{pr:main:bounds}
For all $g\in\N$ and for all integer $k$ in the range
$\{2,3,\dots,3g-3\}$ the following bounds are valid:
\begin{equation}
\label{eq:main:bounds}
1-\frac{2}{6g-1}=a_{g,1} =a_{g,3g-2}
< a_{g,k}
< a_{g,0}=a_{g,3g-1}=1\,.
\end{equation}
\end{Proposition}
Proposition~\ref{pr:main:bounds} is proved in Appendix~\ref{a:proof:2:correlators}.

\subsection{Asymptotic volume contribution of one-cylinder
square-tiled surfaces.}
\label{s:simple:closed:geodesics:asymptotics}

In this section we compute the large genus asymptotics for
the contributions of the stable graphs having a single edge
to the Masur--Veech volume $\Vol\cQ_g$. In other words, we compute
the asymptotic contributions of one-cylinder square-tiled surfaces to the
Masur--Veech volume of the principal stratum $\cQ(1^{4g-4})$.
These contributions obviously provide lower bounds for $\Vol\cQ_g$.

Denote by $\Graph_1(g)$ the stable graph having a single
vertex and having a single edge and such that the vertex is
decorated with integer $g-1$
(see~Figure~\ref{fig:non:separating}). The single edge
forms a loop, so this stable graph represents a surface of
genus $g$ without marked points. The stable graph
$\Graph_1(g)$ encodes the orbit of a simple closed
nonseparating curve on a surface of genus $g$.

\begin{figure}[htb]
\includegraphics{genus_two_graph_11.eps}
\begin{picture}(0,0)(158,10) 
\put(44,-14){$g-1$}
\end{picture}
\includegraphics{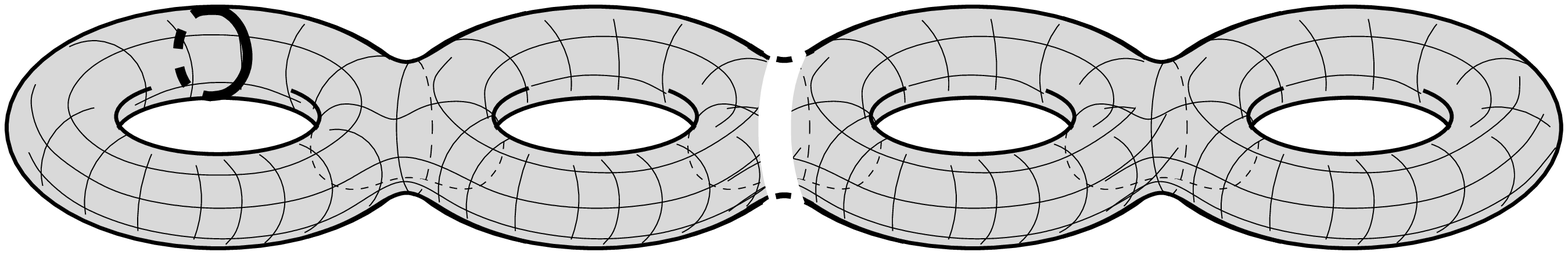}
\vspace*{40pt}
\caption{
\label{fig:one:vertex:one:loop}
The graph $\Graph_1(g)$ on the left
represents a non-separating simple closed curve on a surface
of genus $g$.
}
\label{fig:non:separating}
\end{figure}

\begin{Theorem}
\label{th:contribution:Gamma:1}
Consider the stable graph $\Graph_1(g)$ having a single
vertex decorated with label $g-1$ and single edge (see
Figure~\ref{fig:non:separating}). The contribution
$\Vol\Graph_1(g)$ of $\Graph_1(g)$ to the Masur--Veech
volume $\Vol\cQ_g$ of the moduli space $\cQ_g$ of
holomorphic quadratic differentials on complex curves of
genus $g$ has the following asymptotics:
\begin{equation}
\label{eq:asymptotic:contribution:Gamma:1}
\Vol\Graph_1(g)
=\sqrt{\frac{2}{3\pi g}}
\cdot\left(\frac{8}{3}\right)^{4g-4}
\cdot\left(1+O\left(\frac{1}{g}\right)\right)
\quad\text{as }g\to+\infty\,.
\end{equation}
\end{Theorem}
\begin{proof}
It would be slightly more convenient to shift $g$ by $1$.
Assign variable $b_1$ to the only edge of the graph.
Formula~\eqref{eq:volume:contribution:of:stable:graph} from
Theorem~\ref{th:volume} applied to the graph
$\Graph_1(g+1)$ gives
\begin{multline}
\label{eq:contribution:graph:G1:initial}
\left(\frac{(4g)!}{(6g)!}\cdot 2^{6g}\cdot 12g\right) \cdot
\frac{1}{2}\cdot 1\cdot
b_1 \cdot
N_{g,2}(b_1,b_1)
\\
\xmapsto{\ \cZ\ }
\Vol\Graph_1(g+1)
=(4g)!\cdot2^{g+2}\cdot\zeta(6g)
\cdot\sum_{d_1+d_2=3g-1}
\frac{\langle\psi_1^{d_1}\psi_2^{d_2}\rangle}{d_1!\cdot d_2!}
=\\=
(4g)!\cdot2^{g+2}\cdot\zeta(6g)
\cdot\sum_{k=0}^{3g-1}
\frac{\langle\tau_k\tau_{3g-1-k}\rangle_g}{k!\cdot(3g-1-k)!}
\,.
\end{multline}

Now pass to the normalization~\eqref{eq:a:g:k}
of the correlators
$\langle\tau_k\tau_{3g-1-k}\rangle_g$:
\begin{equation*}
a_{g,k}
=\frac{(2k+1)!!\cdot(6g-1-2k)!!}{(6g-1)!!}
\cdot 24^g\cdot g!
\cdot \langle\tau_k\tau_{3g-1-k}\rangle_g\,.
\end{equation*}
By Proposition~\ref{pr:main:bounds}, the $2$-correlators
admit the following uniform bounds under such
normalization:
$$
1-\frac{2}{6g-1}
\le a_{g,k}
\le 1\,,\quad\text{for }k=0,1,\dots,3g-1\,.
$$

Rewriting the Expression~\eqref{eq:contribution:graph:G1:initial}
for $\Vol\Graph_1(g+1)$ in terms of $a_{g,k}$
we get
\begin{multline}
\label{eq:contribution:graph:G1:a:g:k:init}
\Vol\Graph_1(g+1)/\zeta(6g)
=\\=
(4g)!\cdot2^{g+2}
\cdot\frac{(6g-1)!!}{24^g\cdot g!}
\cdot\sum_{k=0}^{3g-1}
\frac{a_{g,k}}
{k!\cdot(2k+1)!!\cdot(3g-1-k)!\cdot(6g-1-2k)!!}\,.
\end{multline}
Passing from double factorials to factorials,
\begin{align*}
(6g-1)!! &= \frac{(6g)!}{(3g)!\cdot 2^{3g}}\\
(2k+1)!! &= \frac{(2k+1)!}{k!\cdot 2^k}\\
(6g-1-2k)!! &=\frac{(6g-1-2k)!}{(3g-1-k)!\cdot 2^{3g-k-1}}\,,
\end{align*}
we rewrite and simplify Expression~\eqref{eq:contribution:graph:G1:a:g:k:init}
as follows:
\begin{multline}
\label{eq:contribution:graph:G1:a:g:k:intermediate}
\Vol\Graph_1(g+1)/\zeta(6g)
=\\=
(4g)!\cdot 2^{g+2}
\cdot\frac{1}{3^g\cdot 2^{3g}}
\cdot\frac{1}{g!\cdot(3g)!}
\cdot\frac{1}{2}
\sum_{k=0}^{3g-1}
\binom{6g}{2k+1}\cdot a_{g,k}
=\\=
\frac{(4g)!}{g!\cdot(3g)!}
\cdot\frac{1}{3^g\cdot 2^{2g}}
\cdot 2\sum_{k=0}^{3g-1}
\binom{6g}{2k+1}\cdot a_{g,k}\,.
\end{multline}
Taking the difference of binomial expansions of
the left-hand sides of the
identities $(1-1)^{2n}=0$ and $(1+1)^{2n}=2^{2n}$
we derive the classical identity
$$
\sum_{k=0}^{n-1} \binom{2n}{2k+1}= 2^{2n-1}\,.
$$
Combining the latter identity evaluated
for $n=3g$ with
bounds~\eqref{eq:main:bounds} we get the following bounds
for $\Vol\Graph_1(g+1)$:
\begin{equation}
\label{eq:contribution:graph:G1:bounds}
\binom{4g}{g}
\cdot\left(\frac{2^4}{3}\right)^g
\cdot\left(1-\frac{2}{6g-1}\right)
\le
\Vol\Graph_1(g+1)/\zeta(6g)
\le
\binom{4g}{g}
\cdot\left(\frac{2^4}{3}\right)^g
\,.
\end{equation}
Note that
$$
\zeta(6g)\to 1\quad\text{as }g\to+\infty
$$
and the convergence is exponentially fast.

Applying Stirling's formula to the factorials
in the binomial coefficient
$\binom{4g}{g}$ in the latter expression
we get
\begin{equation}
\label{eq:binom:4g:g}
\binom{4g}{g}
=\sqrt{\frac{2}{3\pi g}}
\cdot\left(\frac{2^8}{3^3}\right)^g
\cdot\left(1+O\left(\frac{1}{g}\right)\right)
\quad\text{as }g\to+\infty\,.
\end{equation}
Combining the latter equality with
bounds~\eqref{eq:contribution:graph:G1:bounds}
we get the desired Formula~\eqref{eq:asymptotic:contribution:Gamma:1}
in genus $g+1$:
$$
\Vol\Graph_1(g+1)
=\sqrt{\frac{2}{3\pi g}}
\cdot\left(\frac{8}{3}\right)^{4g}
\cdot\left(1+O\left(\frac{1}{g}\right)\right)
\quad\text{as }g\to+\infty\,.
$$
\end{proof}

\begin{Remark}
\label{rm:potential:detalisation}
Actually, we have a very good control on the asymptotic
expansions of correlators $a_{g,k}$ in powers of
$\frac{1}{g}$, so it would not be difficult to
specify several terms of the asymptotic expansion of
$O\left(\frac{1}{g}\right)$ in
Formula~\eqref{eq:asymptotic:contribution:Gamma:1}. We do
not do it only because we do not currently see any specific
need for a more precise expression.
\end{Remark}

\begin{proof}[Proof of Theorem~\ref{th:asymptotic:lower:bound}]
Inequality~\eqref{eq:asymptotic:lower:bound} now follows from~\eqref{eq:contribution:graph:G1:bounds},
where we use the following estimates for the factorials
involved in the binomial coefficient. By Theorem~1.6 in~\cite{Batir}
for all positive real numbers $x\ge 1$ one has
$$
x^x e^{-x}\sqrt{2\pi(x+a)} < \Gamma(x+1) <
x^x e^{-x}\sqrt{2\pi(x+b)}\,,
$$
where $a = 1/6 = 0.1666666\dots$ and
$b = \frac{e^2}{2\pi}-1 = 0.176005\dots$.
\end{proof}

We proceed with the remaining graphs having a single edge.
This time it has two vertices joined by the edge as in
Figure~\ref{fig:separating}. The two vertices are decorated
with strictly positive integers $g_1,g_2\in\N$ such that
$g_1+g_2=g$ (see~Figure~\ref{fig:non:separating}). Without
loss of generality we may assume that $g_1\le g_2$. This
stable graph encodes the orbit of a simple closed curve
separating the compact surface of genus $g$ without
punctures into subsurfaces of genera $g_1$ and $g_2$.

\begin{figure}[htb]
\includegraphics{genus_two_graph_12.eps}
\begin{picture}(0,0)(145,-5)
\put(-2,-8){$g_1$}
\put(31,-8){$g_2$}
\end{picture}
\includegraphics{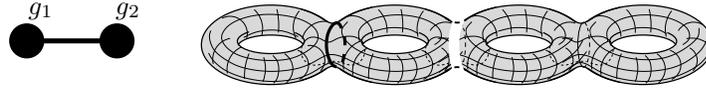}
\vspace*{40pt}
\caption{
\label{fig:two:verticies:one:edge}
The graph $\Separating(g_1,g_2)$
on the left represents a simple closed curve
on a surface of genus $g$ separating the surface
into surfaces of genera $g_1$ and $g_2$, where $g_1+g_2=g$
(on the right).
}
\label{fig:separating}
\end{figure}

\begin{Proposition}
\label{pr::contribution:Gamma:1}
Consider the stable graph $\Separating(g_1,g-g_1)$ having a
single edge joining two vertices decorated with labels
$g_1$ and $g-g_1$ respectively (see
Figure~\ref{fig:separating}). The contribution
$\Vol(\Separating(g_1,g-g_1))$ of $\Separating(g_1,g-g_1)$
to the Masur--Veech volume $\Vol\cQ_g$ is
\begin{equation}
\label{eq:asymptotic:contribution:segment}
\Vol(\Separating(g_1,g-g_1))
=\frac{4\cdot\zeta(6g-6)}{|\operatorname{Aut}(\Separating(g_1,g-g_1))|}
\cdot \binom{4g-4}{g}
\cdot\frac{1}{12^g}
\cdot \binom{g}{g_1}
\cdot \binom{3g-4}{3g_1-2}
\,.
\end{equation}
\end{Proposition}
\begin{proof}
Let $g_2=g-g_1$.
The contribution of the graph $\Separating(g_1,g_2)$ is given by
Formula~\eqref{eq:square:tiled:volume} from
Theorem~\ref{th:volume}:
\begin{equation}
\label{eq:Vol:Graph:g1:g2}
\Vol(\Separating(g_1,g_2))
=\frac{2^{6g-5} \cdot (4g-4)!}{(6g-7)!}\cdot \\
\frac{1}{2} \cdot
\frac{1}{|\operatorname{Aut}(\Graph(g_1,g_2))|}
\cdot
\cZ\big(b\cdot
N_{g_1,1}(b)\cdot N_{g_2,1}(b)
\big)\,,
\end{equation}
where
\begin{equation}
\label{eq:Aut}
|\operatorname{Aut}(\Graph(g_1,g_2))|=
\begin{cases}
2,&\text{when }g_1=g_2\,,\\
1,&\text{otherwise}\,.
\end{cases}
\end{equation}

By the result of E.~Witten~\cite{Witten} one has
the following closed expression for $1$-correlators:
$$
\langle \psi_1^{3g-2} \rangle
=
\int_{\overline{\cM}_{g,1}} \psi_1^{3g-2}
=
\frac{1}{24^g\cdot g!}\,.
$$

Applying Definitions~\eqref{eq:N:g:n}
and~\eqref{eq:c:subscript:d} to $N_{g,1}(b)$
and using the above expression
for $\langle \psi_1^{3g-2} \rangle$
we get the following closed form expression
for the polynomial $N_{g,1}(b)$:
\begin{multline}
N_{g,1}(b)=c_{3g-2} b^{2(3g-2)}=
\frac{1}{2^{5g-6+2}\cdot(3g-2)!}
\cdot\langle \psi_1^{3g-2} \rangle \cdot b^{6g-4}
=\\=
\frac{1}{2^{8g-4}\cdot 3^g\cdot g!\cdot (3g-2)!}
\cdot b^{6g-4}\,.
\end{multline}

Using the Definition~\eqref{eq:cZ} of $\cZ$
and the assumption $g_1+g_2=g$
we can
now develop the rightmost factor
in~\eqref{eq:cZ} as follows
\begin{multline*}
\cZ\big(b\cdot N_{g_1,1}(b)\cdot N_{g_2,1}(b)\big)
=\\=
\cZ\left(
\frac{1}{2^{8g_1-4}\!\cdot\! 2^{8g_2-4}}
\!\cdot\!\frac{1}{3^{g_1}\!\cdot\! 3^{g_2}}
\!\cdot\! \frac{1}{g_1!\, g_2!}
\!\cdot\! \frac{1}{(3g_1-2)! (3g_2-2)!}
\!\cdot\! b\cdot b^{6g_1-4}\cdot b^{6g_2-4}\right)
=\\=
\frac{1}{2^{8g-8}}
\cdot\frac{1}{3^g}
\cdot \frac{1}{g_1!\, (g-g_1)!}
\cdot \frac{1}{(3g_1-2)! (3g-3g_1-2)!}
\cdot(6g-7)!\cdot\zeta(6g-6)
\,.
\end{multline*}
Plugging this expression into~\eqref{eq:Vol:Graph:g1:g2};
multiplying and dividing by $g!$ and by $(3g-4)!$ to
pass to binomial coefficients
and simplifying we get the desired
Formula~\eqref{eq:asymptotic:contribution:segment}.
\end{proof}

We complete this section with the computation of the cumulative
contribution of the graphs $\Separating(g_1,g-g_1)$ to
$\Vol\cQ_g$ coming from all
$g_1=1,\dots,\left[\frac{g}{2}\right]$.

\begin{Proposition}
\label{pr:sum:of:contributions:of:separating:graphs}
As $g\to\infty$, we have the following asymptotic relation
\begin{equation}
\label{eq:sum:of:contributions:of:separating:graphs}
\sum_{g_1=1}^{\left[\frac{g}{2}\right]}
\Vol(\Separating(g_1,g-g_1))
\sim
\frac{2}{3\pi g}\cdot\frac{1}{4^g}
\cdot\left(\frac{8}{3}\right)^{4g-4}\,,
\end{equation}
\end{Proposition}
\begin{proof}
Applying Formula~\eqref{eq:Vol:Graph:g1:g2},
making the summation index range from $1$ to $g-1$
(instead of up to $\left[\frac{g}{2}\right]$)
and taking into consideration that when $g=2g_1$,
the term $\Vol\Separating(g_1,g_1)$ has
$|\operatorname{Aut}(\Separating(g_1,g_1))|=2$
(see Equation~\eqref{eq:Aut}) we get
$$
\sum_{g_1=1}^{\left[\frac{g}{2}\right]}
\Vol(\Separating(g_1,g-g_1))
=
2\cdot\zeta(6g-6)
\cdot \binom{4g-4}{g}
\cdot\frac{1}{12^g}
\cdot
\sum_{g_1=1}^{g-1}
\binom{g}{g_1}
\cdot \binom{3g-4}{3g_1-2}
\,.
$$

The zeta value $\zeta(6g-6)$ tends to $1$ exponentially
fast when $g\to+\infty$. Stirling's formula provides the
following asymptotic value of the binomial coefficient
$$
\binom{4g-4}{g}
\sim
\frac{2}{\sqrt{6\pi g}}
\cdot 3^g\cdot\left(\frac{4}{3}\right)^{4g-4}\,.
$$
Thus, to complete the proof of Formula~\eqref{eq:log:sep:over:non:sep}
and, thus, of Theorem~\ref{th:separating:over:non:separating},
it remains to prove the Lemma below.
\end{proof}

\begin{Lemma}
\label{lm:sum:of:products:of:two:binomials}
The following asymptotic formula holds
\begin{equation}
\label{eq:sum:of:products:of:binomials:asymptotics}
S(g)=\sum_{g_1=1}^{g-1}
\binom{g}{g_1}
\binom{3g-4}{3g_1-2}
\sim
\frac{1}{\sqrt{6\pi g}}\cdot 2^{4g-4}
\end{equation}
as $g\to+\infty$.
\end{Lemma}
\begin{proof}
The probability density of the normal distribution
$\cN(\mu,\sigma^2)$ is given by the function
$$
f(x\,|\,\mu,\sigma^2)=
\frac{1}{\sqrt{2\pi\sigma^2}}
\cdot e^{-\tfrac{(x-\mu)^2}{2\sigma^2}}\,.
$$

Let $g$ be a large positive integer.
By the de Moivre--Laplace theorem,
after normalization by
$2^g$ the distribution of the binomial coefficients
$\binom{g}{k}$, where $k=0,1,\dots,g$, tends to the normal
distribution
$\cN\!\!\left(\tfrac{g}{2},\tfrac{g}{4}\right)$ as
$g\to+\infty$.

Let $m$ be any positive integer which we use as a fixed
parameter. The normalized distribution of the binomial
coefficients $\binom{m\cdot g}{k}$, where $k=0,1,\dots,
m\cdot g$, tends to the normal distribution
$\cN\!\!\left(\tfrac{m\cdot g}{2},\tfrac{m\cdot
g}{4}\right)$ as $g\to\infty$. Hence, the normalized
distribution of binomial coefficients $\binom{m\cdot
g}{m\cdot k}$, where $k=0,1,\dots, g$, tends to the normal
distribution
$\cN\!\!\left(\tfrac{g}{2},\tfrac{g}{4m}\right)$ as
$g\to+\infty$. In particular, letting $m=3$ we see that the normalized distribution
of binomial coefficients $\binom{3g-4}{3k-2}$, where
$k=1,\dots,g-1$, tends to the normal distribution
$\cN\!\!\left(\tfrac{g}{2},\tfrac{g}{12}-\tfrac{1}{9}\right)$.

We have
$$
f(x\,|\,\mu,\sigma_1^2)\cdot f(x\,|\,\mu,\sigma_1^2)
=
\frac{\sqrt{2\pi\sigma^2}}{\sqrt{2\pi\sigma_1^2}\cdot\sqrt{2\pi\sigma_2^2}}\cdot f(x\,|\,\mu,\sigma^2)\,,
\quad\text{where}\quad
\sigma^2=\frac{\sigma_1^2\sigma_2^2}{\sigma_1^2+\sigma_2^2}\,.
$$
In other words,
the normalized distribution of the product of two independent normal
distributions $\cN(\mu,\sigma_1^2)$ and
$\cN(\mu,\sigma_2^2)$ sharing the same mean $\mu$ is the
normal distribution
$\cN\!\!\left(\mu,\tfrac{\sigma_1^2\sigma_2^2}{\sigma_1^2+\sigma_2^2}\right)$.
Hence, after normalization by $S(g)$, the product
$\binom{g}{g_1}\binom{3g-4}{3g_1-2}$ of the two binomial
distributions, where $g_1=1,\dots,g-1$, tends to the normal
distribution $\cN(g/2,\sigma^2)$ with
$$
\sigma^2=
\frac
{\tfrac{g}{4}\left(\tfrac{g}{12}-\tfrac{1}{9}\right)}
{\tfrac{g}{4}+\left(\tfrac{g}{12}-\tfrac{1}{9}\right)}
\sim
\frac{g}{16}\,,\quad\text{as } g\to+\infty\,.
$$
Thus, the asymptotic value of the sum $S(g)$ can be
computed as the value of the product distribution
$\binom{g}{g_1}\binom{3g-4}{3g_1-2}$ at $\mu=g/2$
multiplied by $\sqrt{2\pi\sigma^2}$, that is
\begin{equation}
\label{eq:value:at:mean:times:sqrt:2:pi:sigma}
S(g)\sim
\binom{g}{\left[\frac{g}{2}\right]}
\cdot\binom{3g-4}{3\left[\frac{g}{2}\right]-2}
\cdot\sqrt{2\pi \tfrac{g}{16}}\,.
\end{equation}

From Stirling's formula we get
$$
\binom{2m}{m}\sim \frac{2^{2m}}{\sqrt{\pi m}}
\quad\text{and}\quad
\binom{2m+1}{m}\sim \frac{2^{2m+1}}{\sqrt{\pi m}}
\,.
$$
Applying these asymptotic formulae to each of the binomial
coefficients
in~\eqref{eq:value:at:mean:times:sqrt:2:pi:sigma} we get
$$
S(g)
\sim\frac{2^g}{\sqrt{\pi \tfrac{g}{2}}}
\cdot\frac{2^{3g-4}}{\sqrt{\pi \tfrac{3g-4}{2}}}
\cdot\frac{\sqrt{\pi g}}{\sqrt{8}}
\sim
\frac{1}{\sqrt{6\pi g}}\cdot 2^{4g-4}\,.
$$

\end{proof}

\begin{Remark}
   %
The hypergeometric sum $S(g)$ on the left hand side
of~\eqref{eq:sum:of:products:of:binomials:asymptotics}
satisfies the following recursive relation obtained
applying Zeilberger's algorithm:
\begin{multline*}
S(g+2)=
2\cdot\frac{(324g^4+432 g^3+123 g^2-49 g-8)}{(6g-1)(3g+4)(3g-1)(g+1)}\cdot S(g+1)
+\\+
36\cdot\frac{(6g+5)(4g-1)(4g-3)}{(6g-1)(3g+4)(3g-1)}\cdot S(g)\,.
\end{multline*}
\end{Remark}

\subsection{Frequencies of simple closed geodesics}
\label{ss:simple:closed:geodesics}

In this section we use the setting and the notation as in
Theorem~6.1 of M.~Mirzakhani~\cite{Mirzakhani:grouth:of:simple:geodesics}
reproduced at the end of Section~\ref{ss:Square:tiled:surfaces:and:associated:multicurves}.

Recall that equivalence classes of smooth simple closed
curves on a compact oriented surface of genus $g$ without
boundary are classified as follows. The curve can be
separating or non-separating. All non-separating curves as
in Figure~\ref{fig:non:separating} belong to the same
class; we denote the corresponding frequency by
$c_{g,nonsep}$.

Separating simple closed curves are classified by the
genera $g_1,g_2$ of components in which the curve separates
the surface; see Figure~\ref{fig:separating}. Here
$g_1+g_2=g$; $g_1,g_2\ge 1$, and pairs $g_1,g_2$ and
$g_2,g_1$ correspond to the same equivalence class.
Denote a simple closed curve of this type
by $\gamma_{g_1,g_2}$. The
stable graph corresponding to $\gamma_{g_1,g_2}$ is
$\Separating(g_1,g_2)$, see Figure~\ref{fig:separating}.

Recall that the volume contribution $\Vol(\Separating(g_1,g_2))$
comes from all one-cylinder square-tiled surfaces of genus
$g=g_1+g_2$ such that the waist curve of the single cylinder
separates the surface into two surfaces of genera $g_1$ and
$g_2$ respectively. This single horizontal cylinder can be composed of
$a=1,2,\dots$ horizontal bands of squares.
The contribution $\Vol(\Separating(g_1,g_2))$
is the sum of contributions $\Vol(\Separating(g_1, g_2), H)$
of square-tiled surfaces having a fixed value $H\in\N$,
$$
\Vol(\Separating(g_1,g_2))
=\sum_{H=1}^{+\infty} \Vol(\Separating(g_1,g_2), H)\,.
$$
Recall also, that
$$
\Vol(\Separating(g_1,g_2), H)
=\frac{1}{H^d} \Vol(\Separating_{g_1,g_2})\,,
$$
where $d=6g-6=\dim\cQ_g$, which implies that
$$
\label{eq:Vol:Gamma:g1:g2:1}
\Vol(\Separating(g_1,g_2))
=\zeta(6g-6)\cdot\Vol(\Separating_{g_1,g_2}, 1)\,.
$$
Combining the latter relation with
Formula~\eqref{eq:Vol:gamma:c:gamma} from
Theorem~\ref{th:our:density:equals:Mirzakhani:density} we get
$$
\frac{\Vol(\Separating(g_1,g_2))}{\zeta(6g-6)}
=2\cdot(6g-6)\cdot
(4g-4)!\cdot 2^{4g-3}\cdot
c(\gamma_{g_1,g_2})\,.
$$
Applying
the Expression~\eqref{eq:asymptotic:contribution:segment} for
$\Vol\Graph(g_1,g_2)$ we get the following formula:
\begin{multline}
\label{eq:Mirzakhani:separating}
c(\gamma_{g_1,g-g_1})=
\frac{1}{|\Aut\Separating(g_1,g-g_1)|}
\cdot\\
\cdot\frac{1}{
2^{3g-4}
\cdot 24^g\cdot
g_1!\cdot (g-g_1)!\cdot (3g_1-2)!\cdot(3(g-g_1)-2)!\cdot (6g-6)
}\,.
\end{multline}

In this way we reproduce the formula for the frequency of
simple closed separating geodesics first proved by
M.~Mirzakhani (see page~124
in~\cite{Mirzakhani:grouth:of:simple:geodesics}).

\begin{Remark}
The formula on page~124
in~\cite{Mirzakhani:grouth:of:simple:geodesics} contains
two misprints: the power in the first factor in the
denominator is indicated as $2^{3g-2}$ while it should be
read as $2^{3g-4}$ and the fifth factor is indicated as
$(3g-2)!$ while it should be read as $(3i-2)!$.
Indeed, following Mirzakhani's calculation we have to use
formula~(5.5)
from~\cite{Mirzakhani:grouth:of:simple:geodesics} for
$c(\gamma)$. Using notation of this formula applied to our
particular $\gamma$ we have $n=0$, $k=1$, $a_1=1$,
$s_1=3g-4$. Mirzakhani assumes for simplicity that $g>2i>2$,
which implies that $M(\gamma)=0$ and
$|\operatorname{Sym}(\gamma)|=1$
(and implies that $|\Aut\Separating(g_1,g-g_1)|=1$ in notation
of Formula~\eqref{eq:Mirzakhani:separating} above).
Thus, applying
formula~(5.5)
from~\cite{Mirzakhani:grouth:of:simple:geodesics}
Mirzakhani gets
$$
c(\gamma)=b_\Gamma(2\cdot(3g-4))
=(2\cdot(3g-4))_\Gamma\cdot\frac{(2\cdot(3g-4)+1)!}{(6g-6)!}\,,
$$
where $b_\Gamma(2s_1)$ and $(2s_1)_\Gamma$ are introduced
in~\cite{Mirzakhani:grouth:of:simple:geodesics}
in formula~(5.3) and in the line above it
respectively.

In order to evaluate $(2\cdot(3g-4))_\Gamma$ we compute following
Mirzakhani the product of the coefficients of the leading
terms of the polynomials $V_{i,1}(x)$ and $V_{g-i,1}(x)$
(see the bottom of page~123
in~\cite{Mirzakhani:grouth:of:simple:geodesics}). In this
way we get
\begin{multline*}
(2\cdot(3g-4))_\Gamma
=\\=
\frac{1}{(3i-2)!\cdot i!\cdot 24^i\cdot 2^{3i-2}}
\cdot
\frac{1}{(3(g-i)-2)!\cdot (g-i)!\cdot 24^{g-i}\cdot 2^{3(g-i)-2}}
=\\=
\frac{1}{2^{3g-4}\cdot 24^g\cdot i!\cdot (g-i)!\cdot (3i-2)!\cdot (3(g-i)-2)!}\,,
\end{multline*}
compare to~\eqref{eq:Mirzakhani:separating} replacing
$g_1$ with $i$. The remaining last factor $(6g-6)$ in the
denominator of the formula of Mirzakhani comes from
$\frac{(2\cdot(3g-4)+1)!}{(6g-6)!}=\frac{1}{6g-6}$.
\end{Remark}

\begin{Remark}
The computation of $c(\alpha_2)$ on page~123
in~\cite{Mirzakhani:grouth:of:simple:geodesics},
where $\alpha_2$ denotes a separating simple closed curve
on a surface of genus $2$,
contains some misprints. It is correctly written
that $N(\alpha_2)=2$. However, the factor $\frac{1}{N(\alpha_2)}$
involved in the definition (4.1) of $\Vol_{\mathrm{WP}}(\cM_{g,n}(\Gamma,\boldsymbol{x}))$
on page 113~of~\cite{Mirzakhani:grouth:of:simple:geodesics} is missing
in the formula for $\Vol(\cM(S(\alpha_2),\ell_{\alpha_2}=x))$ on page~123.
We assume that this formula should be read as
$$
\Vol\big(\cM(S(\alpha_2),\ell_{\alpha_2}=x)\big)
=\frac{1}{N(\alpha_2)}\cdot V_{1,1}(x)\times V_{1,1}(x)
=\frac{1}{2}\cdot\left(\frac{x^2}{24}+\frac{\pi^2}{6}\right)^2\,.
$$
Also, the factor $2^{-M(\alpha_2)}$ present in the general formula~(5.5)
on page~118 in~\cite{Mirzakhani:grouth:of:simple:geodesics}
is missing in the computation of $c(\alpha_2)$. Finally, in this
particular case (and only in this case) either $M(\alpha_2)$ should be readjusted
as $M(\alpha_2)=2$, as we suggest in Remark~\ref{rm:one:handle},
or one has to redefine the symmetry group taking into consideration the
hyperelliptic involutions, as suggested in
footnote~2 on pages 12--13 in~\cite{Wright}. We get the
following value for $c(\alpha_2)$:
$$
c(\alpha_2)=c(\gamma_{1,1})=\frac{1}{8\times 24\times 24\times 6}\,.
$$

In the above computation we followed the normalization
conventions chosen by M.~Mirzakhani
in~\cite{Mirzakhani:grouth:of:simple:geodesics}. As it was
pointed out in Remark~\ref{rm:11:and:20}, the case
$(g,n)=(2,0)$ admits an alternative normalization.
Nevertheless, such an alternative normalization changes
$c(\alpha_1)$ and $c(\alpha_2)$ by common scaling factor
and, hence, does not affects the ratio:
$$
\frac{c(\alpha_1)}{c(\alpha_2)}
=\frac{c(\gamma_{nonsep,2})}{c(\gamma_{1,1})}
=48\,.
$$
This value is independently confirmed in~\cite{Bell} experimentally and in~\cite{Arana:Herrera:square:tiled}
and in~\cite{Erlandsson:Souto} theoretically.
\end{Remark}

Denote
by $c(\gamma_{sep,g})$ the sum of the frequencies
$c(\gamma_{g_1,g_2})$ over all equivalence classes
of separating curves, i.e.
over all unordered pairs $(g_1,g_2)$ satisfying
$g_1+g_2=g$; $g_1,g_2\ge 1$.
We are now ready to prove
Theorem~\ref{th:separating:over:non:separating}.

\begin{proof}[Proof of Theorem~\ref{th:separating:over:non:separating}]
By Theorem~\ref{th:our:density:equals:Mirzakhani:density}
we have
$$
\frac{c_{g,sep}}{c_{g,nonsep}}=
\frac{\sum_{g_1=1}^{\left[\frac{g}{2}\right]}
\Vol(\Separating(g_1,g-g_1))}
{\Vol(\Graph_1(g))}\,.
$$

Plugging the asymptotic
values~\eqref{eq:asymptotic:contribution:Gamma:1}
and~\eqref{eq:sum:of:contributions:of:separating:graphs}
respectively in the denominator and numerator of the ratio
on the right hand side we obtain the desired asymptotic
value for the ratio on the left hand side.
\end{proof}

In the table below we present the exact (first line) and
approximate (second line) values of the ratio
$\tfrac{c_{g,sep}}{c_{g,nonsep})}$ of the
two frequencies in small genera and the value given by the
asymptotic Formula~\eqref{eq:log:sep:over:non:sep} (third
line).

$$
\begin{array}{c|c|c|c|c||c}
g&2&3&4&5&11\\
\hline &&&&&\\
[-\halfbls]
\text{\small Exact}
&\frac{1}{48}
&\frac{5}{1776}
&\frac{605}{790992}
&\frac{4697}{27201408}
&\frac{166833285883}{5360555755385245488}
\\ &&&&& \\
[-\halfbls]
\hline &&&&&\\
[-\halfbls]
\text{\small Approximate}
& \scriptstyle 2.08\cdot 10^{-2}
& \scriptstyle 2.82\cdot 10^{-3}
& \scriptstyle 7.65\cdot 10^{-4}
& \scriptstyle 1.73\cdot 10^{-4}
& \scriptstyle 3.11\cdot 10^{-8}
\\ &&&&& \\
[-\halfbls]
\hline &&&&&\\
[-\halfbls]
\text{\small Asymp. formula}
& \scriptstyle 2.03\cdot 10^{-2}
& \scriptstyle 4.16\cdot 10^{-3}
& \scriptstyle 9.00\cdot 10^{-4}
& \scriptstyle 2.01\cdot 10^{-4}
& \scriptstyle 3.31\cdot 10^{-8}
\end{array}
$$


\appendix

\section{Proof of the asymptotic formula for $2$-correlators.}
\label{a:proof:2:correlators}

In this Appendix we prove Proposition~\ref{pr:main:bounds}.
\medskip

\noindent\textbf{Structure of the proof.~}
The equality $a_{g,1}=1-\frac{2}{6g-1}$ immediately
follows from the fact that $a_{g,0}=1$ and the recursive
relations~\eqref{eq:a:g:k:difference}. For $g=1$
bounds~\eqref{eq:main:bounds} are trivial. The
symmetry $a_{g,3g-1-k}=a_{g,k}$ allows us to confine $k$ to
the range $\{2,3,\dots,\left[\frac{3g-1}{2}\right]\}$.

Using recursive relations~\eqref{eq:a:g:k:difference} we
evaluate $a_{g,k}$ explicitly for $k=2,\dots,5$ and prove
in Lemma~\ref{lm:k:up:to:5}
bounds~\eqref{eq:main:bounds} for these small values
of $k$. In genera $g=2,3,4$, the expression
$\left[\frac{3g-1}{2}\right]$ is bounded by $5$ which
implies bounds~\eqref{eq:main:bounds} for any $k$ when
$g=2,3,4$. From this point we always assume that $g\ge 5$
and $k$ is in the range
$\{6,\dots,\left[\frac{3g-1}{2}\right]\}$.

We start the
main part of the proof by rewriting the recurrence
Relations~\eqref{eq:a:g:k:difference} in a form
convenient for estimates. Namely, we introduce the
function
\begin{equation}
\label{eq:R}
R(g,j)=\frac{\binom{3g}{3j} \binom{g}{j}}{\binom{6g}{6j}}
\end{equation}
and express the right-hand side of each of the recurrence
Relations~\eqref{eq:a:g:k:difference} as a product
$R\cdot\frac{P_i}{Q}$, where $P_i$, $i=1,2,3$, and $Q$
are explicit polynomials in $g$ and $j$. In
Lemma~\ref{lm:P:Q} we show that for any $g$ the absolute
value of each of the rational functions $P_i/Q$ on the
range of $j$ corresponding to $k\in
\{6,\dots,\left[\frac{3g-1}{2}\right]\}$ is bounded from
above by $1$. In Lemma~\ref{lm:R} we show that for any
fixed $g$ and $0\le j\le \left[\frac{g-1}{2}\right]$, the
expression $R(g,j)$ is monotonically decreasing as a
function of $j$. Combining these two lemmas we obtain in
Lemma~\ref{lm:difference:bounded:by:jmax:R2} the estimate
$-R(g,2)\cdot\frac{g-3}{2} \le a_{g,k}-a_{g,5} \le
R(g,2)\cdot\frac{3g-11}{3}$ valid for all $g$ and $k$ under
consideration.

We use explicit expressions for rational functions $\elow$
and $\eup$ in $a_{g,5}=1-\frac{2}{6g-1}+\elow =1-\eup$
obtained in Lemma~\ref{lm:k:up:to:5} to prove in
Lemma~\ref{lm:R2:jmax:smaller:than:epsilon} that for
$g\ge5$ the inequalities $R(g,2)\cdot\frac{g-3}{3}<\elow$
and $R(g,2)\cdot\frac{3g-11}{3} <\eup$ hold.

\subsection{Small values of $k$ and $g$}
Recall that $a_{g,0}=1$. Recursive
relations~\eqref{eq:a:g:k:difference} provide the following
first several terms $a_{g,k}$ for $k=1,2,3,4,5$, where for
each $k$ in this range we indicate the smallest value of
$g$ starting from which the corresponding equality holds:
\begin{align}
\label{eq:a:g:0}
a_{g,0}&=1
&\text{ for }g\ge 1\,,
\\
\label{eq:a:g:1}
a_{g,1}&=1-\frac{2}{6g-1}
&\text{ for }g\ge 1\,,
\\
\notag
a_{g,2}
&=1-\frac{12(g-1)}{(6g-1)(6g-3)}
&\text{ for }g\ge 2\,,
\\
\notag
a_{g,3}
&=1-\frac{3(24g^2-49g+30)}{(6g-1)(6g-3)(6g-5)}
&\text{ for }g\ge 3\,,
\\
\notag
a_{g,4}
&=1-\frac{2}{6g-1}+\frac{9(g-2)(34g-35)}{(6g-1)(6g-3)(6g-5)(6g-7)}
&\text{ for }g\ge 3\,,
\\
\label{eq:epsilon:lower:init}
a_{g,5}
&=1-\frac{2}{6g-1}
+\frac{27(68g^3-308g^2+519g-280)}{(6g\!-\!1)(6g\!-\!3)(6g\!-\!5)(6g\!-\!7)(6g\!-\!9)}
=\\
\label{eq:epsilon:upper:init}
&=1-\frac{9 (g-2) (288 g^3-780 g^2+1012 g-525)}{(6g-1)(6g-3)(6g-5)(6g-7)(6g-9)}
&\text{ for }g\ge 4\,.
\end{align}

\begin{Lemma}
\label{lm:k:up:to:5}
For all $g\in\N$ and for all $k\in\N$ satisfying
$2\le k\le \min\left(5,\left[\frac{3g-1}{2}\right]\right)$
the Relations~\eqref{eq:main:bounds} are valid:
\begin{equation*}
1-\frac{2}{6g-1}=a_{g,1} < a_{g,k} < a_{g,0}=1\,.
\end{equation*}
\end{Lemma}
\begin{proof}

By~\eqref{eq:a:g:0}
and~\eqref{eq:a:g:1} the terms $a_{g,0}$ and $a_{g,1}$ indeed have
values as claimed in the statement of
Lemma~\ref{lm:k:up:to:5}.

It follows from recurrence
relations~\eqref{eq:a:g:k:difference} that for $k$
satisfying $0\le k\le \left[\frac{3g-1}{2}\right]-1$ we
have $a_{g,k+1}<a_{g,k}$ if and only if $k\equiv
0\,(\operatorname{mod} 3)$ and we have $a_{g,k+1}>a_{g,k}$
for the remaining $k$ in this range. In particular, for
$g\ge 2$ the difference $a_{g,2}-a_{g,1}$ is strictly
positive, which implies the desired lower
bound~\eqref{eq:main:bounds} for $a_{g,2}$ when $g\ge 2$.
The explicit expression for $a_{g,2}$ when $g\ge 2$ implies
the desired strict upper bound~\eqref{eq:main:bounds}.

Recursive relations~\eqref{eq:a:g:k:difference} imply that
for $g\ge 3$ the difference $a_{g,3}-a_{g,2}$ is strictly
positive. Since $a_{g,2}$ satisfies the desired lower
bound~\eqref{eq:main:bounds}, the term $a_{g,3}$ also does.
The quadratic polynomial $(24g^2-49g+30)$ in the numerator
of the explicit expression for $a_{g,3}$ admits only
strictly positive values, which implies the upper
bound~\eqref{eq:main:bounds} for $a_{g,3}$ when $g\ge 3$.

Recursive relations~\eqref{eq:a:g:k:difference} imply that
for $g\ge 3$ the difference $a_{g,4}-a_{g,3}$ is strictly
negative. Since $a_{g,3}$ satisfies the desired upper
bound~\eqref{eq:main:bounds}, the term $a_{g,4}$ also does.
The explicit expression for $a_{g,4}$ implies the lower
bound~\eqref{eq:main:bounds} for $a_{g,4}$ when $g\ge 3$.

Finally, recursive relations~\eqref{eq:a:g:k:difference}
imply that for $g\ge 4$ the difference $a_{g,5}-a_{g,4}$ is
strictly positive, which implies the desired lower
bound~\eqref{eq:main:bounds} for $a_{g,5}$ when $g\ge 4$.
It remains to verify that the polynomial $(288 g^3-780
g^2+1012 g-525)$ in the explicit
Expression~\eqref{eq:epsilon:upper:init} for $a_{g,5}$
attains only strictly positive values for $g\ge 4$ to prove
the desired upper bound~\eqref{eq:main:bounds} for
$a_{g,5}$. For $g\ge 4$ we have:
\begin{multline*}
288 g^3-780 g^2+1012 g-525
>\\>
288 g^3 - 864 g^2 + 864 g  -576
=288((g-1)^3-1)>0\,.
\end{multline*}
\end{proof}

\begin{Corollary}
\label{cor:g:up:to:4}
For any $g$ in $\{1,2,3,4\}$ and for all $k\in\N$
satisfying $2\le k\le 3g-3$, the desired
bounds~\eqref{eq:main:bounds} are valid:
\begin{equation*}
1-\frac{2}{6g-1}=a_{g,1} < a_{g,k} < a_{g,0}=1\,.
\end{equation*}
\end{Corollary}
\begin{proof}
Recall that the symmetry $a_{g,k}=a_{g,3g-1-k}$ allows to
limit $k$ to the range $2\le k\le
\left[\frac{3g-1}{2}\right]$. Thus, for $g\le 4$ the
largest possible value of $k$ satisfying $k\le
\left[\frac{3g-1}{2}\right]$ equals to $\left[\frac{3\cdot
4-1}{2}\right]=5$. The proof of
bounds~\eqref{eq:main:bounds} for $k\le 5$ and any $g$ is
already completed in Lemma~\ref{lm:k:up:to:5}.
\end{proof}

\subsection{Alternative form of recurrence relations}

We start by extracting the common factor in
Relations~\eqref{eq:a:g:k:difference} and by simplifying
it.

Define the following polynomial in $g$ and $j$:
\begin{equation}
\label{eq:Q}
Q(g,j)=g(6g-6j-1)(6g-6j-3)\,,
\end{equation}
Rewriting double factorials in terms of factorials
and rearranging we get
\begin{multline*}
\cfrac{(6g-6j-5)!!}{(6g-1)!!}\cdot
\cfrac{(6j-1)!!\cdot(g-1)!}{j!\,(g-j)!}
=\\=
\left(\frac{(6g-6j-5)!}{(3g-3j-3)!\,2^{3g-3j-3}}\right)
\left(\frac{(3g-1)!\,2^{3g-1}}{(6g-1)!}\right)
\left(\frac{(6j-1)!}{(3j-1)!\,2^{3j-1}}\right)
\frac{(g-1)!}{j!\,(g-j)!}
=\\=
8\cdot
\left(\frac{(6j-1)!\,(6g-6j-5)!}{(6g-1)!}\right)
\left(\frac{(3g-1)!}{(3j-1)!\,(3g-3j-3)!}\right)
\left(\frac{(g-1)!}{j!\,(g-j)!}\right)
=\\=
8\cdot
\left(\frac{(6j)!\,(6g-6j)!}{(6g)!}\right)
\left(\frac{(3g)!}{(3j)!\,(3g-3j)!}\right)
\left(\frac{g!}{j!\,(g-j)!}\right)
\cdot\\ \cdot
\frac{6g}{(6j)\cdot(6g-6j)(6g-6j-1)(6g-6j-2)(6g-6j-3)(6g-6j-4)}
\cdot\\ \cdot
\frac{(3j)\cdot(3g-3j)(3g-3j-1)(3g-3j-2)}{3g}
\cdot\frac{1}{g}
=\\=
\frac{\binom{3g}{3j} \binom{g}{j}}{\binom{6g}{6j}}
\cdot
\frac{1}{g\cdot(6g-6j-1)(6g-6j-3)}
=R(g,j)\cdot\frac{1}{Q(g,j)}\,.
\end{multline*}
Defining the following polynomials in $g$ and $j$:
\begin{align}
\label{eq:p1}
P_1(g,j)&=(6g-6j-1)(6g-6j-3)(g-2j)\,,\\
\label{eq:p2}
P_2(g,j)&=-2(6g-6j-3)(6j+1)(g-j)\,,\\
\label{eq:p3}
P_3(g,j)&=2(6j+1)(6j+3)(g-j)\,,
\end{align}
we can express the recurrence
relations~\eqref{eq:a:g:k:difference} as
\begin{align}
\label{eq:R:P1:Q}
a_{g,3j}-a_{g,3j-1} &= R(g,j)\cdot \frac{P_1(g,j)}{Q(g,j)}\,,
&\text{ where }3\le 3j\le\left[\frac{3g-1}{2}\right]\,,
\\
\label{eq:R:P2:Q}
a_{g,3j+1}-a_{g,3j} &= R(g,j)\cdot \frac{P_2(g,j)}{Q(g,j)}\,,
&\text{ where }1\le 3j+1\le\left[\frac{3g-1}{2}\right]\,,
\\
\label{eq:R:P3:Q}
a_{g,3j+2}-a_{g,3j+1} &= R(g,j)\cdot \frac{P_3(g,j)}{Q(g,j)}\,,
&\text{ where }2\le 3j+2\le\left[\frac{3g-1}{2}\right]\,.
\end{align}

\begin{Lemma}
\label{lm:P:Q}
For any $g\in\N$ the following bounds are valid
\begin{align}
\label{eq:R:P1:Q:bound}
0<\frac{P_1(g,j)}{Q(g,j)}<1\,,
&&\text{ where }3\le 3j\le\left[\frac{3g-1}{2}\right]\,,
\\
\label{eq:R:P2:Q:bound}
-1<\frac{P_2(g,j)}{Q(g,j)}<0\,,
&&\text{ where }1\le 3j+1\le\left[\frac{3g-1}{2}\right]\,,
\\
\label{eq:R:P3:Q:bound}
0<\frac{P_3(g,j)}{Q(g,j)}<1\,,
&&\text{ where }2\le 3j+2\le\left[\frac{3g-1}{2}\right]\,.
\end{align}
\end{Lemma}
\begin{proof}
The bounds $3\le 3j\le\left[\frac{3g-1}{2}\right]$
in~\eqref{eq:R:P1:Q:bound} imply that $0 < j <
\frac{g}{2}$. Dividing Expression~\eqref{eq:p1} for
$P_1(g,j)$ by Expression~\eqref{eq:Q} for $Q(g,j)$ we get
$$
\frac{P_1(g,j)}{Q(g,j)}=\frac{g-2j}{g}\,.
$$
Clearly,
$$
0<\frac{g-2j}{g}<1
$$
for any $g\in\N$ and for all $j$ satisfying $0 < j <
\frac{g}{2}$.

The bounds $1\le 3j+1\le\left[\frac{3g-1}{2}\right]$
in~\eqref{eq:R:P2:Q:bound} imply that $0\le j
<\frac{g}{2}$. Dividing Expression~\eqref{eq:p2} for
$P_2(g,j)$ by Expression~\eqref{eq:Q} for $Q(g,j)$ we get
\begin{multline*}
\frac{P_2(g,j)}{Q(g,j)}
=-2\cdot\frac{6j+1}{6g-6j-1}\cdot\frac{g-j}{g}
=-\frac{6g-6j}{6g-6j-1}\cdot\frac{2j+\frac{1}{3}}{g}
=\\=
-\left(1+\frac{1}{6g-6j-1}\right)
\cdot\left(\frac{2j+\frac{1}{3}}{g}\right)\,.
\end{multline*}
Since $0\le 2j \le g-1$, the latter expression is always
strictly negative for this range of $j$. Both factors in
the brackets in the latter expression are monotonically
increasing on this range of $j$, so the maximum of the
absolute value of the product is attained at $2j=g-1$. We
get
\begin{multline}
\label{eq:P2:over:Q}
0<\left|\frac{P_2(g,j)}{Q(g,j)}\right|
\le\left(1+\frac{1}{6g-(3g-3)-1}\right)
\cdot\left(\frac{g-\frac{2}{3}}{g}\right)
=\\=
\left(1+\frac{1}{3g+2}\right)
\cdot\left(1-\frac{1}{\frac{3}{2}g}\right)<1
\quad\text{ for }g\in\N\
\text{ and }0\le j < \frac{g}{2}\,.
\end{multline}

The bounds $2\le 3j+2\le\left[\frac{3g-1}{2}\right]$
in~\eqref{eq:R:P3:Q:bound} imply that $0\le j<\frac{g}{2}$
and that $6j\le 3g-5$. Dividing Expression~\eqref{eq:p3}
for $P_3(g,j)$ by Expression~\eqref{eq:Q} for $Q(g,j)$ we
get
$$
\frac{P_3(g,j)}{Q(g,j)}
=
2\cdot\frac{(6j+1)(6j+3)(g-j)}{g(6g-6j-1)(6g-6j-3)}
=\left|\frac{P_2(g,j)}{Q(g,j)}\right|
\cdot\frac{6j+3}{6g-6j-3}\,.
$$
The expression $\frac{6j+3}{6g-6j-3}$ is strictly positive
and is monotonically increasing on the range of $j$ under
consideration, so it attains its maximum on the largest
possible value of $j$. Since $6j\le 3g-5$ we get
$$
0<\frac{6j+3}{6g-6j-3}\le\frac{3g-2}{3g+2}<1
\quad\text{for }
2\le 3j+2\le\left[\frac{3g-1}{2}\right]
\text{ and }g\in\N\,.
$$
Combined with~\eqref{eq:P2:over:Q} this proves the desired
bounds~\eqref{eq:R:P3:Q:bound}
which completes the proof of Lemma~\ref{lm:P:Q}.
\end{proof}

\begin{Lemma}
\label{lm:R}
For any fixed value of $g\in\N$, the expression $R(g,j)$
considered as a function of $j$ is strictly monotonically
decreasing on the range
$\left\{0,1,\dots,\left[\frac{g-1}{2}\right]\right\}$ of the
argument $j$.
\end{Lemma}
\begin{proof}
It is immediate to verify that
\begin{multline}
\label{eq:ratio:of:R}
R(g,j+1)/R(g,j)=
\frac{(6j+5)(6j+3)(6j+1)}{(6g-6j-1)(6g-6j-3)(6g-6j-5)}
\cdot\frac{g-j}{j+1}
=\\=
\frac{j+5/6}{j+1}
\cdot\frac{6j+1}{6g-6j-5}
\cdot\frac{6j+3}{6g-6j-3}
\cdot\frac{6g-6j}{6g-6j-1}
=\\=
\left(1-\frac{1}{6j+6}\right)
\cdot\frac{6j+1}{6g-6j-5}
\cdot\frac{6j+3}{6g-6j-3}
\cdot\left(1+\frac{1}{6g-6j-1}\right)
\end{multline}
For any fixed $g\in\N$ each of the four terms in the last
line of the above expression is
strictly monotonically increasing as
a function of $j$ on the range
$\{0,1,...\left[\frac{g-1}{2}\right]\}$.
It is immediate to verify that
when $2j=g-1$, the product of four terms
in the last line of the above expression
is identically equal to $1$ for all $g\in\N$,
and the Lemma follows.
\end{proof}

\begin{Lemma}
\label{lm:difference:bounded:by:jmax:R2}
For any $g\in\N$ and for any
integer $k$ in the range
$6\le k\le \left[\frac{3g-1}{2}\right]$
the following bounds hold:
\begin{equation}
\label{eq:agk:minus:ag5}
-R(g,2)\cdot\frac{g-3}{2}
\le a_{g,k}-a_{g,5}
\le R(g,2)\cdot\frac{3g-11}{3}\,.
\end{equation}
\end{Lemma}
\begin{proof}
We can assume that $g\ge 5$; otherwise the range of $k$ is
empty and the statement is vacuous.

Consider the sequence
$\left\{a_{g,5},a_{g,6},\dots,a_{g,k}\right\}$, and denote
by $n_+(k)$ the number of entries $m$ in the set
$\{5,6,\dots,k-1\}$, for which the inequality
$a_{g,m+1}>a_{g,m}$ holds. Similarly, denote by $n_-(k)$
the number of entries in the same set for which inequality
$a_{g,m+1}<a_{g,m}$ holds.

Combining the recurrence relations in the
form~\eqref{eq:R:P1:Q}--\eqref{eq:R:P3:Q},
bounds~\eqref{eq:R:P1:Q:bound}--\eqref{eq:R:P3:Q:bound} and
Lemma~\ref{lm:R} we conclude that
$$
-R(g,2)\cdot n_-(k)
\le a_{g,k}-a_{g,5}
\le R(g,2)\cdot n_+(k)\,.
$$
It remains to translate the restriction $k\le
\left[\frac{3g-1}{2}\right]$ into upper bounds for $n_+(k)$
and $n_-(k)$ as functions of $g$.

Recall that it follows from recurrence
relations~\eqref{eq:a:g:k:difference} that for $m$
satisfying $0\le m\le \left[\frac{3g-1}{2}\right]-1$ we
have $a_{g,m+1}<a_{g,m}$ if and only if $m\equiv
0\,(\operatorname{mod} 3)$ and we have $a_{g,m+1}>a_{g,m}$
for the remaining $m$ in this range. This implies that
\begin{align*}
n_+&=2(j-2)
&n_-&=j-2\,,
&\text{when }k=3j-1\,,
\\
n_+&=2(j-2)+1
&n_-&=j-2\,,
&\text{when }k=3j\,,\hspace*{17pt}
\\
n_+&=2(j-2)+1
&n_-&=j-1\,,
&\text{when }k=3j+1\,.
\end{align*}
In all these cases we have
\begin{align*}
n_+(k) &\le \frac{2k-10}{3}\,,
\\
n_-(k) &\le \frac{k-4}{3}\,.
\end{align*}
By assumption $k\le \left[\frac{3g-1}{2}\right]$,
so the latter bounds imply that
\begin{align*}
n_+(k) &\le \frac{3g-11}{3}
\\
n_-(k) &\le \frac{g-3}{2}\,.
\end{align*}
and~\eqref{eq:agk:minus:ag5} follows.
\end{proof}

We assume that $g\ge 5$ and
$k\in\{6,\dots,\left[\frac{3g-1}{2}\right]\}$. Define
\begin{align}
\label{eq:epsilon:lower:def}
\elow&=\frac{27(68g^3-308g^2+519g-280)}{(6g-1)(6g-3)(6g-5)(6g-7)(6g-9)}\,,
\\
\label{eq:epsilon:upper:def}
\eup&=\frac{9 (g-2) (288 g^3-780 g^2+1012 g-525)}{(6g-1)(6g-3)(6g-5)(6g-7)(6g-9)}\,.
\end{align}
In these notation
Expressions~\eqref{eq:epsilon:lower:init}
and~\eqref{eq:epsilon:upper:init} for $a_{g,5}$ can be
written as
\begin{equation}
\label{eq:a:g:5:in:terms:of:epsilon}
a_{g,5}=1-\frac{2}{6g-1}+\elow =1-\eup\,.
\end{equation}

\begin{Lemma}
\label{lm:R2:jmax:smaller:than:epsilon}
For any integer $g$ satisfying $g\ge 5$
the following strict inequalities are valid:
\begin{align*}
R(g,2)&\cdot\frac{g-3}{2}\quad   <\elow
\\
R(g,2)&\cdot\frac{3g-11}{3} <\eup\,.
\end{align*}
\end{Lemma}
\begin{proof}
The proof is a straightforward calculation.

First note that all the quantities $R(g,2), (g-3), (3g-11),
\elow, \eup$ are strictly positive for $g\ge 5$, where
strict positivity of $\elow$ and of $\eup$ was proved in
Lemma~\ref{lm:k:up:to:5}. Thus, it is sufficient to prove
that
\begin{align}
\label{eq:low}
\frac{2\,\elow}{R(g,2)\cdot(g-3)} > 1
\quad\text{for }g\ge 5\,,
\\
\label{eq:up}
\frac{3\,\eup}{R(g,2)\cdot(3g-11)} > 1
\quad\text{for }g\ge 5\,.
\end{align}

Applying Definition~\eqref{eq:R} to evaluate $R(g,2)$ and
cancelling out common factors in the numerator and in the
denominator of the resulting expression we get
\begin{equation}
\label{eq:R:g:2}
R(g,2)=
\frac{10395}{2}
\cdot\frac{g(g-1)}{(6g-1)(6g-3)(6g-5)(6g-7)(6g-9)(6g-11)}\,.
\end{equation}

Plug Expression~\eqref{eq:epsilon:lower:def} for $\elow$
and the above Expression~\eqref{eq:R:g:2} for $R(g,2)$ into
the left-hand side of~\eqref{eq:low} and cancel out the
common factors in the numerator and in the denominator of
the resulting expression. Applying polynomial division with
remainder to the resulting numerator and denominator we get
\begin{multline*}
\frac{2\,\elow}{R(g,2)\cdot(g-3)}=
2\cdot 27\cdot\frac{2}{10395}
\cdot\frac{(68g^3-308g^2+519g-280)(6g-11)}{g(g-1)(g-3)}
=\\=\frac{4}{385}
\cdot\left(
408 g -964 + \frac{1422 g^2- 4497 g + 3080}{g(g-1)(g-3)}
\right)\,.
\end{multline*}
It is immediate to check that $(1422 g^2- 4497 g + 3080)$
is positive for $g\ge 5$. It remains to note that for $g\ge
5$ we have $(408g-964)\ge (408\cdot 5-964)=1076> 385/4$ which completes the
proof of~\eqref{eq:low}.

Performing analogous manipulations we get
\begin{multline*}
\frac{3\,\eup}{R(g,2)\cdot(3g-11)}
=\\=
3\cdot 9\cdot\frac{2}{10395}
\cdot\frac{(g-2) (288 g^3-780 g^2+1012 g-525)(6g-11)}{g(g-1)(3g-11)}
=\\=
\frac{2}{385}
\cdot\left(
576 g^2 - 1080 g + 2964
+\frac{9790 g^2 + 1735 g - 11550}{g(g-1)(3g-11)}
\right)\,.
\end{multline*}
It is immediate to check that $(9790 g^2 + 1735 g - 11550)$
is positive for $g\ge 5$ as well as the denominator of the
corresponding fraction. It remains to note that the
function $(576 g^2 - 1080 g + 2964)$ is monotonically increasing on the interval $[5;+\infty[$, so for $g\ge 5$ we
get:
$$
576 g^2 - 1080 g + 2964\ge
576\cdot 5^2 - 1080\cdot 5 + 2964= 11964
> 385/2\,,
$$
which completes the proof of~\eqref{eq:up}.
\end{proof}

\begin{proof}[Proof of Proposition~\ref{pr:main:bounds}]
For small genera, $g=1,2,3,4$,
Proposition~\ref{pr:main:bounds} was proved in
Corollary~\ref{cor:g:up:to:4}.

For genera $g\ge 5$ and $k=2,3,4,5$,
Proposition~\ref{pr:main:bounds} was proved in
Lemma~\ref{lm:k:up:to:5}. The symmetry
$a_{g,k}=a_{g,3g-1-k}$ implies
Proposition~\ref{pr:main:bounds} for symmetric values of
$k$.

For $g\ge 5$ and $k$ in the range $6\le k\le
\left[\frac{3g-1}{2}\right]$
Proposition~\ref{pr:main:bounds} immediately follows from
combination of
Lemma~\ref{lm:difference:bounded:by:jmax:R2},
Expression~\eqref{eq:a:g:5:in:terms:of:epsilon} for
$a_{g,5}$ and Lemma~\ref{lm:R2:jmax:smaller:than:epsilon}.
The symmetry $a_{g,k}=a_{g,3g-1-k}$ implies
Proposition~\ref{pr:main:bounds} for symmetric values of
$k$.
\end{proof}

\subsection{Asymptotic behavior of normalized $2$-correlators
in large genera}

In this section we briefly describe the behavior of $a_{g,k}$
for $g\gg 1$. More detailed discussion would be given in a
separate (and more general) paper.

When $g\to +\infty$ and $j$ remains bounded,
Expressions~\eqref{eq:p1}--\eqref{eq:p3} and~\eqref{eq:Q}
for polynomials $P_i(g,j)$, $i=1,\dots,4$, and $Q(g,j)$
respectively imply that
\begin{align*}
\frac{P_1(g,j)}{Q(g,j)}
&=1-(2j)\cdot\frac{1}{g}\,,
\\
\frac{P_2(g,j)}{Q(g,j)}
&=-\left(2j+\frac{1}{3}\right)\cdot\frac{1}{g}+o\left(\frac{1}{g}\right)\,,
\\
\frac{P_3(g,j)}{Q(g,j)}
&=\left(2j+\frac{1}{3}\right)\left(j+\frac{1}{2}\right)\cdot\frac{1}{g^2}+o\left(\frac{1}{g^2}\right)\,,
\end{align*}
as $g\to+\infty$. In
particular, for $g\gg1$ we see that for small values of $j$
the ratio $\frac{P_1(g,j)}{Q(g,j)}$ is close to $1$, while
the ratio $\frac{P_2(g,j)}{Q(g,j)}$ is of the order
$\frac{1}{g}$ and the ratio $\frac{P_3(g,j)}{Q(g,j)}$ is of
the order $\frac{1}{g^2}$. Thus, assuming that $g\gg 1$,
and taking consecutive terms $\left\{a_{g,3j-1},\,a_{g,3j},\,a_{g,3j+1},\,a_{g,3j+1}\right\}$
with $j\ll g$
we observe certain increment from $a_{g,3j-1}$ to $a_{g,3j}$,
much smaller decrement from $a_{g,3j}$ to $a_{g,3j+1}$
and very small increment from $a_{g,3j+1}$ to $a_{g,3j+2}$.

For any fixed $g\gg 1$ and $j\ll g$ the quantity $R(g,j)$
defined by~\eqref{eq:R} is very rapidly decreasing
as $j$ grows. We conclude from Expression~\eqref{eq:ratio:of:R}
that for bounded $j$ and $g\to+\infty$ one has
$$
R(g,j+1)=R(g,j)\cdot\frac{(j+\frac{5}{6})(j+\frac{3}{6})(j+\frac{1}{6})}{(j+1)}
\cdot\frac{1}{g^2}\cdot(1+o(1))\,.
$$
Since $R(g,0)=1$ we get the following expressions for $j=0,1,2,3$:
\begin{align*}
R(g,0)&=1\,,
\\
R(g,1)&
=\frac{(0+\frac{5}{6})(0+\frac{3}{6})(0+\frac{1}{6})}{(0+1)}
\cdot\frac{1}{g^2}\cdot\big(1+o(1)\big)
&
=\frac{5}{72}\cdot\frac{1}{g^2}\cdot\big(1+o(1)\big)\,,
\\
R(g,2)&
=\frac{(1+\frac{5}{6})(1+\frac{3}{6})(1+\frac{1}{6})}{(1+1)}
\cdot\frac{5}{72}\cdot\frac{1}{g^4}\cdot\big(1+o(1)\big)
&
=\frac{385}{3456}\cdot\frac{1}{g^4}\cdot\big(1+o(1)\big)\,,
\\
R(g,3)&
=\frac{(2+\frac{5}{6})(2+\frac{3}{6})(2+\frac{1}{6})}{(2+1)}
\cdot\frac{385}{3456}\cdot\frac{1}{g^6}\cdot\big(1+o(1)\big)
\hspace*{-2.8pt}
&
\hspace*{-5.1pt}
=\frac{425425}{746496}\cdot\frac{1}{g^6}\cdot\big(1+o(1)\big)\,.
\end{align*}
Lemma~\ref{lm:difference:bounded:by:jmax:R2}
admits the following immediate generalization:
\begin{Lemma}
\label{lm:difference:bounded:by:Rj}
For any $g\in\N$ and for any
integer $k$ in the range
$3j\le k\le \left[\frac{3g-1}{2}\right]$
the following bounds hold:
\begin{equation}
\label{eq:agk:minus:ag:j}
-R(g,j)\cdot\frac{g-j-1}{2}
\le a_{g,k}-a_{g,3j-1}
\le R(g,j)\cdot\frac{3g-2j-7}{3}\,.
\end{equation}
\end{Lemma}
Thus, having found the asymptotic expansion (when
$g\to+\infty$) in $\frac{1}{g}$ up to the term
$\frac{1}{g^{2j-2}}$ for some $a_{g,3j-1}$, we get the same
asymptotic expansion up to the term $\frac{1}{g^{2j-2}}$
for all $a_{g,k}$ with $k$ satisfying $3j-1\le k\le 3g-3j$.

Finally, no matter whether $g=2j$ or $g=2j+1$ one easily
derives from Stirling's formula that
$$
R(g,j)\approx\frac{1}{2^{2g-1}}\cdot\frac{1}{\sqrt{\pi g}}\,.
$$
Morally, when $g\gg1$ and the index $k$ is located
sufficiently far from the extremities of the range
$\{0,1,\dots,3g-1\}$, the values of $a_{g,k}$ become,
basically, indistinguishable.

\begin{figure}[hbt]
\includegraphics{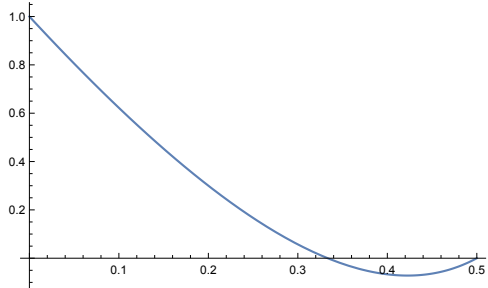}
\vspace{90bp}
\caption{
\label{fig:graph}
Graph of $f(x)$}
\end{figure}

It is curious to note that the sequence
$\{a_{g,2},\,a_{g,5},\dots,a_{g,3j_{max}-1}\}$, where
$j_{max}$ is the maximum integer satisfying
$3j_{max}-1\le\left[\frac{3g-1}{2}\right]$, is not
monotonically increasing for $g\ge 17$. Let $x=\frac{j}{g}$.
Our bounds on $j$ imply that $0<x<\frac{1}{2}$.
For any fixed $g\gg1$ define
$$
f(x)=f\left(\frac{j}{g}\right)=\frac{(P_1(g,j)+P_2(g,j)+P_3(g,j))}{Q(g,j)}\,.
$$
For large values of $g$ the graph of $f(x)$ has the form as
in Figure~\ref{fig:graph}, so up to some
point the function $f$ remains positive and the sequence
$a_{g,2},\,a_{g,5},\dots$ monotonically increases, but then
it attains its maximum and very slowly monotonically
decreases down to $a_{g,3j_{max}-1}$.

\section{Stable graphs: formal definition}
\label{s:stable:graphs}

Following M.~Kontsevich~\cite{Kontsevich}
we now introduce the definition of a stable graph. Let $S$
be a closed oriented surface of genus $g$ without boundary
endowed with $n$ labelled marked points. Let $\gamma$ be
a simple closed primitive multicurve on $S$. Consider the
decomposition of $S \setminus \gamma$ into a union of
surfaces with boundaries endowed with marked points. The
dual graph $\Graph$ to this decomposition is constructed as
follows.
\begin{itemize}
\item Each connected component $S_j$ of $S\setminus\gamma$
gives rise to a vertex $v_j$ of $\Graph$ decorated by the
genus $g(S_j)$. The marked points on $S_j$ are encoded by
the \textit{legs} attached to $v_j$, the boundary
components of $S_j$ correspond to half-edges incident to
$v_j$.
\item Each component $\gamma_i$ of $\gamma$ gives rise to
an edge of $\Gamma$. When $\gamma_i$ is the common boundary
of two distinct connected components of $S\setminus\gamma$,
the corresponding edge of $\Gamma$ joins the two distinct
vertices of $\Gamma$ representing these two connected
components. When both sides of $\gamma_i$ are at the boundary
of the same connected component of
$S\setminus\gamma$, the edge of $\Gamma$ dual to $\gamma_i$
is a loop joining the vertex associated to the
corresponding connected component to itself.
\end{itemize}
Figures in the tables of Appendix~\ref{a:2:0} illustrate the
correspondence between multicurves and stable graphs in genus 2.

We now present the following formal definition

\begin{Definition}
\label{def:stable:graph}
Consider a 6-tuple
$\Graph = (V, H, \iota, \alpha, \mathbf{g}, L)$, where
\begin{itemize}
\item $V$ is a finite set of \textit{vertices}.
\item $H$ is a finite set of \textit{half-edges}.
\item $\iota: H \to H$ is an involution. The fixed points
of $\iota$ are called the \textit{legs} and the 2-cycles of
$\iota$ are called the \textit{edges} of $\Graph$.
\item $\alpha: H \to V$ is a map that attaches a half-edge
to a vertex. The number of half-edges at a given vertex $v$
is denoted by $n_v := |\alpha^{-1}(v)|$.
\item The graph is connected: for each pair of vertices
$(u,v)$ there exists a sequence of
half-edges $(h_1, h'_1, h_2, h'_2, \ldots, h_k, h'_k)$
such that $\iota(h_i) = h'_i$, $u = \alpha(h_1)$,
$v = \alpha(h'_k)$ and $\alpha(h'_i) = \alpha(h_{i+1})$.
\item $\boldsymbol{g} = \{g_v\}_{v \in V}$ is a set
of non-negative integers, one at each vertex,
called the \textit{genus decoration}.
\item $L$ is a bijection from the set of legs to
$\{1,\ldots,n\}$.
\end{itemize}
Such a 6-tuple $\Graph$ is called a
\textit{stable graph} for $\overline{\cM}_{g,n}$ if the
genus decoration $\boldsymbol{g}$ satisfies
the following conditions:
\begin{itemize}
\item
$g(\Graph) = h_1(\Graph) + \sum_{v \in V} g_v$,
where $h_1(\Graph)$ is the first Betti number of the graph,
\item \textit{stability condition}
$2 g_v - 2 + n_v > 0$
at each vertex $v$ of $\Graph$.
\end{itemize}
\end{Definition}

To a stable graph $\Graph$ we associate an underlying graph
whose vertex set is $V$ and each 2-cycle $(h,h')$ of
$\iota$ gives an edge attached to $\alpha(h)$ and
$\alpha(h')$. We denote the set of these edges by $E =
E(\Graph)$. Such a graph can have multiple edges and loops.
The additional information carried by a stable graph is the
genus decoration $\mathbf{g}$ and the $n$ legs.

Two stable graphs $\Graph =
(V,H,\iota, \alpha,\boldsymbol{g}, L)$ and
$\Graph'=(V',H',\iota', \alpha',\boldsymbol{g'}, L')$ are
isomorphic if there exists two bijections $\phi: V \to V'$
and $\psi: H \to H'$ that preserve edges, legs and
genus decoration, that is
\[
\psi\circ\iota   = \iota' \circ \psi\,,
\qquad
L'(\psi(h)) = L(h)\,,
\qquad
g'_{\phi(v)} = g_v\,.
\]
Note that automorphisms of stable graphs are allowed to
interchange edges and vertices respecting the decoration
but not the legs which are numbered by $L$.

We denote by $\cG_{g,n}$ the set of isomorphism classes of
stable graphs with given genus $g$ and number of legs $n$.

As we already mentioned, each stable graph in $\cG_{g,n}$
corresponds to a cycle of the Deligne--Mumford
compactification $\overline{\cM}_{g,n}$. More precisely,
each vertex $v$ of the graph corresponds to the component
of a nodal curve of genus $g_v$ and contains the marked
points corresponding to the legs attached to this vertex.
Each edge of $E(\Graph)$ represents a node. (See the
survey~\cite{Vakil} for an excellent introduction to this
subject and for beautiful illustrations.) Hence, the unique
stable graph with no edge and $n$ legs in $\cG_{g,n}$
corresponds to the component $\overline{\cM}_{g,n}$ (the
smooth curves).

\section{Examples of explicit calculations}
\label{s:explicit:calculations}

\subsection{The cases of $\cQ_{0,3}$ and $\cQ_{1,1}$}
\label{ss:Q03:and:Q11}

We now consider the moduli spaces $\cM_{0,3}$ and $\cM_{1,1}$ and the associated
cotangent bundles $\cQ_{0,3}$ and $\cQ_{1,1}$.

There is a unique complex curve $C$ in $\cM_{0,3}$ which is $\C\Proj^1\setminus \{0,1,\infty\}$. This curve does not admit any non-zero quadratic differentials with
at most simple poles at the marked points $0$, $1$ and $\infty$ and with no other poles
(which is
coherent with the fact that the tangent space to a point is $0$). There is a
unique stable graph for $\cM_{0,3}$:
\[\cG(0,3)=\left\lbrace\begin{array}{c}
\includegraphics{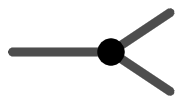}
\begin{picture}(0,0)(0,0)
\put(25,-15){$0$}
\end{picture}
\hspace{60pt}
\vspace{10pt}
\end{array}\right\rbrace.\]
Let us denote this graph by $\Phi_{0,3}$. If we try to apply
Theorem~\ref{th:volume} to define the value of $\Vol \cQ_{0,3}$, the polynomial
$P_{\Phi_{0,3}}$ defined by~\eqref{eq:P:Gamma} is ill-defined since
$4g-4+n = 6g-7+2n = -1$ and $(-1)!$ makes no sense. However, taking limits
as in~\eqref{eq:minus:one:factorial}, we obtain the value
\[
P_{\Phi_{0,3}} = 2\cdot\left.\frac{(-4+n)!}{(-7+2n)!}\right|_{n=3} = 4\,,
\]
leading to $\Vol \cQ_{0,3}=4$ by means of~\eqref{eq:square:tiled:volume}
which is coherent with the value in~\eqref{eq:vol:genus:0} evaluated for $n=3$
and also coherent with~\eqref{eq:convention:Q03}.
\smallskip

For $(g,n) = (1,1)$ there are two stable graphs as given below.
\[\cG(1,1)=\left\lbrace\begin{array}{c}
\includegraphics{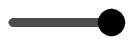}
\begin{picture}(0,0)(0,0)
\put(30,-15){$1$}
\end{picture}
\vspace{10pt}
\hspace{40pt}
\end{array},
\begin{array}{c}
\includegraphics{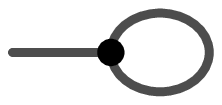}
\begin{picture}(0,0)(0,0)
\put(25,-15){$0$}
\end{picture}
\vspace{10pt}
\hspace{60pt}
\end{array}\right\rbrace.\]
Only the second graph contributes to
Expression~\eqref{eq:square:tiled:volume} from
Theorem~\ref{th:volume} for $\Vol \cQ_{1,1}$ and to
Expression~\eqref{eq:carea} from Theorem~\ref{th:carea} for
$\carea(\cQ_{1,1})$. Let us denote this graph by
$\Phi_{1,1}$. Applying~\eqref{eq:P:Gamma}, one finds
$P_{\Phi_{1,1}} = 4 b_1$ from which we deduce
\[
\Vol \cQ_{1,1} = \cZ\left(P_{\Phi_{1,1}}\right) = 4 \zeta(2) = \frac{2 \pi^2}{3}
\quad \text{and} \quad
\carea(\cQ_{1,1}) = \frac{3}{\pi^2} \cZ(\partial_\Phi P_\Phi)
= 2 \pi^2\,,
\]
which coincides with our convention~\ref{eq:convention:Q11}.

Note that any element of $\cQ_{1,1}$ is the square of a
holomorphic one form (regular at the marked point). In
other words, the principal stratum $\cQ(1,-1)$ is empty and
Definition~\ref{eq:Vol:sq:tiled} of $\Vol \cQ_{g,n}$ does
not apply. One can make geometric sense of the values $\Vol
\cQ_{1,1}$ and $\carea(\cQ_{1,1})$ by considering
square-tiled surfaces for the stratum $\cH(0)$ of
holomorphic Abelian differentials.

\subsection{Holomorphic quadratic differentials in genus two.}
\label{a:2:0}
We start by evaluation of the Siegel--Veech constant
$\carea(Q_2)$. Note that certain graphs do not
contribute to $\carea$ at all.
\medskip

$$
\hspace*{20pt}
\begin{array}{ll|l}

\includegraphics{genus_two_21.eps}
\begin{picture}(0,0)(0,0)
\put(-24,0){$b_1$}
\put(66,0){$b_2$}
\end{picture}
\hspace*{67pt}

&
\frac{128}{5}\cdot
1 \cdot
\frac{1}{8}\cdot
b_1 b_2 \cdot
N_{0,4}(b_1,b_1,b_2,b_2)=

&

(1+1)\cdot

\\


\includegraphics{genus_two_graph_21.eps}
\begin{picture}(0,0)(0,-10)
\put(-9,-19){$b_1$}
\put(52,-19){$b_2$}
\put(23,-30){$0$}
\end{picture}
  %

&
\rule{0pt}{12pt}
=\frac{16}{5}\cdot b_1 b_2\cdot\left(\frac{1}{4}(2b_1^2+2b_2^2)\right)

&

\cdot\!\frac{8}{5}\!\cdot\! 3!\zeta(4)\!\cdot\!1!\zeta(2)

\\

&

\hspace*{45pt}
\xmapsto{\partial_{\Graph}}\frac{8}{5}(1\cdot b_1 b_2^3+1\cdot b_1^3 b_2)\xmapsto{\cZ}

&

=1\cdot\frac{8}{225}\cdot \pi^6

\\
\vspace*{-18pt}\\
&&\\
\hline
&&\\


\includegraphics{genus_two_22.eps}
\begin{picture}(0,0)(0,0)
\put(-24,0){$b_1$}
\put(23,-14){$b_2$}
\end{picture}
\rule{0pt}{12pt}

&

\frac{128}{5}\!\cdot\!
\frac{1}{2}\!\cdot\!
\frac{1}{2}\!\cdot\!
b_1 b_2\!\cdot\! N_{0,3}(b_1,b_1,b_2)\!\cdot\! N_{1,1}(b_2)

&
1\!\cdot\!
\frac{2}{15}\!\cdot\! 1!\zeta(2)\!\cdot\!3!\zeta(4)

\\


\includegraphics{genus_two_graph_22.eps}
\begin{picture}(0,0)(0,0)
\put(-19,-10){$b_1$}
\put(23,-16){$b_2$}
\put(6,-10){$0$}
\put(41,-10){$1$}
\end{picture}

&
\rule{0pt}{15pt}
=\frac{32}{5}\!\cdot\! b_1 b_2 \!\cdot\! 1
\!\cdot\! \left(\frac{1}{48} b_2^2\right)
\xmapsto{\partial_{\Graph}}
\!\frac{2}{15}\!\cdot\! 1\!\cdot\!  b_1 b_2^3\xmapsto{\cZ}
&
=1\cdot\frac{1}{675}\cdot\pi^6

\\&&\\
\vspace*{-18pt}\\
&&\\
\hline
&&\\


\includegraphics{genus_two_31.eps}
\begin{picture}(0,0)(0,0)
\put(-24,0){$b_1$}
\put(23,-14){$b_2$}
\put(66,0){$b_3$}
\end{picture}

&
\frac{128}{5}\!\cdot\!
\frac{1}{2}\!\cdot\!
\frac{1}{8}\!\cdot\!
b_1 b_2 b_3 \!\cdot\!
N_{0,3}(b_1,b_1,b_2)\cdot

&

(1+\frac{1}{2}+1)\cdot

\\


\includegraphics{genus_two_graph_31.eps}
\begin{picture}(0,0)(0,5)
\put(-18,-10){$b_1$}
\put(23,-16){$b_2$}
\put(63,-10){$b_3$}
\put(5.5,-10){$0$}
\put(42,-10){$0$}
\end{picture}

\rule{0pt}{12pt}

&

\cdot N_{0,3}(b_2,b_3,b_3)
=\frac{8}{5}\!\cdot\!
{b_1} {b_2} {b_3}
\!\cdot\! (1) \!\cdot\! (1)

&
\cdot\frac{8}{5}\cdot \left(1!\,\zeta(2)\right)^3
\\


&
\rule{0pt}{15pt}
\xmapsto{\partial_{\Graph}}
\!\frac{8}{5}\!\left(1\!\cdot\! b_1 b_2 b_3
\!+\!\frac{1}{2} b_1 b_2 b_3
\!+\! 1\!\cdot\! b_1 b_2 b_3\right)\!\!\xmapsto{\cZ}\!
\hspace*{-3pt}
&
=\frac{5}{2}\cdot\frac{1}{135}\cdot\pi^6

\\
\vspace*{-18pt}\\
&&\\
\hline
&&\\


\includegraphics{genus_two_32.eps}
\begin{picture}(0,0)(0,0)
\put(-24,0){$b_1$}
\put(20,12){$b_2$}
\put(66,0){$b_3$}
\end{picture}

&

\frac{128}{5}\!\cdot\!
\frac{1}{2}\!\cdot\!
\frac{1}{12}\!\cdot\!
b_1 b_2 b_3 \!\cdot\!
N_{0,3}(b_1,b_1,b_2)\cdot

&

(1+1+1)\cdot

\\

\includegraphics{genus_two_graph_32.eps}
\begin{picture}(0,0)(0,0)
\put(4,-16){$b_1$}
\put(27.5,-16){$b_2$}
\put(43,-16){$b_3$}
\put(31,2){$0$}
\put(31,-35){$0$}
\end{picture}

&

\rule{0pt}{15pt}

\cdot N_{0,3}(b_2,b_3,b_3)
=\frac{16}{15}\!\cdot\! {b_1 b_2 b_3}\!\cdot\! (1) \!\cdot\! (1)

&
\frac{16}{15}\!\cdot\! \left(1!\,\zeta(2)\right)^3
\\


&
\rule{0pt}{15pt}

\xmapsto{\partial_{\Graph}}
\!\frac{8}{5}\!\left(1\!\cdot\! b_1 b_2 b_3
\!+\!\frac{1}{2} b_1 b_2 b_3
\!+\! 1\!\cdot\! b_1 b_2 b_3\right)\!\!\xmapsto{\cZ}\!
\hspace*{-3pt}
&
=3\cdot\frac{2}{405}\cdot\pi^6
\end{array}
$$
\bigskip

Taking the sum of the four contributions we obtain:
$$
\left(
\left(1\cdot\frac{8}{225}+1\cdot\frac{1}{675}\right)
+
\left(\frac{5}{2}\cdot\frac{1}{135}+3\cdot\frac{2}{405}\right)
\right)\cdot \pi^6
=
\left(\frac{1}{27}+\frac{1}{30}\right)\cdot \pi^6
=
\frac{19}{270}\cdot \pi^6\,.
$$
Dividing by $\Vol\cQ_2=\cfrac{1}{15}\cdot\pi^6$ we get the answer which
matches the value found in~\cite{Goujard:carea}:
$
\cfrac{\pi^2}{3}\cdot \carea(\cQ_2)=\left(\cfrac{19}{270}\pi^6\right):
\left(\cfrac{1}{15}\pi^6\right)=\cfrac{19}{18}\,.
$

The computation of the
Masur--Veech volume $\Vol\cQ_2$ was presented in
Table~\ref{tab:2:0}
in
Section~\ref{ss:intro:Masur:Veech:volumes}.
The first two graphs in Table~\ref{tab:2:0} represent the
contribution to the volume $\Vol\cQ_2$ of square-tiled
surfaces having single maximal cylinder. The resulting
contribution
$
\left(\tfrac{16}{945}+\tfrac{1}{2835}\right)\pi^6
=\frac{7}{405} \pi^6 = \frac{49}{3} \zeta(6)
$
was found in Appendix~C
in~\cite{DGZZ:initial:arXiv:one:cylinder:with:numerics} by
completely different technique.

The third and the fourth graph together represent the
volume contribution
$
\left(\tfrac{8}{225}+\tfrac{1}{675}\right)\pi^6
=\frac{1}{27}\pi^6
$
of square-tiled surfaces having two maximal cylinders. The last two
graphs --- the contribution
$
\left(\tfrac{1}{135}+\tfrac{2}{405}\right)\pi^6
=\frac{1}{81}\pi^6
$
of square-tiled surfaces having three maximal cylinders.

Normalizing the contribution of $1,2,3$-cylinder
square-tiled surfaces by the entire volume $\Vol\cQ_2$
of the stratum we get the quantity $p_k(\cQ_2)$ which
can be interpreted as the ``probability'' for a ``random''
square-tiled surface in the stratum $\cQ_2$ to have
exactly $k$ horizontal cylinders. These same quantities
$p_k$ provide ``probabilities'' of getting a $k$-band
generalized interval exchange transformation (linear
involution) taking a ``random'' generalized interval
exchange transformation in the Rauzy class representing the
stratum $\cQ_2$ (see section~3.2
in~\cite{DGZZ:initial:arXiv:one:cylinder:with:numerics} for
details). The latter quantities are particularly simple to
evaluate in numerical experiments. The resulting
proportions
$$
\big(p_1(\cQ_2),p_2(\cQ_2),p_3(\cQ_2)\big)=
\left(\frac{7}{405},\frac{1}{27},\frac{1}{81}\right):\frac{1}{15}=
\left(\frac{7}{27},\frac{15}{27},\frac{5}{27}\right)
$$
match the numerical experiments obtained earlier
in Appendix~C
of~\cite{DGZZ:initial:arXiv:one:cylinder:with:numerics}.

\subsection{Holomorphic quadratic differentials in genus three.}
\label{a:3:0}

In genus three there are already $41$  different decorated ribbon
graphs. We do not provide the graph-by-graph calculation as we
did in genus two but only the contributions of $k$-cylinder
square-tiled surfaces for all possible values $k=1,\dots,6$ of
cylinders.

$$
\begin{array}{|c|c|c|c|}
\hline &&&\\
[-\halfbls]
\text{Number of}       &\text{Number of}          &\text{Contribution}               &\text{Relative}          \\
\text{cylinders } k    &\text{graphs } \Graph    &\text{to the volume}              &\text{contribution } p_k \\
&&& \\
[-\halfbls]
\hline
&&&\\
[-\halfbls]
1                      &           2              &\frac{94667}{126299250}\cdot\pi^{12}  &\frac{757336}{3493125}   \\
&&& \\
[-\halfbls]
\hline
&&&\\
[-\halfbls]
2                      &           5              &\frac{150749}{108256500}\cdot\pi^{12} &\frac{4220972}{10479375} \\
&&& \\
[-\halfbls]
\hline
&&&\\
[-\halfbls]
3                      &           9              &\frac{84481}{86605200}\cdot\pi^{12}   &\frac{591367}{2095875}   \\
&&& \\
[-\halfbls]
\hline
&&&\\
[-\halfbls]
4                      &           12             &\frac{5989}{21651300}\cdot\pi^{12}    &\frac{167692}{2095875}   \\
&&& \\
[-\halfbls]
\hline
&&&\\
[-\halfbls]
5                      &           8              &\frac{1}{17820}\cdot\pi^{12}          &\frac{28}{1725}          \\
&&& \\
[-\halfbls]
\hline
&&&\\
[-\halfbls]
6                      &           5              &\frac{1}{144342}\cdot\pi^{12}         &\frac{56}{27945}         \\
&&& \\
\hline
\end{array}
$$

The resulting contribution of $1$-cylinder surfaces was confirmed by
the alternative combinatorial study of the Rauzy class of the stratum
$\cQ(1^8)$ (see section~3.2
in~\cite{DGZZ:initial:arXiv:one:cylinder:with:numerics}.
The approximate values
$p_k(\cQ(1^8))$ were confirmed by numerical experiments with
statistics of $k$-band generalized interval exchange transformations.

Taking the sum of all contributions we get the volume of the moduli
space $\cQ_3$ of holomorphic quadratic differentials in genus $3$:
$
\Vol\cQ_3=\cfrac{115}{33264}\cdot \pi^{12}\,.
$

\subsection{Meromorphic quadratic differentials in genus one.}
\label{a:1:2}

In this section we apply Formula~\eqref{eq:square:tiled:volume} to compute
the Masur--Veech volume $\Vol\cQ(1^2,-1^2)$ of the moduli space
$\cQ_{1,2}$ of meromorphic quadratic differentials in genus $g=1$
with two simple poles $p=2$.

We use the same convention on the order of numerical factors
in every first line of the middle column as in section~\ref{a:2:0}.
Namely, for $(g,p)=(1,2)$ we have
$
\zeroes!\cdot 2\dprinc
\cdot
\frac{2^{\dprinc}}{\dprinc!}
=\frac{32}{5}\,,
$
which is the first factor. The second factor is $1/2^{|V(\Graph)|-1}$.

The third factor is $|\Aut(\Graph)|^{-1}$. We remind
that the vertices and edges of $\Graph$ are not labeled while
the two legs are labeled. An automorphism of $\Graph$ preserves
the decoration of vertices and the labeling of the legs. For example,
the graph $\Graph$ in the second line does not have any nontrivial
automorphisms.

\begin{table}[hbt]
$$
\hspace*{20pt}
\begin{array}{llr}


\includegraphics{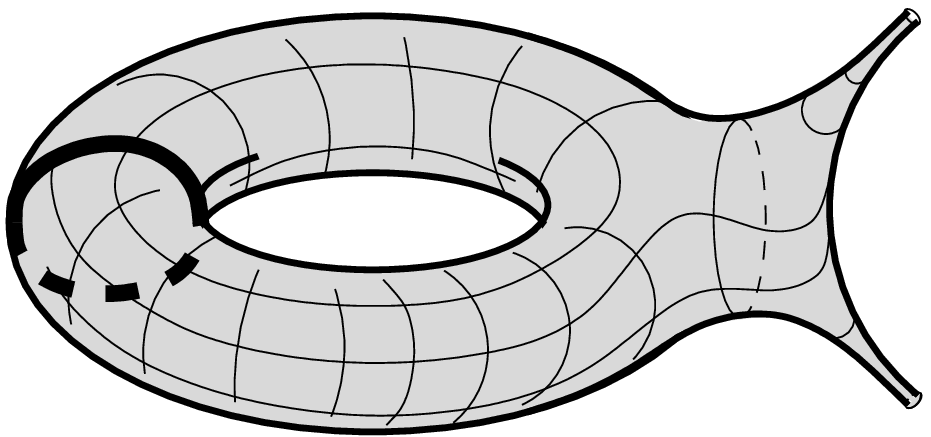}
\begin{picture}(0,0)(0,0)
\put(-9,0){$b_1$}
\end{picture}
\hspace*{45pt} 
\vspace*{10pt}

&

\frac{32}{3}\cdot
1\cdot
\frac{1}{2}\cdot
b_1 \cdot
N_{0,4}(b_1,b_1,0,0) =

&

\\ 
\includegraphics{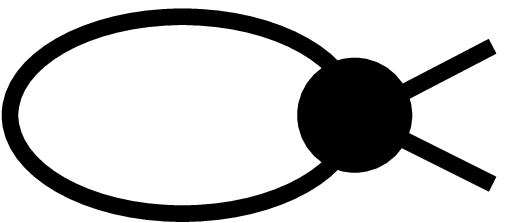}
\begin{picture}(0,0)(0,0)
\put(12,0){$b_1$}
\put(40,-9){$0$}
\end{picture}
\vspace*{10pt}

&

\hspace*{17pt}
=
\frac{16}{3}\cdot
b_1 \cdot \left(\frac{1}{4}(2b_1^2)\right)
\hspace*{5pt} =\hspace*{5pt}
\frac{8}{3}\cdot b_1^3\
\hspace*{1pt}
\xmapsto{\ \cZ\ }

&

\hspace*{-20pt}
\frac{8}{3}\cdot 3!\cdot\zeta(4)
=
\frac{8}{45}\pi^4

\\\hline&&\\


\includegraphics{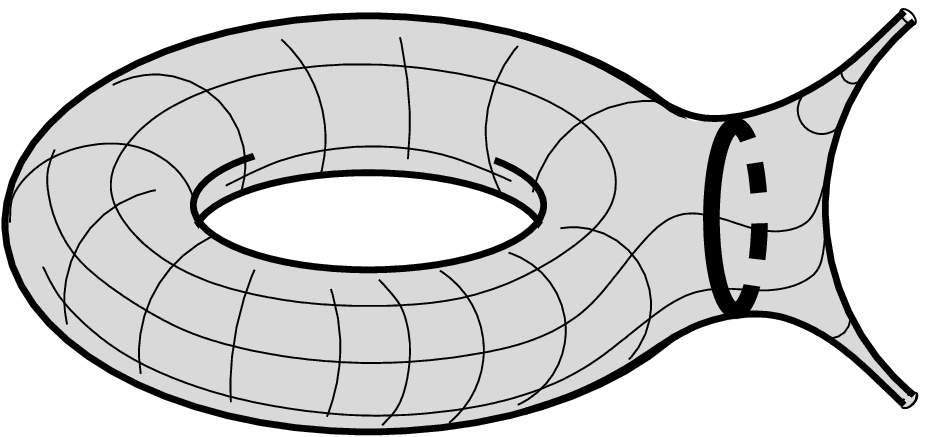}
\begin{picture}(0,0)(0,0)
\put(38,12){$b_1$}
\end{picture}
\vspace*{10pt}

&

\frac{32}{3}\cdot
\frac{1}{2}\cdot
1\cdot
b_1 \cdot
N_{1,1}(b_1)\cdot N_{0,3}(0,0,b_1) =

&

\\ 

\includegraphics{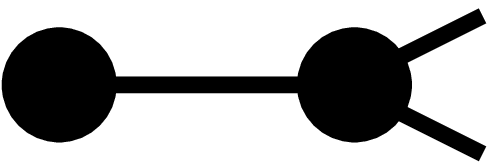}
\begin{picture}(0,0)(0,0)
\put(29,5){$b_1$}
\put(22,-9){$1$}
\put(39.5,-9){$0$}
\end{picture}
\vspace*{10pt}

&

\hspace*{17pt}
= \frac{16}{3}\cdot b_1 \cdot \left(\frac{1}{48} b_1^2\right) \cdot (1)
=\frac{1}{9}\cdot b_1^3\ \xmapsto{\ \cZ\ }

&

\hspace*{-20pt}
\frac{1}{9}\cdot 3!\cdot \zeta(4) = \frac{1}{135}\cdot\pi^4

\\\hline&&\\


\includegraphics{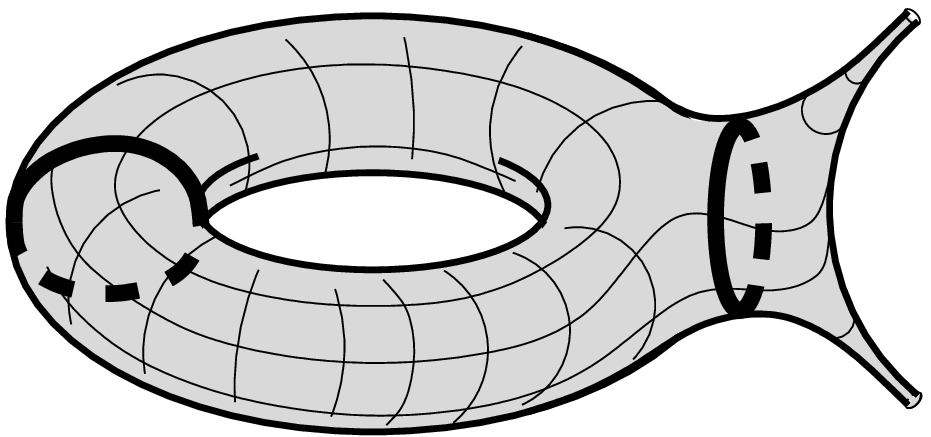}
\begin{picture}(0,0)(0,0)
\put(-9,0){$b_1$}
\put(38,12){$b_2$}
\end{picture}
\vspace*{10pt}

&

\frac{32}{3}\cdot
\frac{1}{2}\cdot
\frac{1}{2}\cdot
b_1 b_2\cdot
N_{0,3}(b_1,b_1,b_2)\cdot N_{0,3}(b_1,0,0)

&

\\ 

\includegraphics{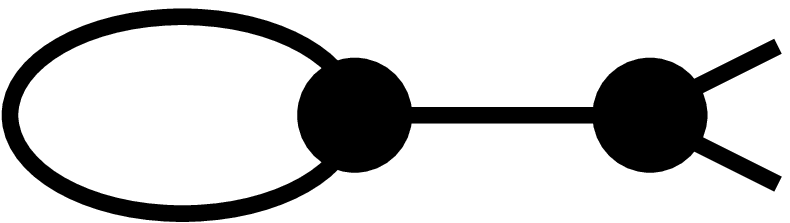}
\begin{picture}(0,0)(0,0)
\put(-5,0){$b_1$}
\put(29,5){$b_2$}
\put(22,-9){$0$}
\put(39.5,-9){$0$}
\end{picture}
\vspace*{10pt}

&

\hspace*{17pt}
= \frac{8}{3}\cdot
b_1 b_2\cdot (1)\cdot (1)=
\frac{8}{3}\cdot b_1 b_2
\
\hspace*{1pt}
\xmapsto{\ \cZ\ }

&

\hspace*{-20pt}
\frac{8}{3}\cdot \big(\zeta(2)\big)^2
=
\frac{2}{27}\pi^4

\\\hline&&\\


\includegraphics{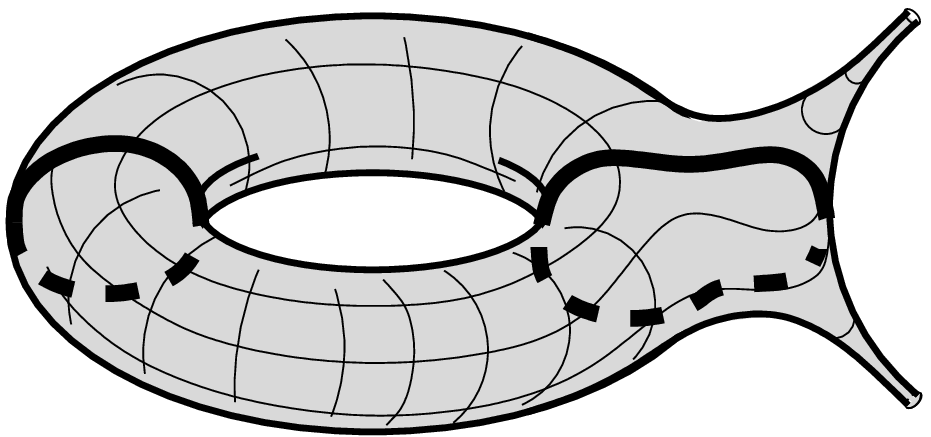}
\begin{picture}(0,0)(0,0)
\put(-9,0){$b_1$}
\put(38,12){$b_2$}
\end{picture}
\vspace*{15pt}

&

\frac{32}{3}\cdot
\frac{1}{2}\cdot
\frac{1}{2}\cdot
b_1 b_2\cdot
N_{0,3}(0,b_1,b_2)\cdot N_{0,3}(b_1,b_2,0)

&

\\ 

\includegraphics{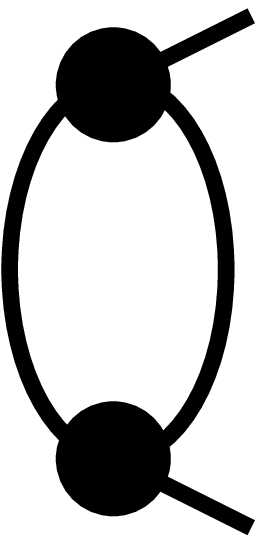}
\begin{picture}(0,0)(0,0)
\put(15,0){$b_1$}
\put(41,0){$b_2$}
\put(29.5,1.5){$0$}
\put(29.5,-20){$0$}
\end{picture}
\vspace*{20pt}

&

\hspace*{17pt}
= \frac{8}{3}\cdot
b_1 b_2\cdot (1)\cdot (1)=
\frac{8}{3}\cdot b_1 b_2

\
\hspace*{1pt}
\xmapsto{\ \cZ\ }

&

\hspace*{-20pt}
\frac{8}{3}\cdot \big(\zeta(2)\big)^2
=
\frac{2}{27}\pi^4
\end{array}
$$
\caption{
\label{tab:1:2}
Computation of $\Vol\cQ_{1,2}$. The left column represents
stable graphs $\Gamma\in\cG_{1,2}$ and associated multicurves;
the middle column gives associated polynomials $P_\Gamma$; the right column
provides volume contributions$\Vol(\Gamma)$.
}
\end{table}

The resulting value
\begin{equation*}
\Vol\cQ_{1,2}=
\left(\left(\frac{8}{45}+\frac{1}{135}\right)+\left(\frac{2}{27}+\frac{2}{27}\right)\right)\cdot\pi^4
=
\left(\frac{5}{27}+\frac{4}{27}\right)\cdot\pi^4
=\cfrac{\pi^4}{3}
\end{equation*}
matches the one found in~\cite{Goujard:volumes}. The
contribution of $1$-cylinder square-tiled surfaces and the
proportion $5:4$ between $1$-cylinder and $2$-cylinder
contributions match the corresponding quantities found in
Appendix~C
in~\cite{DGZZ:initial:arXiv:one:cylinder:with:numerics}.

\newpage

Now we evaluate the Siegel--Veech constant $\carea(Q(1^4))$.
\bigskip

$$
\hspace*{5pt}
\begin{array}{llr}


\includegraphics{genus_one_21.eps}
\begin{picture}(0,0)(0,0)
\put(-9,0){$b_1$}
\put(38,12){$b_2$}
\end{picture}
\vspace*{10pt}
\hspace*{43pt}
&

\frac{32}{3}\cdot
\frac{1}{2}\cdot
\frac{1}{2}\cdot
b_1 b_2\cdot
N_{0,3}(b_1,b_1,b_2)\cdot N_{0,3}(b_1,0,0)

&

\\ 

\includegraphics{genus_one_graph_21.eps}
\begin{picture}(0,0)(0,0)
\put(-5,0){$b_1$}
\put(29,5){$b_2$}
\put(22,-9){$0$}
\put(39.5,-9){$0$}
\end{picture}
\vspace*{10pt}

&

= \frac{8}{3}\cdot
b_1 b_2 \cdot 1\cdot 1=
\frac{8}{3}\cdot  b_1 b_2
\
\xmapsto{\cZ\circ\partial_{\Graph}}

&

\hspace*{-45pt}
\left(1 + \frac{1}{2}\right)\cdot
\frac{8}{3}\cdot \big(\zeta(2)\big)^2
=
\frac{3}{2}\cdot
\frac{2}{27}\pi^4

\\\hline&&\\


\includegraphics{genus_one_22.eps}
\begin{picture}(0,0)(0,0)
\put(-9,0){$b_1$}
\put(38,12){$b_2$}
\end{picture}
\vspace*{15pt}

&

\frac{32}{3}\cdot
\frac{1}{2}\cdot
\frac{1}{2}\cdot
b_1 b_2\cdot
N_{0,3}(0,b_1,b_2)\cdot N_{0,3}(b_1,b_2,0)

&

\\ 

\includegraphics{genus_one_graph_22.eps}
\begin{picture}(0,0)(0,0)
\put(15,0){$b_1$}
\put(41,0){$b_2$}
\put(29.5,1.5){$0$}
\put(29.5,-20){$0$}
\end{picture}
\vspace*{20pt}

&

= \frac{8}{3}\cdot
b_1 b_2\cdot 1\cdot 1=
\frac{8}{3}\cdot b_1 b_2

\
\hspace*{1pt}
\xmapsto{\cZ\circ\partial_{\Graph}}

&

\hspace*{-45pt}
(1+1)\cdot
\frac{8}{3}\cdot \big(\zeta(2)\big)^2
=
2\cdot
\frac{2}{27}\pi^4
\end{array}
$$
\bigskip

Taking the sum of the two contributions we obtain the answer:

$$
\left(\frac{3}{2}\cdot\frac{2}{27}+2\cdot\frac{2}{27}\right)\cdot\pi^4
=
\frac{7}{27}\cdot \pi^4
\,.
$$
Dividing by $\Vol\cQ(1^2,-1^2)=\frac{\pi^4}{3}$ we get the answer which
matches the value found in~\cite{Goujard:carea}.
$$
\frac{\pi^2}{3}\cdot \carea(\cQ(1^2,-1^2))=
\left(\frac{7}{27}\cdot\pi^4\right):
\left(\frac{1}{3}\cdot\pi^4\right)=\frac{7}{9}\,.
$$

\section{Tables of volumes and of Siegel--Veech constants}
\label{a:tables}

In this appendix we present numerical data for $\Vol(\cQ_{g,n})$ and
$\carea(\cQ_{g,n})$ corresponding to small values of $g$ and $n$.
Table~\ref{table:vol:SV} gathers the numerical values of the
volumes, of the Siegel--Veech constants and of the sums
$\Lambda^+$ (respectively $\Lambda^-$) of
the top $g$ (respectively $g_{\mathit{eff}}$)
Lyapunov exponents of the Kontsevich--Zorich cocycle over the principal
stratum $\QuadStrat(1^{4g-4+n}, -1^n)$.
By Formulae~(2.3) and~(2.4) in the paper of A.~Eskin,
M.~Kontsevich and A.~Zorich~\cite{Eskin:Kontsevich:Zorich},
these quantities are related as follows:
$$
\begin{aligned}
\Lambda^+&=\frac{(5g-5-n)}{18} +\frac{\pi^2}{3}\carea(\cQ_{g,n})\,.
\\
\Lambda^-   &= \Lambda^+ + \frac{(g-1+n)}{3}\,.
\end{aligned}
$$
It is possible to approximate $\Lambda^+$
and $\Lambda^-$ numerically
by computer simulations of the accelerated Rauzy induction.  The values of $\Lambda^+$ and of $\Lambda^-$ based on the values of $\carea(\cQ_{g,n})$ computed in this paper match the
approximate values obtained in these numerical experiments.

One has $\Lambda^+=0$ in genus zero, so in genus zero the
Siegel--Veech constant admits a simple closed formula. By
Formula~(1.1) in the paper of J.~Athreya, A.~Eskin and
A.~Zorich~\cite{AEZ:genus:0} one has
$$
\Vol\cQ(1^{n-4},-1^n)=2\pi^2\left(\frac{\pi^2}{2}\right)^{n-4}\,.
$$
We get the same expressions in the cells of Table~\ref{table:vol:SV}
corresponding to genus $0$ computed by Formulae~\eqref{eq:square:tiled:volume}
and~\eqref{eq:carea} of the current paper.

\begin{table}[ht]
 \renewcommand{\arraystretch}{1.2}
$\begin{array}{|c|c|c|c|c|c|c|}
\hline
g& n & \textrm{Stratum} &\textrm{Volume} & {\pi^2}/{3}\cdot \carea & \Lambda^+ & \Lambda^- \\
\hline
0 & 5 & \cQ(1, -1^5) & \pi^4 &{5}/{9} &  0 & {4}/{3} \\
\hline
0 & 6 & \cQ(1^2, -1^6) & {1}/{2}\cdot\pi^6 & {11}/{18} & 0 & {5}/{3} \\
\hline
0 & 7 & \cQ(1^3, -1^7) & {1}/{4}\cdot\pi^8 & 2/3 & 0 & 2\\
\hline
1 & 2 & \cQ(1^2, -1^2) & {1}/{3}\cdot\pi^4 & {7}/{9} & {2}/{3} & {4}/{3}\\
\hline
1 & 3 & \cQ(1^3, -1^3) & {11}/{60}\cdot\pi^6 & {47}/{66} & {6}/{11} & {17}/{11}\\
\hline
1 & 4 & \cQ(1^4, -1^4) & {1}/{10}\cdot\pi^8 & {44}/{63} & {10}/{21} & {38}/{21}\\
\hline
1 & 5 & \cQ(1^5, -1^5) & {163}/{3024}\cdot\pi^{10} & {2075}/{2934} & {70}/{163} & {1025}/{489}\\
\hline
2 & 0 & \cQ(1^4) & {1}/{15}\cdot\pi^6 & {19}/{18} & {4}/{3} & {5}/{3} \\
\hline
2 & 1 & \cQ(1^5, -1) & {29}/{840}\cdot\pi^8 & {230}/{261} & {32}/{29} & {154}/{87}\\
\hline
2 & 2 & \cQ(1^6, -1^2) & {337}/{18144}\cdot\pi^{10} & {8131}/{10110} & {1636}/{1685} & {3321}/{1685}\\
\hline
3 & 0 & \cQ(1^8) & {115}/{33264}\cdot\pi^{12} & {24199}/{25875}  & {4286}/{2875}  & {18608}/{8625} \\
\hline
4 & 0 & \cQ(1^{12}) & \pi^{18}\cdot{2106241}/ & {283794163}/ &{91179048}/ &{143835073}/\\
& & & {11548293120} & {315936150} & {52656025} & {52656025} \\
\hline
\end{array}$
\vspace{0.5cm}
\caption{
\label{tab:Vol:Q:g:n}
Numerical values of volumes, of Siegel--Veech constants and of
sums of Lyapunov exponents for low-dimensional strata\label{table:vol:SV}}
\end{table}

The recent paper of D.~Chen, M.~M\"oller and
A.~Sauvaget~\cite{Chen:Moeller:Sauvaget} suggested
alternative formulae for the Masur--Veech volumes and for
the area Siegel--Veech constants of the principal strata
$\QuadStrat(1^{4g-4+n},-1^n)$ as weighted sums of certain
very special linear Hodge integrals. The subsequent paper
of M.~Kazarian~\cite{Kazarian} provided very efficient
recursive formula for these Hodge integrals, which allows
one to compute $\Vol\cQ_{g,n}$ and $\carea(\cQ_{g,n})$ for
all sufficiently small values of $g$ and $n$ fast enough.
In particular, following this alternative approach one
obtains the same data as in Table~\ref{table:vol:SV}

Further numerical data can be found in~\cite{DGZZ:tables}
where we apply Formulae~\eqref{eq:square:tiled:volume}
and~\eqref{eq:carea} to express respectively the
Masur--Veech volumes $\Vol(\cQ_{g,n})$ and the Siegel--Veech
constants $\carea(\cQ_{g,n})$ as polynomials in the
intersection numbers of $\psi$-classes.
Recall that applying Formula~\eqref{eq:carea:Elise}
one can express the
Siegel--Veech constant $\carea$ in terms of the volumes of the
principal boundary strata.
One more table in~\cite{DGZZ:tables} provides the
corresponding explicit expressions for low-dimensional strata.


\end{document}